\documentclass[12pt]{amsart}
\usepackage{amssymb}
\usepackage{amsfonts}
\usepackage{latexsym}
\usepackage{amscd}
\usepackage[mathscr]{euscript}
\usepackage{xy} \xyoption{all}
\newcommand{\N}{{\mathbb{N}}}
\newcommand{\Z}{{\mathbb{Z}}}

\newcommand{\Zplus}{{\mathbb{Z}^+}}
\newcommand{\calA}{{\mathcal A}}
\newcommand{\calB}{{\mathcal B}}

\newcommand{\calD}{{\mathcal D}}
\newcommand{\calE}{{\mathcal E}}
\newcommand{\calG}{{\mathcal G}}
\newcommand{\calH}{{\mathcal H}}
\newcommand{\calL}{{\mathcal L}}
\newcommand{\calS}{{\mathcal S}}
\newcommand{\calV}{{\mathcal V}}
\newcommand{\calZ}{{\mathcal Z}}

\newcommand{\ol}{\overline}
\newcommand{\uloopr}[1]{\ar@'{@+{[0,0]+(-4,5)}@+{[0,0]+(0,10)}@+{[0,0] +(4,5)}}^{#1}}
\newcommand{\uloopd}[1]{\ar@'{@+{[0,0]+(5,4)}@+{[0,0]+(10,0)}@+{[0,0]+ (5,-4)}}^{#1}}
\newcommand{\dloopr}[1]{\ar@'{@+{[0,0]+(-4,-5)}@+{[0,0]+(0,-10)}@+{[0, 0]+(4,-5)}}_{#1}}
\newcommand{\dloopd}[1]{\ar@'{@+{[0,0]+(-5,4)}@+{[0,0]+(-10,0)}@+{[0,0 ]+(-5,-4)}}_{#1}}

\newcommand{\Cfin}{C_{\rm fin}}
\newcommand{\Cinf}{C_{\infty}}
\newcommand{\Cvfin}{C_{v,{\rm fin}}}
\newcommand{\Cvinf}{C_{v,\infty}}
\newcommand{\Mod}{\mbox{\rm-Mod}}

\newcommand{\luloop}[1]{\ar@'{@+{[0,0]+(-8,2)}@+{[0,0]+(-10,10)}@+{[0, 0]+(2,2)}}^{#1}}

\newcommand{\SGr}{\mathbf{SGr}}
\newcommand{\SSGr}{\mathbf{SSGr}}
\newcommand{\FSGr}{\mathbf{FSGr}}
\newcommand{\Sets}{\mathbf{Sets}}
\newcommand{\Mon}{\mathbf{Mon}}
\newcommand{\FP}{\mathbf{FP}}
\newcommand{\taubar}{\overline{\tau}}
\newcommand{\ebar}{\overline{e}}
\newcommand{\Hbar}{\overline{H}}

\newcommand{\Ehat}{\hat{E}}
\newcommand{\Rhat}{\hat{R}}
\newcommand{\thetahat}{\hat{\theta}}
\newcommand{\etatil}{\tilde{\eta}}
\newcommand{\phitil}{\tilde{\phi}}

\newcommand{\ftil}{\tilde{f}}
\newcommand{\qtil}{\tilde{q}}
\newcommand{\vtil}{\tilde{v}}

\newcommand{\Ctil}{\tilde{C}}
\newcommand{\Etil}{\tilde{E}}
\newcommand{\Ftil}{\tilde{F}}
\newcommand{\Mtil}{\tilde{M}}
\newcommand{\Stil}{\tilde{S}}
\newcommand{\Xtil}{\tilde{X}}

\newcommand{\bfr}{\mathbf{r}}
\newcommand{\bfrtil}{\tilde{\mathbf{r}}}
\newcommand{\bfrho}{\boldsymbol{\rho}}
\newcommand{\bfsig}{\boldsymbol{\sigma}}

\newcommand{\mon}[1]{\calV(#1)} %Monoide de projectius
\newcommand{\HCS}{\calH_{C,S}}
\newcommand{\Ctilfin}{\tilde{C}_{\rm fin}}
\newcommand{\Si}{{\rm Sink}}
\DeclareMathOperator{\Idem}{Idem}
\DeclareMathOperator{\Tr}{Tr}
\newcommand{\ACS}{\calA_{C,S}}
\newcommand{\bgast}{\mbox{\Large$*$}}

\DeclareMathOperator{\id}{id}
\DeclareMathOperator{\wt}{wt}

\theoremstyle{plain}
\newtheorem{theorem}{Theorem}[section]
\newtheorem{lemma}[theorem]{Lemma}
\newtheorem{proposition}[theorem]{Proposition}
\newtheorem{corollary}[theorem]{Corollary}

\theoremstyle{definition}
\newtheorem{definition}[theorem]{Definition}
\newtheorem{example}[theorem]{Example}

\newtheorem{remark}[theorem]{Remark}
\newtheorem*{remark*}{Remark}

\newtheorem*{assumption*}{Assumption}

\newtheorem{construction}[theorem]{Construction}
\newtheorem{subsec}[theorem]{}

\numberwithin{equation}{section}

\begin{document}

\title[Leavitt path algebras of separated graphs]{Leavitt path algebras of separated graphs}%

\author{P. Ara}

\address{Departament de Matem\`atiques, Universitat Aut\`onoma de Barcelona,
08193 Bellaterra (Barcelona), Spain.} \email{para@mat.uab.cat}

\author{K. R. Goodearl}

\address{Department of Mathematics, University of
California, Santa Barbara, CA 93106. }\email{goodearl@math.ucsb.edu}
%\dedicatory{}

\thanks{The first-named author was partially supported by DGI MICIIN-FEDER
MTM2008-06201-C02-01, and by the Comissionat per Universitats i
Recerca de la Generalitat de Catalunya. A substantial part of this
work was carried out during a research stay of the first-named
author at the Department of Mathematics of UCSB in 2009, supported
by the grant PR2008 0250 from the Ministerio de Ciencia e
Innovaci\'{o}n, Spain.}



%%\date{\today}

\begin{abstract} The construction of the Leavitt path algebra associated to a directed
graph $E$ is extended to incorporate a family $C$ consisting of partitions of the sets of
edges emanating from the vertices of $E$. The new algebras, $L_K(E,C)$, are analyzed in terms of
their homology, ideal theory, and K-theory. These algebras are proved to be hereditary, and it
is shown that any conical abelian monoid occurs as the monoid $\mon{L_K(E,C)}$ of isomorphism classes of finitely generated projective modules over one of these algebras. The lattice of trace ideals of $L_K(E,C)$ is determined by graph-theoretic data, namely as a lattice of certain pairs consisting of a subset of $E^0$ and a subset of $C$. Necessary conditions for $\mon{L_K(E,C)}$ to be a refinement monoid are developed, together with a construction that embeds $(E,C)$ in a separated graph $(E_+,C^+)$ such that $\mon{L_K(E_+,C^+)}$ has refinement.
\end{abstract}
\maketitle

%%%%%%%%%%%%%%%%%%%%%%%%%%%%%%%%%%%%%
\section{Introduction}
\label{sect:intro}

Leavitt introduced in \cite{lea} a class of algebras $L_K(m,n)$ (in current notation), for
$1\le m\le n$, over an arbitrary field $K$, which have a universal
isomorphism between the free modules of ranks $m$ and $n$. Some years
later and independently, Cuntz constructed and studied in
\cite{Cuntz77}, \cite{Cuntz81} the class of $C^*$-algebras $\mathcal
O _n$ nowadays known as {\it Cuntz algebras}, which are generated by
$n$ isometries $S_1,\dots ,S_n$ such that $\sum _{i=1}^n
S_iS_i^*=1$. It turns out that the Leavitt algebra $L_{\mathbb
C}(1,n)$ of type $(1,n)$ is isomorphic to a dense $*$-subalgebra of
the Cuntz algebra $\mathcal O _n$, and both algebras are simple and
purely infinite for $n\ge 2$.

\smallskip

Cuntz and Krieger \cite{CuntzKrieger} generalized the construction
of the Cuntz algebras $\mathcal O _n$ by considering a class of
$C^*$-algebras associated to finite square  matrices with entries in
$\{0,1\}$, and established an important connection between these
\emph{Cuntz-Krieger algebras} and the theory of dynamical systems.
Subsequently, it was realized that the Cuntz-Krieger algebras are
specific cases of a more general structure, the \emph{graph
$C^*$-algebras} $C^*(E)$ initially studied in depth in \cite{KPR}. We refer
the reader to \cite{Raeburn} for further information on this
important class of $C^*$-algebras. The $C^*$-analogs $U_{(m,n)}^{\text
nc}$ of the Leavitt
algebras $L_K(m,n)$ for $1<m\le n$ were studied by Brown
\cite{Brown} and McClanahan \cite{McCla1} in the case $n=m$ and by
McClanahan \cite{McCla2} in general. However, no description in terms
of graphs has been developed so far for this class of
$C^*$-algebras.

\smallskip

\emph{Leavitt path algebras} $L_K(E)$, natural algebraic versions of
graph $C^*$-algebras, were introduced  and studied firstly in
\cite{AA1} and \cite{AMP} for row-finite graphs, and in \cite{AA2}
and \cite{Tomf07} for general graphs; see also \cite{AB},
\cite{ABC}, \cite{AGPS}, \cite{APS}, \cite{Gdirlim}. These algebras
generalize the classical Leavitt algebras of type $(1,n)$ in much
the same way as graph $C^*$-algebras generalize the Cuntz algebras
$\mathcal O_n$. The Leavitt path algebra of a directed graph
$E=(E^0, E^1, r,s)$ is obtained from $E$ by first adding a copy of
$E^1$, written $E^{1*}=\{ e^*\mid e\in E^1 \}$, and extending $r$
and $s$ so that $r(e^*)=s(e)$ and $s(e^*)=r(e)$. The Leavitt path
algebra over a field $K$, $L_K(E)$, is then defined to be the
quotient of the ordinary path algebra of $(E^0, E^1\sqcup E^{1*},
r,s)$ obtained by imposing the Cuntz-Krieger relations. The new algebras we introduce are also quotients of the path algebra of $(E^0, E^1\sqcup E^{1*},
r,s)$, but with Cuntz-Krieger relations only imposed on selected sets of edges.

\medskip

The present paper initiates the study of much larger classes of algebras and $C^*$-algebras based on the concept of a \emph{separated graph}
$(E,C)$, namely a directed graph $E$ together with a family $C$ that
gives partitions of the set of edges departing from each vertex
of $E$. The associated Leavitt path algebra $L_K(E,C)$ and graph $C^*$-algebra $C^*(E,C)$ incorporate the existing versions $L_K(E)$ and $C^*(E)$ for a particular choice of $C$ on the one hand, and many algebras of classical Leavitt and Cuntz type on the other. As an immediate indicator of the broader generality of the new classes of algebras, we mention the following examples. First, any free
product of algebras $L_K(1,n)$ (or even $L_K(1,\aleph)$, where
$\aleph$ is an arbitrary cardinal) appears as $L_K(E,C)$ for
suitable $E$ and $C$ (Proposition \ref{1-vertex}). Second, for any
$n\ge m\ge 1$, there is a separated graph $(E,C)$ such that
$L_K(E,C) \cong M_{m+1}(L_K(m,n)) \cong M_{n+1}(L_K(m,n))$ and the
corner of $L_K(E,C)$ corresponding to one vertex of $E$ is
isomorphic to $L_K(m,n)$ itself (Proposition \ref{keyexample}). Similarly, free products of Cuntz algebras $\mathcal O_n$, and matrix $C^*$-algebras $M_{m+1}(U_{(m,n)}^{\text
nc})$, appear as $C^*(E,C)$ for suitable $E$ and $C$  \cite{AGsepCstar}.
\smallskip

The construction of the graph $C^*$-algebra $C^*(E,C)$ of a
separated graph $(E,C)$ is given in \cite{AGsepCstar}, where we also
prove that the Leavitt path algebra $L_{\mathbb C}(E,C)$ embeds as a
dense $*$-subalgebra of $C^*(E,C)$. As mentioned above, the
$C^*$-algebras $C^*(E,C)$ enable us to incorporate into the theory
the $C^*$-algebras $U_{(m,n)}^{\text nc}$ studied by Brown and
McClanahan. Moreover, by using results of Thomsen \cite{Thomsen} we
will compute the $K$-theory of these algebras. Also, in a
forthcoming paper we will attach a dynamical system to a suitable
quotient $\mathcal O _{m,n}$ of the $C^*$-algebra $U_{(m,n)}^{\text
nc}$, generalizing the classical construction of Cuntz and Krieger
for $\mathcal O _n$. Indeed, the $C^*$-algebra $\mathcal O_{m,n}$ is
the crossed product corresponding to the universal $(m,n)$-dynamical
system, where an \emph{$(m,n)$-dynamical system} consists of a
compact Hausdorff space $\Omega $, admitting two clopen
decompositions $
 \Omega=\bigsqcup _{i=1}^n U_i=\bigsqcup _{j=1}^m V_j$, together
with homeomorphisms
$$\Delta_i \colon U_i\to U_1\,,\qquad \Theta_j \colon V_j\to V_1\,,\qquad \Phi_{11} \colon V_1\to U_1\,,$$
for $i=2,\dots ,n$ and $j=2,\dots , m$.
\medskip

Our main motivation for the study of the new classes of algebras $L_K(E,C)$ and $C^*(E,C)$ is
of a $K$-theoretical nature. While ordinary Leavitt path
algebras constitute quite a large class, their monoids of finitely generated
projective modules  satisfy some restrictive conditions, as was
proved in \cite{AMP}. In particular, these monoids are always {\it
separative monoids} and satisfy the {\it Riesz refinement property}
(see below for definitions). For the resolution of some important
open problems in both ring theory and $C^*$-algebra theory, we need a
larger class of algebras, whose monoids satisfy the latter property
but not the former. The main open problem we wish to address, in the
context of $C^*$-algebras, is the construction of a $C^*$-algebra of
real rank zero containing both finite and properly infinite {\it
full} projections. It is worth to mention that R\o rdam has
constructed in \cite{Rordam} examples of simple $C^*$-algebras
having both finite and (properly) infinite projections. However, R\o rdam's
examples do not have real rank zero. The analogous question of
whether an {\it exchange ring} can have both finite and properly
infinite full idempotents seems to be wide open. Indeed, such an
example would provide a solution to the fundamental \emph{Separativity
Problem} for exchange rings of \cite{AGOP}. Note that by
\cite[Theorem 7.2]{AGOP}, the $C^*$-algebras of real rank zero are
exactly the $C^*$-algebras that happen to be exchange rings, and
that for an exchange ring $R$, the monoid $\mon{R}$ is always a
refinement monoid \cite[Corollary 1.3]{AGOP}.  The methods developed
in the present paper enable us to construct Leavitt path algebras
$L_K(E_+,C^+)$ of separated graphs $(E_+,C^+)$ having finite and
properly infinite full idempotents, and such that the monoids
$\mon{L_K(E_+,C^+)}$ satisfy the refinement property. For instance,
one can obtain such an algebra $L_K(E_+,C^+)$ by refining the Leavitt
path algebra $L_K(E,C)$ attached to the classical Leavitt algebra
$L_K(m,n)$, as shown in Example \ref{mainrevisited}. Although the
algebra $L_K(E_+,C^+)$ is not an exchange ring, we expect that the
use of techniques of universal localization, as in \cite{AB}, will lead to the
construction of a class of exchange algebras with the same structure
of finitely generated projective modules.

\begin{subsec}
{\bf Contents.} We now explain in more detail the contents of this paper. In Section
2, we will define our basic object of study, the Leavitt path
algebras $L_K(E,C)$ of separated graphs $(E,C)$, as well as the
corresponding class of \emph{Cohn path algebras} $C_K(E,C)$,
generalizing a class of algebras studied by Cohn in \cite{Cohn}, and
algebras we call \emph{Cohn-Leavitt algebras}. The Cohn path
algebras are obtained from the Leavitt path algebras by omitting one
of the so-called Cuntz-Krieger relations, and the Cohn-Leavitt
algebras interpolate between the two classes by admitting a
selection of Cuntz-Krieger relations. (See Definitions
\ref{def:LPASG}, \ref{def:CPASG}, and \ref{defCLalg} for the precise
definitions.) We will denote by $CL_K(E,C,S)$ the Cohn-Leavitt
algebra associated to $(E,C,S)$, where $S$ is a family consisting of
some of the finite sets in $C$. The Cohn-Leavitt algebras not only
provide an interesting larger class of algebras that can be analyzed
by the same techniques, but they allow us to write all Leavitt path
algebras (in fact, all Cohn-Leavitt algebras) as direct limits of
algebras based on finite graphs (see Proposition
\ref{CLalg-dirlim}). The algebras $L_K(E,C)$, $C_K(E,C)$, and
$CL_K(E,C,S)$ are homologically well behaved: all are hereditary (in
an appropriate nonunital sense when $E$ has infinitely many
vertices); moreover, the $K$-algebra unitizations of these
algebras are hereditary in the standard sense (Theorem
\ref{CLhereditary} and proof).

\smallskip

Much of our effort is devoted to analyzing the abelian monoids $\mon{A}$ associated to the algebras $A$ under
consideration. (As recalled below, $\mon{A}$ consists of the isomorphism classes of finitely generated projective $A$-modules,
equivalently, the Murray-von Neumann equivalence classes of idempotent matrices over $A$.)
In the present setting, these monoids are entirely determined by graph-theoretic data, just
as in the case of ordinary Leavitt path algebras. We refer to \cite{AMP} for the latter,
where a monoid $M(E)$ was defined for an arbitrary row-finite graph $E$ and shown to be naturally
isomorphic to $\mon{L_K(E)}$. Here, we define analogous monoids $M(E,C,S)$ and construct natural
isomorphisms $M(E,C,S) \cong \mon{CL_K(E,C,S)}$ (Theorem \ref{computVMECS}).
(The non-separated case reduces to that of ordinary Leavitt path algebras, and extends the result of \cite{AMP} to non-row-finite graphs.)

\smallskip

The monoids $M(E,C)\cong \mon{L_K(E,C)}$ (corresponding to Leavitt path algebras of separated graphs $(E,C)$) provide
a measure of the breadth of this new class of algebras: Every conical abelian monoid is
isomorphic to $M(E,C)$ for some $(E,C)$ (Proposition \ref{allconical}). Combined with
Theorems \ref{CLhereditary} and \ref{computVMECS}, this provides both an extension and a more ``visual''
version of a result of Bergman and Dicks \cite{BD}: Given any conical abelian monoid $M$, there exists a hereditary
$K$-algebra $A= L_K(E,C)$ such that $\mon{A}\cong M$ (Corollary \ref{cam=V}).

\smallskip

In the case of a non-separated graph $E$, the monoid $M(E)\cong \mon{L_K(E)}$ is a refinement monoid, as proved
for row-finite graphs in \cite{AMP} and extended here to arbitrary graphs (Corollary \ref{M(E)refinement}). We find
sufficient conditions for $M(E,C,S)$ to have refinement (Theorem \ref{generalrefinement}), and we develop a
construction by which a separated graph $(E,C)$ can be embedded in a separated graph $(E_+,C^+)$
such that $M(E_+,C^+)$ has refinement and preserves key properties of $M(E,C)$ (Theorem \ref{unitary3}).
In particular, if $M(E,C)$ does not have separative cancellation $(2x=x+y=2y \implies x=y$),
then $M(E_+,C^+)$ does not have it either. Further, if $M(E,C)$ is simple, we can arrange the construction so that
it embeds in the monoid of order-units of $M(E_+,C^+)$ together with zero and so that the
latter is a simple, divisible, refinement monoid (Theorem \ref{simple-emb}).
Given that $M(E,C)$ can be an arbitrary simple conical abelian monoid, this provides a ``visual'' version of an embedding theorem of Wehrung \cite{W}.

\smallskip

For ordinary Leavitt path algebras $L_K(E)$, the lattice of graded ideals (and even the lattice of all ideals,
if condition (K) holds) is isomorphic to a lattice formed from graph-theoretic data \cite{Tomf07}. Such a result
does not hold for separated graphs, since the algebras $L_K(E,C)$ and $CL_K(E,C,S)$ typically have far more complicated ideal
structures than $L_K(E)$. Nonetheless, we can capture the lattice of \emph{trace ideals} of $CL_K(E,C,S)$ (these coincide
with the idempotent-generated ideals). This lattice is isomorphic to the lattice of order-ideals
of $M(E,C,S)$ (Propositions \ref{LVCLisoTrCL} and \ref{LVAisoTrA}) and to a certain lattice
of pairs $(H,G)$ where $H$ is a hereditary subset of $E^0$ and $G\subseteq C$ (Theorem \ref{lattmonoid}).
Consequently, we derive necessary and sufficient conditions for $M(E,C,S)$ to be simple, equivalently, for $CL_K(E,C,S)$
to be ``trace-simple'' (Theorems \ref{simplicity} and \ref{Ccofinal}).
\end{subsec}

\begin{subsec} \label{introV(R)}
Recall that for a unital ring $R$, the monoid $\mon{R}$ is usually defined as the set of isomorphism classes $[P]$ of
finitely generated projective (left, say) $R$-modules $P$, with an addition operation given by $[P]+[Q]= [P\oplus Q]$.
For a nonunital version, see Definition \ref{FP}.

For arbitrary rings, $\mon{R}$ can also be described in terms of equivalence classes of idempotents from the ring $M_\infty(R)$
of $\omega\times\omega$ matrices over $R$ with finitely many nonzero entries. The equivalence relation is
\emph{Murray-von Neumann equivalence}: idempotents $e,f\in M_\infty(R)$ satisfy $e\sim f$ if and only if
there exist $x,y\in M_\infty(R)$ such that $xy=e$ and $yx=f$. Write $[e]$ for the equivalence class of $e$;
then $\mon{R}$ can be identified with the set of these classes. Addition in $\mon{R}$ is given by the rule $[e]+[f]= [e\oplus f]$,
where $e\oplus f$ denotes the block diagonal matrix $\left( \begin{smallmatrix} e&0\\ 0&f \end{smallmatrix} \right)$.
With this operation, $\mon{R}$ is an abelian monoid, and it is \emph{conical}, meaning that $a+b=0$ in $\mon{R}$ only when $a=b=0$.

An (abelian) monoid $M$ is said to be a {\it refinement monoid} if
whenever $a+b=c+d$ in $M$, there exist $x,y,z,t$ in $M$ such that
$a=x+y$ and $b=z+t$ while $c=x+z$ and $d=y+t$.
\end{subsec}

\begin{subsec} \label{Bergmanmachinery}
We will make frequent use of the machinery developed by Bergman in
\cite{Berg1} and \cite{Berg2}, which enables us to modify algebras
in such a way that we have total control on the $\calV$-monoid of
the resulting algebra. We repeatedly use one particular construction
from \cite{Berg2}, so we find it useful to explain here the
construction as well as the phrasing of it that we will use later
on. Let $R$ be a unital $K$-algebra and let $P$ and $Q$ be two
finitely generated projective left $R$-modules. Bergman constructs
in \cite[page 38]{Berg2} a unital $K$-algebra $S:=R\langle
i,i^{-1}\colon \ol{P}\cong \ol{Q}\rangle $, together with an algebra
homomorphism $R\to S$, such that there is a universal isomorphism
$i\colon \ol{P}\to \ol{Q}$, where $\ol{X}=S\otimes _R X$ for a left
$R$-module $X$. The universal property of $i$ is expressed as
follows. If $R\to T$ is an algebra homomorphism and $\phi \colon
T\otimes _R P\to T\otimes _R Q$ is an isomorphism of $T$-modules,
then there is a unique algebra homomorphism $\psi\colon S\to T$ such
that $\id_T\otimes_S\, i=\phi$, where $T$ is an $S$-module via $\psi$.

We shall refer to the algebra $S$ described above as \emph{the
Bergman algebra obtained from $R$ by adjoining a universal isomorphism
between $P$ and $Q$.} By \cite[Theorem 5.2]{Berg2}, the monoid
$\mon{S}$ is exactly the quotient monoid of $\mon{R}$ modulo the
congruence generated by $([P],[Q])$, so that we modify $\mon{R}$ by
just introducing a single new relation $[P]=[Q]$.
\end{subsec}

\begin{subsec}
Throughout the paper, $K$ will denote a field. All graphs in this
paper will be directed graphs $E=(E^0,E^1,s,r)$, where $E^0$ and
$E^1$ are sets, $E^0$ is nonempty, and $s$, $r$ denote the source and range maps
$E^1\rightarrow E^0$. We make no finiteness or
countability assumptions. Paths of positive length in $E$ are written in the form $\alpha= e_1e_2\cdots e_n$ where the $e_i\in E^1$ and $r(e_i)= s(e_{i+1})$ for $i<n$, while the paths of length zero in $E$ are identified with the vertices in $E^0$. The maps $s$ and $r$ are applied to paths in the obvious manner: $s(\alpha)=s(e_1)$ and $r(\alpha)= r(e_n)$ for $\alpha= e_1e_2\cdots e_n$, while $s(v)=r(v)=v$ for paths $v$ of length zero. We write ${\rm Path} (E)$ for the set of all paths in $E$. This is a partial semigroup, with $\alpha\beta$ defined and equal to the concatenation ``$\alpha$ followed by $\beta$" whenever $r(\alpha)= s(\beta)$. Concatenation with a path of length zero simply absorbs that vertex; e.g., $v\beta=\beta$ for any path $\beta$ with source vertex $v$.

 The path algebra of $E$ with coefficients in
$K$ will be denoted by $P_K(E)$; this is the $K$-algebra with basis ${\rm Path} (E)$ and multiplication induced from the partial multiplication in ${\rm Path} (E)$ together with the rule that $\alpha\beta=0$ for any paths $\alpha$, $\beta$ with $r(\alpha)\ne s(\beta)$.
\end{subsec}

%%%%%%%%%%%%%%%%%%%%%%%%%%%%%%%%%%%%%
\section{The algebras}
\label{sect:algebras}

In this section, we formally introduce separated graphs and the Leavitt, Cohn, and Cohn-Leavitt algebras based on them.

\begin{definition} \label{defsepgraph}
A \emph{separated graph} is a pair $(E,C)$ where $E$ is a graph,  $C=\bigsqcup
_{v\in E^ 0} C_v$, and
$C_v$ is a partition of $s^{-1}(v)$ (into pairwise disjoint nonempty
subsets) for every vertex $v$. (In case $v$ is a sink, we take $C_v$
to be the empty family of subsets of $s^{-1}(v)$.)

If all the sets in $C$ are finite, we shall say that $(E,C)$ is a \emph{finitely separated} graph.

The constructions we introduce revert to existing ones in case $C_v= \{s^{-1}(v)\}$
for each non-sink $v\in E^0$. We refer to a \emph{non-separated graph} or a \emph{trivially separated graph} in that situation.
\end{definition}

The separating partitions $C$ in the above definition can be viewed in terms of edge-labellings, if desired. On one hand,  if $\ell: E^1\to A$ is an edge-labelling, the sets $\ell^{-1}(a)\cap s^{-1}(v)$ for $a\in A$ partition $s^{-1}(v)$ for each $v\in E^0$, and the collection of nonempty such sets forms a separating partition $C$. On the other, given any separating partition $C$, the map $E^1\rightarrow C$ sending $e\in E^1$ to the unique set $X\in C$ such that $e\in X$ is an edge-labelling. The notation of partitions is more convenient for the work below than that of edge-labellings; moreover, nothing we do involves any relation between edges departing from different vertices.

We now introduce the first of three algebras based on a separated graph $(E,C)$. All three are quotients of the path algebra of the \emph{double} of $E$, that is, the graph $\Ehat$ obtained
from $E$ by adjoining, for each $e\in E^1$, an edge $e^*$ going in
the reverse direction of $e$, that is $s(e^*)=r(e)$ and
$r(e^*)=s(e)$. The map $e\mapsto e^*$ from $E^1\rightarrow \Ehat^1$ extends, first, to a bijection of $E^0\sqcup \Ehat^1$ with itself such that  $v^*=v$ for $v\in E^0$ and $(e^*)^*=e$ for $e\in E^1$; second, to an order-reversing bijection of ${\rm Path}(\Ehat)$ with itself; and, finally, to a $K$-algebra involution (i.e., an anti-automorphism of order $2$) on $P_K(\Ehat)$. This involution induces involutions on each of the three quotients of $P_K(\Ehat)$ defined below.

\begin{definition}
\label{def:LPASG} {\rm The {\it Leavitt path algebra of the
separated graph} $(E,C)$ with coefficients in the field $K$ is the
$K$-algebra $L_K(E,C)$ with generators $\{ v, e, e^*\mid v\in E^0,\
e\in E^1 \}$, subject to the following relations:}
\begin{enumerate}
\item[] (V)\ \ $vv^{\prime} = \delta_{v,v^{\prime}}v$ \ for all $v,v^{\prime} \in E^0$ ,
\item[] (E1)\ \ $s(e)e=er(e)=e$ \ for all $e\in E^1$ ,
\item[] (E2)\ \ $r(e)e^*=e^*s(e)=e^*$ \ for all $e\in E^1$ ,
\item[] (SCK1)\ \ $e^*e'=\delta _{e,e'}r(e)$ \ for all $e,e'\in X$, $X\in C$, and
\item[] (SCK2)\ \ $v=\sum _{ e\in X }ee^*$ \ for every finite set $X\in C_v$, $v\in E^0$.
\end{enumerate}

The path algebra $P_K(\Ehat)$ is the $K$-algebra with generating set $E^0\sqcup \Ehat^1$ and relations (V), (E1), (E2), so $L_K(E,C)$ is the quotient of $P_K(\Ehat)$ obtained by imposing the additional relations (SCK1), (SCK2).

The Leavitt path algebra $L_K(E)$ is just $L_K(E,C)$ where $C_v= \{
s^{-1}(v)\}$ if $s^{-1}(v)\ne \emptyset $ and $C_v=\emptyset $ if
$s^{-1}(v)=\emptyset$. Despite the great similarity in the
definitions, the Leavitt path algebras of separated graphs encompass
a much larger class of algebras than Leavitt path algebras of non-separated graphs. For
instance, they include free products of Leavitt algebras and algebras closely
related to the Leavitt algebras $L_K(n,m)$ (see Propositions \ref{1-vertex} and
\ref{keyexample} below).
\end{definition}

The Leavitt path algebras $L_K(E,C)$ can be viewed as algebraic analogs of the \emph{$C^*$-algebras of edge-labelled graphs} introduced by Duncan in \cite{Dun}.

In \cite{haz}, Hazrat gives a construction of \emph{weighted Leavitt path algebras} which incorporate all of the Leavitt algebras $L_K(n,m)$. Neither his construction nor ours is a particular case of the other.
Hazrat's basic data can be expressed as a separated graph $(E,C)$ such that each of the sets in $C$ is a finite collection of edges with the same source and the same range. However, the Cuntz-Krieger type relations he imposes involve sums of products $e^*e'$ and $e'e^*$ for edges $e$ and $e'$ that may lie in different sets in $C$. These relations agree with our (SCK1) and (SCK2) only when each vertex of $E$ emits at most one edge.
\medskip

Before we give some key examples of our construction, we develop a
normal form for the elements of $L_K(E,C)$. For this purpose, we will
use the Diamond Lemma \cite{Berg3}. We also define and study two other useful
constructs: the Cohn path algebra of a separated graph and the
Cohn-Leavitt path algebras of separated graphs with distinguished
subsets.

\begin{definition}
\label{def:CPASG} The {\it Cohn path algebra of the separated
graph} $(E,C)$ with coefficients in the field $K$ is the
$K$-algebra $C_K(E,C)$ with generators $\{ v, e, e^*\mid v\in E^0,\;
e\in E^1 \}$, subject only to the relations (V), (E1), (E2) and
(SCK1) of Definition \ref{def:LPASG}. In other words, $C_K(E,C)$ is the quotient of the path algebra $P_K(\Ehat)$ obtained by imposing only (SCK1).
\end{definition}

The $C^*$-analog of the Cohn path algebra of a non-separated graph
is the Toeplitz-Cuntz-Krieger $C^*$-algebra of the graph, see
\cite[Theorem 4.1]{FR}.

\medskip

We will
give $K$-bases of $C_K(E,C)$ and $L_K(E,C)$ by specifying particular
sets of paths in $\Ehat$, that is, by finding vector space sections
of the canonical surjections $P_K(\Ehat)\to C_K(E,C)$ and
$P_K(\Ehat)\to L_K(E,C)$. The choice is canonical in the case of the
Cohn algebra and depends on the choice of an edge in each finite
$X\in C$ in the case of the Leavitt algebra.

For two paths $\gamma ,\mu \in {\rm Path} (E)$ of positive length,
with $s(\gamma )=s(\mu )=v\in E^0$, we say that $\gamma $ and $\mu $
are {\it $C$-separated} if the initial edges of $\gamma $ and $\mu $
belong to different sets $X,Y\in C_v$.

\begin{proposition}
\label{basisCE} Let $(E,C)$ be a separated graph. Then the set
$\calB$ of those paths of the form
\begin{equation}
\label{normalform1} \lambda _1\nu_1 ^*\lambda _2\nu_2^*\cdots
\lambda _r\nu _r^* , \qquad \lambda_i, \nu_i \in {\rm Path}(E),\
r\ge 1,
\end{equation} such that $\nu
_i$ and $\lambda_{i+1}$ are $C$-separated paths for each $i=1,\dots
,r-1$, is a $K$-basis of $C_K(E,C)$. In particular,
$\nu_1,\dots,\nu_{r-1}$ and $\lambda_2,\dots,\lambda_r$ must have
positive length when $r>1$. However, $\lambda_1$ and $\nu_r$ are
allowed to have length zero.
\end{proposition}

\begin{proof}
Let $A$ be the vector space with basis $\calB$. We define a binary
product on $\calB\sqcup\{0\}$ by the following formula: If $b=\lambda_1\nu_1 ^*
\cdots \lambda _r\nu _r^*$ and $b'=\gamma _1\mu_1 ^*\cdots \gamma
_s\mu _s^*$, then
$$bb' :=\begin{cases} \lambda _1\nu_1 ^*\cdots \lambda
_r(\nu ' _r) ^*\gamma _1'\mu_1 ^*\cdots \gamma _s\mu _s^* &
\text{if\ } \nu _r=\tau \nu _r' \text{\ and\ } \gamma _1=\tau \gamma
_1' \text{\ with\ }
\nu_r' \text{\ and\ } \gamma _1' \\ &\quad \text{being $C$-separated} \\
\lambda _1 \nu_1 ^*\cdots (\lambda _r\gamma
_1')\mu_1 ^*\cdots \gamma _s\mu _s^* & \text{if\ } \gamma _1=\nu_r\gamma _1' \text{\ for some\ } \gamma _1' \in {\rm Path}(E)\\
\lambda _1 \nu_1 ^*\cdots \lambda _r(\mu _1\nu_r')^* \cdots \gamma
_s\mu _s^* & \text{if\ } \nu_r=\gamma_1\nu _r' \text{\ for some\ } \nu_r' \in {\rm Path}(E)\\
0 & \text{otherwise.}
\end{cases}$$

It is readily seen that with this product, $\calB\sqcup\{0\}$ becomes a semigroup. Extending this product
linearly gives a structure of associative algebra to $A$. It is
clear that, with this structure, $A$ is isomorphic to $C_K(E,C)$.
This shows the result.
\end{proof}

We next introduce the ``mixed case'' of Cohn-Leavitt algebras.

\begin{definition} \label{defCLalg}
Let $(E,C)$ be a separated  graph. Let us denote by $C_{\rm fin}$ the subset of $C$ consisting of those
$X$ such that $| X | < \infty$, and let $S$ be any subset of $\Cfin$. Then let $CL_K(E,C,S)$ be the $K$-algebra with generators $\{ v, e, e^*\mid v\in E^0,\;
e\in E^1 \}$, subject to the relations (V), (E1), (E2) and
(SCK1) of Definition \ref{def:LPASG} together with the relations (SCK2) for the sets $X\in S$. Observe that $CL_K(E,C,\emptyset)=C_K(E,C)$ and
$CL_K (E,C, \Cfin )=L_K(E,C)$. We call $CL_K(E,C,S)$ the {\it
Cohn-Leavitt algebra of the triple $(E,C,S)$}.
\end{definition}

The $C^*$-analog of Cohn-Leavitt path algebras, for a
non-separated graph $E$ and a subset $V$ of regular vertices of $E$,
is introduced in \cite[Definition 3.5]{MT} under the name of {\it
relative graph algebra}. It is shown in \cite[Theorem 3.7]{MT} that
the relative graph $C^*$-algebra $C^*(E,V)$ is canonically
isomorphic to the graph $C^*$-algebra $C^*(E_V)$ of a suitable graph
$E_V$.

\medskip

We are now in a position to describe a basis of $CL_K(E,C,S)$, and
so, in particular, of $L_K(E,C)$.

\begin{definition}
\label{def:redform1} Let $(E,C)$ be a separated  graph and
$S\subseteq \Cfin$. For each $X\in S$ we select an edge $e_X\in X$.
Let $\gamma$ and $\nu$ be paths in $E$ such that $r(\gamma )=
r(\nu)$ and $|\gamma |>0$, $|\nu |>0$. Let $\alpha$ and $\beta$ be
the terminal edges of $\gamma $ and $\nu$ respectively. Then the
path $\gamma \nu ^*$ in $\Ehat$ is said to be \emph{reduced with
respect to $S$} in case for every $X\in S$ we have $(\alpha ,\beta
)\ne (e_X,e_X)$. Moreover, $\gamma \nu ^*$ is called \emph{reduced
with respect to $S$} in case either $\gamma$ or $\nu$ has length
zero, i.e., all real and ghost paths (including the trivial ones)
are automatically reduced with respect to any subset $S$ of $\Cfin$.

A path $p$ as in (\ref{normalform1}), such that $\nu_i$ and
$\lambda_{i+1}$ are $C$-separated paths for each $i=1,\dots ,r-1$,
is said to be \emph{reduced with respect to $S$} in case
$\lambda_i\nu^*_i$ is reduced with respect to $S$ for each
$i=1,\dots,r$.
\end{definition}

\begin{theorem}
\label{basisCSE} Let $(E,C)$ be a separated graph and let $S$ be a
subset of $\Cfin$. Then a $K$-basis of $CL_K(E,C,S)$ is given by the
family $\calB'$ consisting of all the paths in the set $\calB$
described in Proposition {\rm\ref{basisCE}} which are reduced with
respect to $S$.
\end{theorem}

\begin{proof} Let $CL_K(E,C,S)\oplus K\cdot\tilde1$ be the formal unitization of $CL_K(E,C,S)$, where $\tilde1$ is a new identity element. We use Bergman's version of the Diamond Lemma \cite[Lemma 1.1, Theorem 1.2]{Berg3} to show that $\calB'\sqcup \{\tilde1\}$ is a basis for $CL_K(E,C,S)\oplus K\cdot\tilde1$, from which the theorem follows.

Let $W$ be the free abelian monoid on the set $E^0\sqcup E^1\sqcup (E^1)^*=
E^0\sqcup \Ehat^1$. For $X\in S$, choose $e_X\in X$ as in Definition
\ref{def:redform1}, and set $X' := X\setminus \{e_X\}$. Define a
weight function $\wt:W \to \Zplus$ so that
\begin{enumerate}
\item[] $\wt(v)=1$ for all $v\in E^0$,
\item[] $\wt(e_X)=2$ for all $X\in S$,
\item[] $\wt(f)=1$ for all other $f\in \Ehat^1$,
\end{enumerate}
and so that the weight of any word is the sum of the weights of its
letters. (In particular, $\wt(1_W)=0$.) Then define a partial order
$\le$ on $W$ by the following rule:
$$a\le b \ \iff\ \ a=b \text{\ or\ } \wt(a)<\wt(b).$$
Observe that $\le$ is a semigroup ordering on $W$, and that it
satisfies the descending chain condition.

Next, let $F$ denote the monoid algebra $K[W]$, that is, the free
unital $K$-algebra on $E^0\sqcup \Ehat^1$, and let $\calS$ be the
reduction system in $F$ consisting of the following pairs:
\begin{enumerate}
\item $(vw,\delta_{v,w}v)$ for $v,w\in E^0$,
\item $(ve,\delta_{v,s(e)}e)$ for $v\in E^0$ and $e\in \Ehat^1$,
\item $(ew,\delta_{w,r(e)}e)$ for $w\in E^0$ and $e\in \Ehat^1$,
\item $(ef,0)$ for $e,f\in \Ehat^1$ with $r(e)\ne s(f)$,
\item $(e^*f,\delta_{e,f}r(e)$ for $e,f\in X\in C$,
\item $(e_Xe^*_X,\, s(e_X)-\sum_{e\in X'} ee^*)$ for $X\in S$.
\end{enumerate}
Then $CL_K(E,C,S)$ may be presented as $F/I$ where $I$ is the ideal
$\langle W_\sigma-f_\sigma \mid (W_\sigma,f_\sigma) \in
\calS\rangle$ of $F$. Some of these reductions are redundant as far
as generating $I$ is concerned, namely the cases $(ab,0)$ of (2), (3), (4). However, these reductions are needed in order to resolve
certain ambiguities. Observe that the partial order $\le$ on $W$ is
compatible with $\calS$. (The assignment $\wt(e_X)=2$ ensures
compatibility with the reductions (6).)

In order to apply the Diamond Lemma, we must verify that all
ambiguities of $\calS$ are resolvable. Since the first terms
$W_\sigma$ of the pairs $(W_\sigma,f_\sigma) \in \calS$ all have
length 2 (as words in $W$), there are no inclusion ambiguities. There are four
families of overlap ambiguities, corresponding to certain products of the
following types:
\begin{enumerate}
\item[(a)] $vwx$, for $v,w,x\in E^0$,
\item[(b)] $vwe$, $vew$, $evw$, for $v,w\in E^0$ and $e\in \Ehat^1$,
\item[(c)] $vef$, $evf$, $efv$, for $v\in E^0$ and $e,f\in \Ehat^1$,
\item[(d)] $efg$, for $e,f,g\in \Ehat^1$.
\end{enumerate}
For instance, a product $vef$ as in (c) is an ambiguity only if $r(e)\ne s(f)$, or $e^*,f\in X$ for some $X\in C$, or $(e,f)= (e_X,e^*_X)$ for some $X\in C$.
We indicate resolutions for various cases of these ambiguities,
leaving the others to the reader. Reductions will be denoted by
$\longmapsto$.

The resolution of an overlap ambiguity of type (a) can be given as
follows:
\begin{align*}
(vw)x &\longmapsto (\delta_{v,w}w)x \longmapsto \delta_{v,w}\delta_{w,x}w \\
v(wx) &\longmapsto v(\delta_{w,x}w) \longmapsto
\delta_{w,x}\delta_{v,w}w.
\end{align*}
Those of type (b) are resolved in the same manner.

For overlaps of type (c), assume first that $r(e)\ne s(f)$. We then
have
\begin{align*}
(ve)f &\longmapsto (\delta_{v,s(e)}e)f \longmapsto 0 \\
v(ef) &\longmapsto v(0) = 0,
\end{align*}
and similarly for the cases $evf$ and $efv$. The case $evf$ with
$r(e)=s(f)$ resolves trivially. Otherwise, we only have ambiguities
to resolve for $ve^*f$ and $e^*fv$ with $e,f\in X\in C$, and for
$ve_Xe^*_X$ and $e_Xe^*_Xv$ with $X\in S$. The case $ve^*f$ resolves
as
\begin{align*}
(ve^*)f &\longmapsto (\delta_{v,r(e)}e^*)f \longmapsto \delta_{v,r(e)}\delta_{e,f} r(e) \\
v(e^*f) &\longmapsto v(\delta_{e,f} r(e))\longmapsto \delta_{e,f}
\delta_{v,r(e)} r(e),
\end{align*}
and the case $e^*fv$ is similar. Now consider $ve_Xe^*_X$ for $v\in
E^0$ and $X\in S$, say $X\in C_w$. On one hand,
$$(ve_X)e^*_X \longmapsto (\delta_{v,w}e_X)e^*_X \longmapsto \delta_{v,w} \bigl( w- \sum_{e\in X'} ee^* \bigr).$$
On the other hand,
$$v(e_Xe^*_X) \longmapsto v \bigl( w- \sum_{e\in X'} ee^* \bigr) \longmapsto \bigl( |X| \text{\ reductions} \bigr) \longmapsto \delta_{v,w} \bigl( w- \sum_{e\in X'} ee^* \bigr),$$
since $vw\longmapsto \delta_{v,w}w$ and $vee^* \longmapsto
\delta_{v,w}ee^*$ for $e\in X'$. The resolution of $e_Xe^*_Xv$ is
similar.

The overlaps of type (d) must be separated according to whether or
not $r(e)=s(f)$ and whether or not $r(f)=s(g)$. The cases in which
$r(e)\ne s(f)$ and/or $r(f)\ne s(g)$ are resolved in the same manner
as those above. For the remaining situation, note that there are no
overlap ambiguities $abc$ in which $ab$ and $bc$ are both of type
$e^*f$ or both of type $e_Xe^*_X$. This leaves only the cases
$e_Xe^*_Xf$ and $f^*e_Xe^*_X$ with $f\in X\in S$, say $X\in C_v$. We
give the resolution of $e_Xe^*_Xf$; that for $f^*e_Xe^*_X$ is
analogous. On one hand,
$$(e_Xe^*_X)f \longmapsto \bigl( v- \sum_{e\in X'} ee^* \bigr) f \longmapsto f- \sum_{e\in X'} ee^*f \longmapsto\cdots\longmapsto \begin{cases} f-0=e_X &(\text{if\ } f=e_X)\\ f-ff^*f\longmapsto 0 &(\text{if\ } f\ne e_X). \end{cases}$$
Thus, $(e_Xe^*_X)f$ reduces to $\delta_{f,e_X}e_X$. Since also
$$e_X(e^*_Xf) \longmapsto e_X (\delta_{e_X,f} r(e_X)) \longmapsto \delta_{f,e_X} e_X,$$
this ambiguity is resolved.

Thus, all ambiguities of $\calS$ are resolvable. Hence, the cosets
of the irreducible words in $W$ form a basis for $CL_K(E,C,S)\oplus
K\cdot\tilde1$. It only remains to show that the irreducible words
are precisely the elements of $\calB'\sqcup \{\tilde1\}$. That the
elements of the latter set are irreducible is clear. Conversely, let
$w= a_1a_2\cdots a_n$ be an irreducible word in $W$, where the
$a_i\in E^0 \sqcup \Ehat^1$. We cannot have any $a_i\in E^0$ unless
$n=0,1$, in which case either $w=\tilde1$ or $w=a_1\in E^0$. Now
assume that $n\ge2$ and that all $a_i\in \Ehat^1$. For
$i=1,\dots,n-1$, we must have $r(a_i)= s(a_{i+1})$, and we cannot
have either $(a_i,a_{i+1})= (e^*,f)$ with $e,f\in X\in C$ or
$(a_i,a_{i+1})= (e_X,e^*_X)$ with $X\in S$. From this, we conclude
that $w\in \calB'$, as desired, and the proof is complete.
\end{proof}

When $S=\Cfin $ we simply refer to the elements of $\calB '$ as the
\emph{reduced elements} of $\calB$. As $L_K(E,C) =CL_K(E,C,\Cfin)$,
we immediately get:

\begin{corollary}\label{basisL(E,C)}
Let $\calB$ be the canonical basis of $C_K(E,C)$ given in
Proposition {\rm\ref{basisCE}}. Then the set $\calB '$ of all the
reduced elements of $\calB$ is a $K$-basis for $L_K(E,C)$.
\end{corollary}

\begin{remark}
\label{ind-field} Note that the fact that $K$ is a field does not
play any role in the determination of the bases $\calB$ and $\calB'$
of the algebras $C_K(E,C)$ and $L_K(E,C)$ respectively, so that the
same arguments would show that the corresponding $R$-algebras
$C_R(E,C)$ and $L_R(E,C)$ are free as $R$-modules, for any nonzero
unital commutative ring $R$.
\end{remark}

We are ready now to give some concrete examples. In the first place,
we analyse the algebras corresponding to separated graphs with only
one vertex. For a cardinal number $\aleph $, denote by $L_K(\aleph
)$ the Leavitt algebra of type $(1,\aleph )$, namely, the Leavitt
path algebra of the (non-separated) graph with one vertex and
$\aleph$ edges. (For $\aleph =1$, $L_K(1)=K[t,t^{-1}]$.)

\begin{proposition}
\label{1-vertex} Assume that $(E,C)$ is a separated graph and that
$|E^0|=1$. Then we have
$$L_K(E,C)\cong \bgast_{X\in C}\; L_K(|X|) , $$
that is, $L_K(E,C)$ is a free product over $K$ of Leavitt path
algebras of type $(1,|X|)$, for $X\in C$.
\end{proposition}

\begin{proof}
Denote by $L$ the free product $\bgast_{X\in C}\; L_K(|X|)$. To keep the
different algebras $L_K(|X|)$ apart (since many sets in $C$ may have
the same cardinality), identify the copy of $L_K(|X|)$ corresponding
to a set $X\in C$ with the Leavitt path algebra $L_K(E_X)$ where
$E_X$ is the subgraph of $E$ with $E^0_X= E^0$ and $E^1_X= X$. Each
inclusion map $E_X\to E$ induces a $K$-algebra homomorphism
$L_K(|X|) \to L_K(E,C)$, and the family of these homomorphisms
extends uniquely to a unital $K$-algebra homomorphism $\varphi: L\to
L_K(E,C)$.

Each $e\in E^1$ belongs to a unique $X\in C$, and the symbols $e$
and $e^*$ represent elements of both $L_K(E,C)$ and $L_K(|X|)$. Let
$\ebar$ and $\ebar^*$ denote the elements corresponding to $e$ and
$e^*$ in the canonical copy of $L_K(|X|)$ inside $L$. The elements
$1\in L$ and $\ebar$, $\ebar^*$ for $e\in E^1$ satisfy the defining
relations of $L_K(E,C)$, so there is a unique unital $K$-algebra
homomorphism $\psi: L_K(E,C)\to L$ sending $e\mapsto \ebar$ and
$e^*\mapsto \ebar^*$ for all $e\in E^1$. It is clear that
$\varphi\psi$ and $\psi\varphi$ are identity maps. This shows the
result.
\end{proof}

\begin{example}  \label{Lmn}

We consider now a key class of examples, the separated graphs which
correspond to the Leavitt algebras $L_K(m,n)$ for $1\le m \le n$.
Indeed we can think of these Leavitt path algebras as versions of
$L_K(m,n)$ which are generated by ``partial isometries". Let us
consider the separated graph $(E(m,n),C(m,n))$, where
\begin{enumerate}
\item $E(m,n)^0 := \{v,w\}$ (with $v\ne w$).
\item $E(m,n)^1 :=\{\alpha_1,\dots , \alpha_n,\beta_1,\dots ,\beta_m\}$ (with $n+m$ distinct edges).
\item $s(\alpha_i)=s(\beta _j) =v$ and $r(\alpha _i)=r(\beta _j)=w$
for all $i$, $j$.
\item $C(m,n)= C(m,n)_v := \bigl\{ \{\alpha_1,\dots ,\alpha_n\},\, \{\beta _1,\dots, \beta _m
\} \bigr\}$.
\end{enumerate}
We will show that the structure of
$A_{m,n}:=L_K(E(m,n),C(m,n))$ is closely related to the structure of
the classical Leavitt algebra $L_K(m,n)$.

Observe that the corner algebras $vA_{m,n}v$ and $wA_{m,n}w$ are full corners of
$A_{m,n}$, that is, $A_{m,n}vA_{m,n}=A_{m,n}wA_{m,n}=A_{m,n}$. In particular,
$vA_{m,n}v$, $wA_{m,n}w$, and $A_{m,n}$ are Morita equivalent to each other. We will describe
their structure below.
\end{example}

Recall that $L_K(m,n)$ is generated by elements $X_{ij}$ and $X^*_{ij}$,
for $i=1,\dots,m$ and $j=1,\dots,n$, such that $XX^* = I_m$ and $X^*X =
I_n$, where $X$ denotes the $m\times n$ matrix $(X_{ij})$ and $X^*$
denotes the *-transpose of $X$, i.e., the $n\times m$ matrix with entries
$(X^*)_{ji} = X^*_{ij}$.

The isomorphism $wA_{m,n}w \cong L_K(m,n)$ below was discovered by Pardo \cite{Par}.
We thank him for permission to give it here.

\begin{proposition}  \label{keyexample}
Let $m\le n$ be positive integers, and define $E(m,n)$, $C(m,n)$,
$A_{m,n}$ as above.
\begin{enumerate}
\item There are $K$-algebra isomorphisms
\begin{align*}
A_{m,n} &\cong M_{m+1}(L_K(m,n)) \cong M_{n+1}(L_K(m,n)) \\
vA_{m,n}v &\cong M_m(L_K(m,n)) \cong M_n(L_K(m,n))  &wA_{m,n}w
&\cong L_K(m,n).
\end{align*}
\item The monoids $\mon{L_K(m,n)}$, $\mon{A_{m,n}}$, $\mon{vA_{m,n}v}$, and $\mon{wA_{m,n}w}$ are all of the form $\langle x\mid mx=nx\rangle$, where the generator $x$ corresponds to the classes $[1]$, $[w]$, $[\alpha_1\alpha^*_1]$, $[w]$ in the four respective cases.
\item Moreover, $vA_{m,n}v$ is the Bergman algebra obtained from
$R:=M_n(K)*M_m(K)$ by adjoining a universal isomorphism between the
left $R$-modules $Rf_{11}$ and $Rg_{11}$, where $(f_{ij})_{i,j=1}^n$
and $(g_{ij})_{i,j=1}^m$ are sets of matrix units corresponding to
the factors $M_n(K)$ and $M_m(K)$ respectively.
\item There is a surjective unital $K$-algebra homomorphism $\rho: L_K(m,n) \rightarrow vA_{m,n}v$. If $\mon{L_K(m,n)}$ and $\mon{vA_{m,n}v}$ are identified with $\langle x\mid mx=nx\rangle$ as in {\rm(2)}, then $\mon{\rho}$ is given by multiplication by $m$ {\rm(}equivalently, multiplication by $n${\rm)}.
\end{enumerate}
\end{proposition}

\begin{proof} Set $A := A_{m,n}$ and $L := L_K(m,n)$. We construct various $K$-algebra homomorphisms between algebras presented by generators and relations. In all cases, it is routine to check that the appropriate relations are satisfied by the proposed images for the generators, and we omit these details.

(1) We identify $L$ with the diagonal copies of itself in the
various matrix algebras $M_d(L)$. Let $(e_{ij})_{i,j=1}^{m+1}$ be
the standard family of matrix units in $M_{m+1}(L)$, and observe
that $M_{m+1}(L)$ is presented by the generators $e_{ij}$, $X_{ij}$,
$X^*_{ij}$ together with the following three types of relations:
\begin{enumerate}
\item[(a)] The defining relations for the $X_{ij}$ and $X^*_{ij}$ in $L$.
\item[(b)] The matrix unit relations for the $e_{ij}$.
\item[(c)] The commutation relations $e_{kl}X_{ij}= X_{ij}e_{kl}$ and $e_{kl}X^*_{ij}= X^*_{ij}e_{kl}$ for all $i$, $j$, $k$, $l$.
\end{enumerate}
Moreover, $M_{m+1}(L)$ is a free left (or right) $L$-module with
basis $\{e_{ij} \mid 1\le i,j\le m+1\}$. Analogous statements hold
for $M_m(L)$, and we identify $M_m(L)$ with the corner
$eM_{m+1}(L)e$ where $e:= e_{11}+ \cdots+ e_{mm}$.

There exist a $K$-algebra homomorphism $\psi: A\rightarrow
M_{m+1}(L)$ such that
\begin{align*}
\psi(v) &= e  &\psi(w) &= e_{m+1,m+1} \\
\psi(\alpha_i) &= \sum_{l=1}^m X_{li}e_{l,m+1}  &\psi(\alpha^*_i) &= \sum_{l=1}^m X^*_{li}e_{m+1,l}  &&(i=1,\dots,n) \\
\psi(\beta_j) &= e_{j,m+1}  &\psi(\beta^*_j) &= e_{m+1,j}
&&(j=1,\dots,m),
\end{align*}
and a $K$-algebra homomorphism $\phi: M_{m+1}(L) \rightarrow A$ such
that
\begin{align*}
\phi(e_{ij}) &= \beta_i\beta^*_j  &&(i,j= 1,\dots,m) \\
\phi(e_{i,m+1}) &= \beta_i  &&(i=1,\dots,m) \\
\phi(e_{m+1,j}) &= \beta^*_j  &&(j=1,\dots,n) \\
\phi(e_{m+1,m+1}) &= w \\
\phi(X_{ij}) &= \beta^*_i\alpha_j+ \sum_{l=1}^m \beta_l\beta^*_i\alpha_j\beta^*_l  &&(i=1,\dots,m;\ j=1,\dots,n) \\
\phi(X^*_{ij}) &= \alpha^*_j\beta_i+ \sum_{l=1}^m
\beta_l\alpha^*_j\beta_i\beta^*_l &&(i=1,\dots,m;\ j=1,\dots,n).
\end{align*}
Moreover, $\phi$ and $\psi$ are mutual inverses. Thus, $A\cong
M_{m+1}(L)$.

Isomorphisms between $A$ and $M_{n+1}(L)$ are obtained in a similar
fashion, by interchanging the roles of the $\alpha_i$ and $\beta_j$
in the constructions of $\psi$ and $\phi$ above.

Since $\psi$ maps $v$ to $e$, it restricts to an isomorphism of
$vAv$ onto $eM_{m+1}(L)e \equiv M_m(L)$. Similarly, $vAv\cong
M_n(L)$. Since $\psi$ maps $w$ to $e_{m+1,m+1}$, it restricts to an
isomorphism of $wAw$ onto $e_{m+1,m+1} M_{m+1}(L) e_{m+1,m+1}$, and
the latter algebra is isomorphic to $L$.

(2) By \cite[Theorem 6.1]{Berg2}, $\mon{L}= \langle x\mid
mx=nx\rangle$ with $x$ corresponding to the class of the free module
$_LL$, that is, to the class $[1_L]$. Applying the last isomorphism
of part (1) immediately yields $\mon{wAw}= \langle x\mid
mx=nx\rangle$ with $x$ corresponding to $[w]$. In view of the
equivalence
$$Aw \otimes_{wAw} (-): wAw\Mod \longrightarrow A\Mod,$$
it follows that $\mon{A}= \langle x\mid mx=nx\rangle$ with $x$
corresponding to $[w]$. Note that $[w]= [\alpha_1\alpha^*_1]$ in
$\mon{A}$ and that $\alpha_1\alpha^*_1\in vAv$. In view of the
equivalence
$$vA \otimes_A (-): A\Mod \longrightarrow vAv\Mod,$$
it thus follows that $\mon{vAv}= \langle x\mid mx=nx\rangle$ with
$x$ corresponding to $[\alpha_1\alpha^*_1]$.

(3) Let $\Rhat$ be the Bergman algebra obtained from $R$ by
adjoining a universal isomorphism between the modules $Rf_{11}$ and
$Rg_{11}$. Thus, $\Rhat$ is presented by generators $f_{ij}$,
$g_{ij}$, $u$, $u^*$ where
\begin{enumerate}
\item[(a)] The $f_{ij}$ satisfy the relations for a complete set of $n\times n$ matrix units.
\item[(b)] The $g_{ij}$ satisfy the relations for a complete set of $m\times m$ matrix units.
\item[(c)] $u=f_{11}ug_{11}$, $u^*= g_{11}u^*f_{11}$, $uu^*= f_{11}$, and $u^*u= g_{11}$.
\end{enumerate}
There is a $K$-algebra homomorphism $\theta: R\rightarrow vAv$ such
that
\begin{align*}
\theta(f_{ij}) &= \alpha_i\alpha^*_j  &\theta(g_{ij}) &=
\beta_i\beta^*_j
\end{align*} for all $i$, $j$. The universal property of the Bergman construction implies that $\theta$ extends uniquely to a $K$-algebra homomorphism $\thetahat: \Rhat \rightarrow vAv$ such that
\begin{align*}
\theta(u) &= \alpha_1\beta^*_1  &\theta(u^*) &= \beta_1\alpha^*_1
\,.
\end{align*}
There is a $K$-algebra homomorphism $\xi: M_m(L) \rightarrow \Rhat$
such that
\begin{align*}
\xi(e_{ij}) &= g_{ij} &&(i,j=1,\dots,m) \\
\xi(X_{ij}) &= \sum_{l=1}^m g_{li}f_{j1}ug_{1l}    &&(i=1,\dots,n;\ j=1,\dots,m) \\
\xi(X^*_{ij}) &= \sum_{l=1}^m g_{l1}u^*f_{1j}g_{il} &&(i=1,\dots,n;\
j=1,\dots,m).
\end{align*}
Let $\psi': vAv\rightarrow M_m(L)$ be the isomorphism obtained by
restricting $\psi$ to $vAv$. Then $\thetahat$ and $\xi\psi'$ are
mutual inverses, proving that $\Rhat\cong vAv$.

(4) There is a $K$-algebra homomorphism $\rho: L\rightarrow vAv$
such that
\begin{align*}
\rho(X_{ij}) &= \alpha_j\beta^*_i  &\rho(X^*_{ij}) &=
\beta_i\alpha^*_j
\end{align*}
for $i=1,\dots,m$ and $j=1,\dots,n$. Since $vAv$ is generated by the
elements $\alpha_j\alpha^*_i$, $\alpha_j\beta^*_i$,
$\beta_i\alpha^*_j$, $\beta_i\beta^*_j$, we see that $\rho$ is
surjective. As in (2), $\mon{L}= \langle x\mid mx=nx\rangle$ with
$x$ corresponding to the class $[1]$, and  $\mon{vAv}= \langle x\mid
mx=nx\rangle$ with $x$ corresponding to $[\alpha_1\alpha^*_1]$.
Since $\mon{\rho}$ maps $[1]$ to $[v]= \sum_{l=1}^m
[\alpha_l\alpha^*_l]= m[\alpha_1\alpha^*_1]$, we conclude that
$\mon{\rho}$ is given by multiplication by $m$. In the given monoid,
this is the same as multiplication by $n$.
\end{proof}

\begin{remark}
\label{grading}{\rm There is a canonical way to define a $\Z$-graded
structure in any $L_K(E,C)$, namely by setting $\text{deg}(e)=1$ and
$\text{deg}(e^*)=-1$ for every $e\in E^1$, and $\text{deg}(v)=0$ for
every $v\in E^0$. However in the particular case of the algebra
$A_{m,n}=L_K(E(m,n),C(m,n))$, this $\Z$-grading does not seem too
interesting, since for instance all elements of $vA_{m,n}v$ are
homogeneous of degree $0$, and thus this grading is not compatible
with the natural grading induced by the surjective homomorphism
$\rho \colon L_K(m,n)\to vA_{m,n}v$ of Proposition
\ref{keyexample}(4), in which the elements $\alpha _j\beta_i^*$
have degree $1$ and the elements $\beta_i\alpha _j^*$ have
degree $-1$. Observe that this $\Z$-grading coincides with the
grading on $vA_{m,n}v$ induced by the $\Z$-grading on $A_{m,n}$
obtained by setting $\text{deg}(\alpha _i)=1$,
$\text{deg}(\alpha_i^*)=-1$, and
$\text{deg}(\beta_j)=\text{deg}(\beta_j^*)=0$ for all $i$, $j$. In
terms of the isomorphic algebra $\hat{R}$, this $\Z$-grading agrees
with the one obtained by giving degree $0$ to all elements in
$M_n(K)*M_m(K)$ and giving degree $1$ to $u$ and degree $-1$ to
$u^*$. }
\end{remark}

%%%%%%%%%%%%%%%%%%%%%%%%%%%%%%%%%%%%%
\section{Direct limits}
\label{sect:colimits}

In this section we introduce two categories $\SSGr$ and $\SGr$ of
separated graphs and we study the functoriality and continuity of
the constructions in Section \ref{sect:algebras}. In particular, each object of $\SSGr$ is a direct limit of subobjects based on finite graphs, from which we obtain that every Cohn-Leavitt algebra is a direct limit of Cohn-Leavitt algebras based on finite graphs. This will be used in the following section in determining the structure of the $\calV$-monoids of Cohn-Leavitt algebras. In the present section, we also prove that every Cohn-Leavitt algebra is hereditary.

\begin{definition} \label{defSSGr}{\rm
Define a category $\SSGr$ of separated graphs with distinguished
subsets as follows. The objects of $\SSGr$ are triples $(E,C,S)$,
where $(E,C)$ is a separated graph and $S$ is a subset of $\Cfin$. A
morphism $\phi: (E,C,S) \rightarrow (F,D,T)$ in $\SSGr$ is any graph
morphism $\phi: E\rightarrow F$ such that
\begin{enumerate}
\item $\phi^0$ is injective.
\item For each $X\in C$ there is a (unique) $Y\in D$ such that $\phi ^1 $
restricts to an injective map $X\to Y$. Note that the assignment
$X\mapsto Y$ defines a set map $\phitil\colon C\to D$. Moreover, for
all $v\in E^0$ and $X\in C_v$, we have $\phitil(X)\in
D_{\phi^0(v)}$. Hence, $\phitil(C_v) \subseteq D_{\phi^0(v)}$.
\item The map $\phitil\colon C\to D$ restricts to a map $S\to T$ and, for each $X\in S$,
$\phi ^1$ restricts to a bijection $X\to \phitil(X)$.
\end{enumerate}}
\end{definition}

\begin{definition} \label{defSGr}{\rm
Define a category $\SGr$ of separated graphs as the full subcategory
of $\SSGr$ whose objects are all triples $(E,C, \Cfin )$. When
working with the category $\SGr$ we will use the simplified notation
$(E,C)$ instead of $(E,C,\Cfin)$, since the third component is
determined by the second.

An \emph{SG-subgraph} of a separated graph $(F,D)$ is any separated
graph $(E,C)$ such that $E$ is a subgraph of $F$ and the inclusion
map $E\rightarrow F$ is a morphism in $\SGr$. }
\end{definition}

\begin{proposition}
\label{colim} The categories $\SSGr$ and $\SGr$ admit arbitrary
direct limits. Indeed, given a directed system $\calD:= \{
(E_i,C_i,S_i),\; f_{ji}\mid i,j\in I,\; j\ge i\}$ in either category, the underlying sets $E^0$, $E^1$, $C$, $S$ of a direct limit $(E,C,S)$ of this system
are just direct limits of the respective systems of sets $E_i^0$, $E_i^1$
$C_i$, $S_i$ in the category of sets. Moreover, $E$ is a
direct limit of the system $\{E_i,f_{ji}\}$ in the category of
directed graphs.
\end{proposition}

\begin{proof} For $l=0,1$, let $E^l$ be a direct limit of the directed system $\{ E^l_i,\, f^l_{ji} \}$ in the category of sets, with limit maps $\eta^l_i: E^l_i \rightarrow E^l$. Observe that since all the maps $f^0_{ji}$ are injective, all the maps $\eta^0_i$ are injective. For all $j\ge i$ in $I$, we have $s_{E_j}f^1_{ji}= f^0_{ji}s_{E_i}$, whence $\eta^0_js_{E_j}f^1_{ji}= \eta^0_jf^0_{ji}s_{E_i}= \eta^0_is_{E_i}$. Consequently, there is a unique map $s_E: E^1\rightarrow E^0$ such that $s_E\eta^1_i= \eta^0_is_{E_i}$ for all $i$. Likewise,  there is a unique map $r_E: E^1\rightarrow E^0$ such that $r_E\eta^1_i= \eta^0_ir_{E_i}$ for all $i$. Therefore, $E:= (E^0,E^1,s_E,r_E)$ is a graph. The pairs $\eta_i := (\eta^0_i,\eta^1_i)$ are graph morphisms, and $E$ together with the $\eta_i$ is a direct limit of the system $\{ E_i,\, f_{ji} \}$ in the category of directed graphs.

Consider a vertex $w\in E^0$ which is not a sink. Define a relation $\sim_w$ on $s_E^{-1}(w)$ as follows:
\begin{enumerate}
\item[] $e\sim_w f$ if and only if there exist $i\in I$, $X\in C_i$, and $e',f'\in X$ such that $\eta^1_i(e')=e$ and $\eta^1_i(f')=f$. (Necessarily, $X\in (C_i)_v$ for some $v\in E^0_i$ such that $\eta^0_i(v)=w$.)
\end{enumerate}
Observe that $\sim_w$ is an equivalence relation, and let $C_w$ be the set of all $\sim_w$-equivalence classes in $s_E^{-1}(w)$.

For any sink $w\in E^0$, set $C_w := \emptyset$. Now set $C := \bigsqcup_{w\in E^0} C_w$; then $(E,C)$ is a separated graph.

Consider $i\in I$ and $X\in C_i$. We claim that there is a unique set $Y\in C$ such that $\eta^1_i$ restricts to an injection $X\rightarrow Y$, and that if $X\in S_i$, then $\eta^1_i(X)=Y$.

For $j\ge i$, since $f_{ji}$ is a morphism in $\SSGr$, the map $f^1_{ji}$ restricts to an injection $X\rightarrow \ftil_{ji}(X)\in C_j$. In particular, it follows that $\eta^1_i$ restricts to an injection $X\rightarrow E^1$. Now $X\in (C_i)_v$ for some $v\in E^0_i$, and if $w := \eta^0_i(v)$, then $\eta^1_i(e') \sim_w \eta^1_i(f')$ for all $e',f'\in X$. Thus, by definition of $C_w$, there is a unique $Y\in C_w$ such that $\eta^1_i(X) \subseteq Y$; in fact, $Y= \bigcup_{j\ge i} \eta^1_j \bigl( \ftil_{ji}(X) \bigr)$. In case $X\in S_i$, we have that $f^1_{ji}$ restricts to a bijection $X\rightarrow \ftil_{ji}(X)$ for all $j\ge i$, from which it follows that $Y= \eta^1_i(X)$. This establishes the claim.

Now set $S := \{ \eta^1_i(X) \mid i\in I,\; X\in S_i \}$. We thus have an object $(E,C,S) \in \SSGr$, and each $\eta_i$ is a morphism $(E_i,C_i,S_i) \rightarrow (E,C,S)$ in $\SSGr$. It is routine to check that $(E,C,S)$ together with the morphisms $\eta_i$ is a direct limit in $\SSGr$ for $\calD$. Further, the sets $C$ and $S$, together with the set maps $\etatil_i$, are direct limits for the respective systems $\{C_i,\, \ftil_{ji}\}$ and $\{S_i,\, \ftil_{ji}\}$.

Finally, one checks that if all the $(E_i,C_i,S_i)$ are objects in
$\SGr$, then so is $(E,C,S)$, and this object together with the
morphisms $\eta_i$ forms a direct limit for $\calD$ in $\SGr$.
\end{proof}

\begin{definition}
\label{def:completesubobject} {\rm Let $(E,C,S)$ be an object in
$\SSGr$. A \emph{complete subobject} of $(E,C,S)$ is an object
$(F,D,T)$ of $\SSGr$ such that $F$ is a subgraph of $E$ and moreover
\begin{enumerate}
\item For each $v\in F^0$, we have $D_v=\{Y\cap F^1\mid Y\in C_v,\; Y\cap F^1\ne
\emptyset\}$.
\item $T=\{ Y\in S \mid Y\cap F^1\ne \emptyset \}$.
\end{enumerate}
Note that if $(F,D,T)$ is a complete subobject of $(E,C,S)$, the
inclusion $F\to E$ induces a morphism $(F,D,T)\to (E,C,S)$ in
$\SSGr$.}
\end{definition}

Condition (2) of the definition says
that if $Y\in S$ and $Y\cap F^1\ne \emptyset$, then $Y\subseteq
F^1$. It follows that $T= S\cap D$. Note, however, that in the presence of condition (1), the equality $T= S\cap D$ is weaker than condition (2). This definition generalizes the usual definition of complete
subgraphs of row-finite graphs, given in \cite[p. 161]{AMP}, if we
interpret a row-finite graph $E$ as a triple $(E, C,S)$ with
$C_v=\{s^{-1}(v)\}$ if $s^{-1}(v)\ne \emptyset$ and $C_v=\emptyset$
if $s^{-1}(v)=\emptyset$, and $S=\Cfin =C$.

\begin{proposition}
\label{dirlimcomplete} Every object $(E,C,S)$ in $\SSGr$ is the
direct limit of its finite complete subobjects $(F,D,T)$, that is,
of its complete subobjects $(F,D,T)$ such that both $F^0$ and $F^1$
are finite sets.
\end{proposition}

\begin{proof}
To show that an object $(E,C,S)$ of $\SSGr$ is the direct limit of
its finite complete subobjects, it is enough to prove that for every
finite subset $A$ of $E^0\sqcup E^1$, there is a finite complete
subobject $(F,D,T)$ of $(E,C,S)$ such that $A\subseteq F^0\sqcup F^1$.
Let $E_1$ be the subgraph of $E$ generated by $A$, that is,
$E_1^1=A\cap E^1$ and $E_1^0=(A\cap E^0)\cup s_E(E_1^1)\cup
r_E(E_1^1)$. For $v\in E_1^0$, set
$$ F_v:= s_{E_1}^{-1}(v)\cup \bigcup _{X\in S\cap C_v,\, X\cap A\ne \emptyset } X . $$
Observe that $F_v$ is a finite set for every $v\in E_1^0$. Let $F$
be the subgraph of $E$ generated by $E^0_1 \sqcup \bigsqcup_{v\in
E_1^0} F_v$, so that $s_F^{-1}(v)=F_v$ for $v\in E_1^0\subseteq F^0$
and $s_F^{-1}(v)=\emptyset $ for $v\in F^0\setminus E_1^0$. For
$v\in F^0$, set
$$D_v :=\{ Y\cap F^1\mid Y\in C_v,\; Y\cap F^1\ne \emptyset \}=\{ Y\cap F^1\mid
Y\in C_v,\; Y\cap A\ne \emptyset \},$$ and then set $D :=
\bigsqcup_{v\in F^0} D_v$. Finally, set
$$T :=\{ Y\in S \mid Y\cap F^1\ne \emptyset \} = S\cap D.$$
 It follows
that $(F,D,T)$ is a finite complete subobject of $(E,C,S)$ such that
$A\subseteq F^0\sqcup F^1$, as desired.
\end{proof}

Recall that a functor is said to be \emph{continuous} if it preserves direct limits.

\begin{proposition}
\label{CLalg-dirlim} The assignment $(E,C,S)\mapsto CL_K(E,C,S)$
extends to a continuous functor $CL_K$ from the category $\SSGr$ to
the category of {\rm(}not necessarily unital\,{\rm)} $K$-algebras. Moreover,
every algebra $CL_K(E,C,S)$ is the direct limit, with injective
transition maps, of the algebras $CL_K(F,D,T)$, where $(F,D,T)$ runs
over the directed system of all the finite complete subobjects of
$(E,C,S)$.
\end{proposition}

\begin{proof}
Let $\phi \colon (E,C,S)\to (E',C',S')$ be a morphism in $\SSGr$. We
check that there is a unique $K$-algebra homomorphism
$CL_K(\phi)\colon CL_K(E,C,S)\to CL_K(E',C',S')$ such that
$$CL_K(\phi)(v)= \phi ^0(v),\qquad CL_K(\phi)(e)= \phi ^1(e),\qquad CL_K(\phi)(e^*)=\phi ^1(e)^*,$$
for $v\in E^0$ and $e\in E^1$. The relations (V) are preserved
because $\phi^0$ is injective, while (E1) and (E2) are preserved
because $\phi$ is a graph homomorphism. Now, let us show that
(SCK1), and (SCK2) for the members of $S$, are preserved by
$CL_K(\phi)$. Let $v\in E^0$ and $X\in C_v$. Then there is (a
unique) $Y\in C'_{\phi ^0(v)}$ such that $\phi ^1$ restricts to an
injective map $X\to Y$. Then for $e,f\in X$ we have $\phi^1 (e),
\phi^1 (f)\in Y$ and so $\phi^1 (e)^*\phi^1 (f)=\delta _{\phi^1
(e),\phi^1 (f)}r(\phi^1 (e))=\delta _{e,f}\phi ^0(r(e))=CL_K(\phi)
(e^*f)$. This verifies (SCK1). Note that we need $\phi ^1$ to be
injective on $X$ in order to guarantee that $\delta_{\phi^1
(e),\phi^1 (f)}=\delta _{e,f}$. Now assume that the above set $X$
belongs to $S$. Then $\phi ^1$ restricts to a bijection $X\to Y$, so
that
$$\phi ^0(v)=\sum _{e\in X}\phi^1 (e)\phi^1 (e)^*,$$
showing that the relation (SCK2) for $X\in S$ is preserved by
$CL_K(\phi)$. Since the defining relations for $CL_K(E,C,S)$ are
preserved by $CL_K(\phi)$, the existence and uniqueness of this map
follows. Functoriality is clear.

If $(F,D,T)$ is a complete subobject of $(E,C,S)$ and $\phi: (F,D,T)
\to (E,C,S)$ is the inclusion, then the homomorphism $CL_K(\phi):
CL_K(F,D,T)\to CL_K(E,C,S)$ is injective. This follows from Theorem
\ref{basisCSE}. Indeed, if a monomial from $CL_K(F,D,T)$ as in
(\ref{normalform1}) is in reduced form with respect to $T$, then
this monomial will be also in reduced form with respect to $S$,
thanks to properties (1) and (2) of the definition of complete
subobject, assuming that we have made a coherent choice of edges
$e_X\in X$ for $X\in T=\{ Y\in S \mid Y\cap F^1\ne \emptyset \}$. It follows from Theorem
\ref{basisCSE} that $CL_K(\phi)$ maps a basis for $CL_K(F,D,T)$
to a subset of a basis for $CL_K(E,C,S)$, and therefore $CL_K(\phi)$
is injective.

We show now that the functor $CL_K$ is continuous. Consider a directed system in the category $\SSGr$, say $\{(E_i,C_i,S_i),\; \phi _{ji} \mid i,j\in
I,\; j\ge i \}$, with
direct limit $(E,C,S)$ and limit maps $\eta_i: (E_i,C_i,S_i) \to
(E,C,S)$. Let the $K$-algebra $A$ be the direct limit of the
directed system $\{CL_K(E_i,C_i,S_i),\, CL_K(\phi_i)\}$, with limit
maps $\lambda_i: CL_K(E_i,C_i,S_i) \to A$. There is a unique
$K$-algebra homomorphism $\theta: A\to CL_K(E,C,S)$ such that
$\theta\lambda_i= CL_K(\eta_i)$ for all $i$, and we must show that
$\theta$ is an isomorphism. Since each $CL_K(E_i,C_i,S_i)$ is
generated by $E^0_i\sqcup \Ehat^1_i$, we see that $A$ is generated
by $\bigcup_{i\in I} \lambda_i(E^0_i\sqcup \Ehat^1_i)$. Given any
$w\in E^0$, write $w= \eta^0_i(v)$ for some $i\in I$ and $v\in
E^0_i$, and set $\xi^0(w)= \lambda_i(v)\in A$. Note that $\xi^0(w)$
is independent of the representation $w= \eta^0_i(v)$, and so we
have a well-defined map $\xi^0: E^0\to A$. Similarly, there is a
well-defined map $\xi^1: \Ehat^1\to A$ such that
$\xi^1(\eta^1_i(e))= \lambda_i(e)$ and $\xi^1(\eta^1_i(e)^*)=
\lambda_i(e)^*$ for all $i\in I$ and $e\in E^1_i$. Observe that the
elements $\xi^0(w)$, for $w\in E^0$, and $\xi^1(e)$, $\xi^1(e)^*$,
for $e\in E^1$, satisfy the defining relations of $CL_K(E,C,S)$.
Hence, $\xi^0\sqcup \xi^1$ extends uniquely to a $K$-algebra
homomorphism $\xi: CL_K(E,C,S) \to A$, and $\xi$ is an inverse for
$\theta$. Therefore $\theta$ is an isomorphism, as required.

Finally, observe that, by Proposition \ref{dirlimcomplete}, any
object $(E,C,S)$ in $\SSGr$ is the direct limit in $\SSGr$ of the
directed system of its finite complete subobjects. This, together
with what we have already shown before, gives the desired result
about the direct limit representation of $CL_K(E,C,S)$.
\end{proof}

With another dose of direct limits, we can prove that all Cohn-Leavitt algebras are hereditary in a suitable non-unital sense (see Definition \ref{hereditary}). This relies on work of Bergman and Dicks \cite{BD}, which requires us to work with direct limits of unital algebras over intervals of ordinals. Let us write $A^\sim$ for the $K$-algebra unitization of any $K$-algebra $A$, as in Definition \ref{defAtilde}.

The following easy observation will be useful: If $(R_t)_{t\in
T}$ is a nonempty family of unital hereditary $K$-algebras, then
$R := (\bigoplus_{t\in T} R_t)^\sim$ is hereditary. This holds because every left ideal of $R$ has one of the forms $\bigoplus_t I_t$ or $R(1-e)\oplus \bigoplus_{j=1}^n I_{t_j}$ where $I_t$ is a left ideal of $R_t$ and $e$ (in the second case) is the sum of the identity elements of $R_{t_1},\dots,R_{t_n}$, and similarly for right ideals.

\begin{theorem} \label{CLhereditary}
For any object $(E,C,S)$ of $\SSGr$, the algebra $CL_K(E,C,S)$ is hereditary.
\end{theorem}

\begin{proof} Set $A := CL_K(E,C,S)$. By Proposition \ref{alghered}, it suffices to show that $A^\sim$ is hereditary.

We first construct an algebra that corresponds to the unitization of the subalgebra of $A$ generated by $E^0$ together with the idempotents $ee^*$ for edges $e$ lying in sets $X\in C\setminus S$. There are three steps:
\begin{enumerate}
\item For each $X\in C\setminus S$, let $B_X$ be the $K$-algebra direct sum
of $|X|$ copies of $K$.
\item For each $v\in E^0$, let $A_v$ be the unital $K$-algebra coproduct of the
algebras $B_X^\sim$ for $X\in C_v\setminus S$. (The coproduct of an empty
family is $K$ itself.)
\item Set $A_0 := (\bigoplus_{v\in E^0} A_v)^\sim$.
\end{enumerate}
Note that each $B_X$ has a $K$-basis $(b_e)_{e\in X}$
consisting of a family of $|X|$ pairwise orthogonal idempotents. By the
observation above, $B_X^\sim$ is a unital (commutative) hereditary $K$-algebra. Next,  \cite[Theorem 3.4]{BD}
implies that each $A_v$ is hereditary, and then $A_0$ is hereditary by the observation above. For $v\in E^0$, identify $v$ with the identity element of
$A_v$. Then set $E_0= (E^0,\emptyset)$, and identify $L_K(E_0)$ with the subalgebra $\bigoplus_{v\in E^0} Kv \subseteq A_0$. There is a unique unital $K$-algebra homomorphism $\phi_0: A_0\rightarrow A^\sim$ such that $\phi_0(v)=v$ for $v\in E^0$ and $\phi_0(b_e)= ee^*$ for $e\in X\in C\setminus S$.

Now let $E^1_1$ be the union of all the sets $X\in C\setminus S$. Set $E_1=
(E^0,E^1_1)$ and $C_1= C\setminus S$. We next construct an algebra corresponding to $C_K(E_1,C_1)^\sim$. For the purpose, choose an ordinal $\gamma$ and a bijection $\alpha \mapsto e_\alpha$ from $[0,\gamma)$ to $E^1_1$. We build unital $K$-algebras $A_\alpha$ for $\alpha\in [0,\gamma]$ as follows. (We do not specifically label the obvious connecting homomorphisms $A_\alpha \rightarrow A_\beta$ for $0\le \alpha<\beta \le \gamma$ that the construction carries, and we treat them as inclusion maps.)
\begin{enumerate}
\item[(4)] Start with $A_0$ as in (3).
\item[(5)] For $\alpha\in [0,\gamma)$, let $A_{\alpha+1}$
be obtained from $A_\alpha$ by adjoining a universal isomorphism between the finitely generated projective left $A_\alpha$-modules
$A_\alpha b_{e_\alpha}$ and $A_\alpha r(e_\alpha)$.
\item[(6)] If $\beta\le\gamma$ is a limit ordinal and $A_\alpha$ has been
defined for all $\alpha<\beta$, take $A_\beta$ to be the direct limit of
$(A_\alpha)_{\alpha<\beta}$.
\end{enumerate}
We apply \cite[Theorem 3.4]{BD} a second time to see that $A_\gamma$ is hereditary. For $\alpha\in [0,\gamma)$, the algebra $A_{\alpha+1}$ is generated by $A_\alpha$ together with elements $x_\alpha$ and $y_\alpha$ universally satisfying the relations
\begin{align*}
x_\alpha &= b_{e_\alpha}x_\alpha r(e_\alpha)  &y_\alpha &= r(e_\alpha)y_\alpha b_{e_\alpha}  &x_\alpha y_\alpha &= b_{e_\alpha}  &y_\alpha x_\alpha &= r(e_\alpha).
\end{align*}
The homomorphism $\phi_0$ extends uniquely to a unital $K$-algebra homomorphism $\phi_1: A_\gamma \rightarrow A^\sim$ such that $\phi_1(x_\alpha)= e_\alpha$ and $\phi_1(y_\alpha)= e^*_\alpha$ for  $\alpha\in [0,\gamma)$.

Finally, we construct a further direct limit to reach $A^\sim$. Choose an ordinal $\mu>\gamma$ and a bijection $\kappa\mapsto
X_\kappa$ from $[\gamma,\mu)$ to $S$. Then build unital $K$-algebras $A_\kappa$ for
$\kappa\in [\gamma,\mu]$ as follows.
\begin{enumerate}
\item[(7)] Start with $A_\gamma$ as in the previous construction.
\item[(8)] For $\kappa\in [\gamma,\mu)$, let $A_{\kappa+1}$
be obtained from $A_\kappa$ by adjoining a universal isomorphism
between the finitely generated projective $A_\kappa$-modules
$A_\kappa v_\kappa$ and $\bigoplus_{e\in X_\kappa} A_\kappa r(e)$,
where $v_\kappa \in E^0$ is such that $X_\kappa\in C_{v_\kappa}\cap
S$.
\item[(9)] If $\lambda\in (\gamma,\mu]$ is a limit ordinal and $A_\kappa$
has been defined for all $\kappa<\lambda$, take $A_\lambda$ to be the
direct limit of $(A_\kappa)_{\gamma\le\kappa<\lambda}$.
\end{enumerate}
A third application of \cite[Theorem 3.4]{BD} yields that $A_\mu$ is hereditary. For $\kappa \in [\gamma,\mu)$, the algebra $A_{\kappa+1}$ is generated by $A_\kappa$ together with elements $x_e$ and $y_e$ for $e\in X_\kappa$ universally satisfying the relations
\begin{align*}
x_e &= v_\kappa x_er(e)  &y_e &= r(e)y_e v_\kappa \\
\sum_{e\in X_\kappa} x_ey_e &= v_\kappa  &y_ex_f &= \delta_{e,f} r(e)
\end{align*}
for $e,f\in X_\kappa$. The homomorphism $\phi_1$ extends uniquely to a unital $K$-algebra homomorphism $\phi: A_\mu \rightarrow A^\sim$ such that $\phi(x_e)= e$ and $\phi(y_e)= e^*$ for $e\in X_\kappa$, $\kappa\in [\gamma,\mu)$.

In the reverse direction, there is a unique unital $K$-algebra homomorphism $\psi: A^\sim \rightarrow A_\mu$ such that
\begin{align*}
\psi(v) &= v &&(v\in E^0) \\
\psi(e_\alpha) &= x_\alpha \text{\ and\ } \psi(e^*_\alpha) = y_\alpha &&(\alpha\in [0,\gamma)) \\
\psi(e) &= x_e \text{\ and\ } \psi(e^*)= y_e &&(e\in X_\kappa,\ \kappa\in [\gamma,\mu)).
\end{align*}
Clearly, $\psi$ and $\phi$ are mutual inverses. Therefore $A_\mu
\cong A^\sim$, and the theorem is
proved.
\end{proof}

%%%%%%%%%%%%%%%%%%%%%%%%%%%%%%%%%%%%%
\section{The monoid $\mon{CL_K(E,C,S)}$}
\label{sect:VCL}

We define an abelian monoid $M(E,C,S)$ for any separated graph $(E,C)$
with a distinguished subset $S\subseteq\Cfin$, and we prove that
$\mon{CL_K(E,C,S)}$ is naturally isomorphic to $M(E,C,S)$. This
extends previous results for Leavitt path algebras $L_K(E)$ of
row-finite graphs $E$ \cite[Theorem 3.5]{AMP}, for which case it was also proved that the graph monoid $M(E) \cong \mon{L_K(E)}$ is a refinement monoid \cite[Proposition 4.4]{AMP}. We shall prove that $M(E)$ is a refinement monoid for any non-separated graph $E$ (Corollary \ref{M(E)refinement}). In contrast, refinement does not always hold in $\mon{CL_K(E,C,S)}$ -- in fact, $\mon{CL_K(E,C,S)}$ can be an arbitrary conical monoid (Proposition \ref{allconical}).

\begin{definition}
\label{defgraphmon} Let $(E,C,S)$ be an object in $\SSGr$. We define
the \emph{graph monoid} $M(E,C,S)$ as the abelian monoid given by
the set of generators
$$E^0\sqcup
\{q_Z'\mid Z\subseteq X\in C,\; 0<|Z|<\infty\}$$ and the following
relations:
\begin{enumerate}
\item
$v= \bfr(Z)+q_Z'$ for $v\in E^0$, $Z\subseteq X\in C_v$, and
$0<|Z|<\infty$, where for a finite subset $Y$ of $E^1$ we set
${\mathbf r}(Y):=\sum _{e\in Y} r(e) .$
\item
$q_{Z_1}'=\bfr(Z_2\setminus Z_1)+q_{Z_2}'$ for finite nonempty
subsets $Z_1$ and $Z_2$ of $X\in C$ with $Z_1 \subsetneq Z_2$.
\item
$q_X'=0$ for $X\in S$.
\end{enumerate}
Of course the elements $q_Z'$ are intended to represent the
equivalence classes of the idempotents $v-\sum _{e\in Z}ee^*$ in
$CL_K(E,C,S)$, for $Z$ a finite nonempty subset of $X\in C_v$.

There is some redundancy among these generators and relations. In particular, we could omit the generators $q'_Z$ for nonempty proper subsets $Z$ of a set $X\in \Cfin$, since relation (2) gives $q'_Z$ in terms of $q'_X$, and relation (1) for $Z$ follows from the corresponding relation for $X$ in light of (2). In general, (1) could be viewed as a form of (2) with $Z_1=\emptyset$, except that the notation $q'_\emptyset$ would not be well-defined.

When working with objects $(E,C)$ from $\SGr$, we abbreviate the notation for the corresponding monoid to
$$M(E,C) := M(E,C,\Cfin).$$

For many later purposes, we shall assume that $(E,C)$ is finitely separated. In that case, the generators $q'_Z$ for $\emptyset \ne Z
\subsetneq X \in C$ are redundant, as  noted above.
Then, we can present $M(E,C,S)$ with the set of generators $E^0
\sqcup \{q'_X\mid X\in C\setminus S\}$ and the relations
\begin{enumerate}
\item[(4)] $v= \bfr(X)$ for $v\in E^0$ and $X\in C_v\cap S$.
\item[(5)] $v= \bfr(X)+q'_X$ for $v\in E^0$ and $X\in C_v\setminus S$.
\end{enumerate}

Now return to the general case, and consider a morphism $\phi:
(E,C,S)\to (E',C',S')$ in $\SSGr$. There is a unique monoid
homomorphism $M(\phi): M(E,C,S)\to M(E',C',S')$ sending $v\mapsto
\phi^0(v)$ for $v\in E^0$ and $q'_Z\mapsto q'_{\phi^1(Z)}$ for
nonempty finite sets $Z\subseteq X\in C$. (The latter assignments
are well-defined because if $Z$ is a nonempty finite subset of some
$X\in C$, then $\phi^1(Z)$ is a nonempty finite subset of
$\phitil(X) \in C'$.) The assignments $(E,C,S) \mapsto M(E,C,S)$ and
$\phi\mapsto M(\phi)$ define a functor $M$ from $\SSGr$ to the
category $\Mon$ of abelian monoids. It is easily checked (just as
for the functor $CL_K$ in Proposition \ref{CLalg-dirlim}) that $M$
is continuous.
\end{definition}

\begin{lemma} \label{nonzeroconical}
If $(E,C,S)$ is an object in $\SSGr$, then $M(E,C,S)$ is a nonzero,
conical monoid.
\end{lemma}

\begin{proof} We can present $M(E,C,S)$ as the quotient of the free abelian monoid $F$ on the set
$$E^0\sqcup
\{q_Z'\mid Z\subseteq X\in C,\; 0<|Z|<\infty\} \setminus \{q'_X\mid
X\in S\}$$ modulo the congruence $\sim$ generated by the relations
(1) and (2), where (1) is rewritten $v=\bfr(Z)$ in case $Z=X\in S$
and (2) is rewritten $q'_{Z_1}= \bfr(Z_2\setminus Z_1)$ in case
$Z_2=X\in S$. Observe that $\alpha\sim 0$ occurs in $F$ only when
$\alpha=0$. The lemma follows immediately, given that $E^0$ is
assumed to be nonempty.
\end{proof}

\begin{theorem}
\label{computVMECS} There is an isomorphism $\Gamma: M\to \calV\circ CL_K$ of functors $\SSGr\to
{\bf Mon}$, given as follows. For each object $(E,C,S)$ of $\SSGr$,
$$\Gamma(E,C,S) \colon
M(E,C,S)\to \mon{CL_K(E,C,S)}$$
is the monoid homomorphism sending $v\mapsto [v]$ for
$v\in E^0$ and and $q'_Z \mapsto [v-\sum _{e\in Z} ee^*]$ for finite
nonempty subsets $Z\subseteq X\in C_v$.
\end{theorem}

\begin{proof} It is easily seen that the maps $\Gamma(E,C,S)$ are well-defined monoid
homomorphisms, and that $\Gamma $ defines a natural transformation from $M$ to $\calV\circ
CL_K$.

We have observed that $M$ is continuous, and so is $\calV\circ
CL_K$, taking into account that $\calV$ is continuous and
Proposition \ref{CLalg-dirlim}. Thus, by making use of the second
part of Proposition \ref{CLalg-dirlim} we see that it is sufficient
to show that $\Gamma(E,C,S) $ is an isomorphism in the case where $E$ is a
finite graph.

We use induction on $|C|$ to establish the result for finite objects
$(E,C,S)$ in $\SSGr$. The result is trivial if $|C|=0$ (i.e., if
there are no edges in $E$). Assume that $\Gamma(E,C,S) $ is an isomorphism
for finite objects $(F,D,T)$ of $\SSGr$ with $|D|<n-1$ for some
$n\ge 1$, and let $(E,C,S)$ be a finite object in $\SSGr$ such that
$|C|=n$. Select $X$ in $C_v$, for some $v\in E^0$. We can apply
induction to the triple $(F,D,T)$ obtained from $(E,C,S)$ by
deleting all the edges in $X$, and leaving intact the structure
corresponding to the remaining subsets $Y\in C$ (keeping $F^0=E^0$).

Assume first that $X\in S$. Then $M(E,C,S)$ is obtained from
$M(F,D,T)$ by factoring out the relation $v=\bfr(X)$. On the other
hand, the algebra $CL_K(E,C,S)$ is the Bergman algebra obtained from
$CL_K(F,D,T)$ by adjoining a universal isomorphism between the
finitely generated projective modules $CL_K(F,D,T)v$ and $\bigoplus
_{e\in X} CL_K(F,D,T)r(e)$. Accordingly, it follows from
\cite[Theorem 5.2]{Berg2} that $\mon{CL_K(E,C,S)}$ is the quotient
of $\mon{CL_K(F,D,T)}$ modulo the relation $[v]=[\bfr(X)]$. Since
$\Gamma(F,D,T): M(F,D,T)\to \mon{CL_K(F,D,T)}$ is an isomorphism by
the induction hypothesis, we obtain that $\Gamma(E,C,S)$ is an
isomorphism in this case.

Assume now that $X\notin S $. In this case, $M(E,C,S)$ is obtained
from $M(F,D,T)$ by adjoining a new generator $q'_X$ and factoring
out the relation $v= \bfr(X)+ q'_X$. On the $K$-algebra side, we
shall make use of another of Bergman's constructions, namely ``the
creation of idempotents". Write $X=\{e_1,\dots ,e_m\}$, and recall
that $X\in C_v$, so that $s_E(e_i)=v$ for all $i$. Let $R$ be the
algebra obtained from $CL_K(F,D,T)$ by adjoining $m+1$ pairwise
orthogonal idempotents $g_1, \dots , g_m, q_X$ with
$$v= g_1+\cdots + g_m+q_X.$$
It follows from \cite[Theorem 5.1]{Berg2} that $\mon{R}$ is the
monoid obtained from $\mon{CL_K(F,D,T)}$ by adjoining $m+1$ new
generators $z_1,\dots ,z_m, q_X''$ and factoring out the relation
$[v]=\sum _{j=1}^m z_j+q_X''$.

Now, it is clear that $CL_K(E,C,S)$ is isomorphic to the Bergman
algebra obtained from $R$ by consecutively adjoining universal
isomorphisms between the left modules generated by the idempotents
$r(e_i)$ and $g_i$, for $i=1,\dots ,m$. It follows that
$\mon{CL_K(E,C,S)}$ is the monoid obtained from $\mon{CL_K(F,D,T)}$
by adjoining a new generator $q''_X$ and factoring out the relation
$[v]=[\bfr(X)]+q''_X$. Therefore, applying the induction hypothesis
to $(F,D,T)$, we again conclude that $\Gamma(E,C,S)$ is an
isomorphism.
\end{proof}

We conclude this section by noting that all conical abelian monoids appear as graph monoids of separated graphs.

\begin{proposition} \label{allconical}
If $M$ is any conical abelian monoid, there exists a finitely separated graph $(E,C)$ such that
$$M\cong M(E,C) \cong \mon{L_K(E,C)}.$$
\end{proposition}

\begin{proof} Choose a presentation of $M$ with a nonempty set
$\{ x_j\mid j\in J\}$ of generators and a nonempty set
$\{r_i\mid i\in I\}$ of relations
\begin{equation} \label{Mpres}
r_i\,:\qquad \sum_{j\in J} a_{ij}x_j= \sum_{j\in J} b_{ij}x_j\,,
\end{equation}
where for each $i$, we have $a_{ij}\ne 0$ for at least one but only
finitely many $j\in J$, and similarly for the $b_{ij}$. The relations can be chosen with these restrictions because $M$ is conical. For instance, if, for some $i$, all $a_{ij}=0$, then $\sum_j b_{ij}x_j =0$ and consequently $x_j=0$ whenever $b_{ij}>0$. In this case, we may delete those $x_j$ for which $b_{ij}>0$ from our set of generators.

Let $(E,C)$
be the finitely separated graph constructed as follows:
\begin{enumerate}
\item $E^0 := \{u_i\mid i\in I\}\sqcup \{v_j\mid j\in J\}$.
\item All the $u_i$ are sources, and all the $v_j$ are sinks.
\item For each $i\in I$ and $j\in J$, there are exactly $a_{ij}+b_{ij}$ edges with source $u_i$ and range $v_j$.
\item Each $s^{-1}(u_i)= X_{i1}\sqcup X_{i2}$ where $X_{i1}$ contains exactly $a_{ij}$ edges $u_i \rightarrow v_j$ for each $j$, and $X_{i2}$ contains exactly $b_{ij}$ edges $u_i \rightarrow v_j$ for each $j$.
Thus, $E^1= \bigsqcup_{i\in I} (X_{i1}\sqcup X_{i2})$.
\item $C := \bigsqcup_{v\in E^0} C_v$ where each $C_{u_i} := \{X_{i1},X_{i2}\}$ and each $C_{v_j} := \emptyset$.
\end{enumerate}

Since $(E,C)$ is finitely separated, the monoid $M(E,C)$ can be presented (as noted in Definition \ref{defgraphmon}) with the set of generators $E^0$ and the relations $v= \bfr(X)$ for $v\in E^0$ and $X\in C_v$. These relations are of the form
$$u_i= \sum_{j\in J} a_{ij} v_j \qquad\quad \text{and} \qquad\quad u_i= \sum_{j\in J} b_{ij} v_j$$
for $i\in I$. The isomorphism $M\cong M(E,C)$ follows immediately, and the final isomorphism is given by Theorem \ref{computVMECS}.
\end{proof}

\begin{corollary}  \label{cam=V}
Let $M$ be a conical abelian monoid. Then there exists a hereditary $K$-algebra $A:= L_K(E,C)$, for some finitely separated graph $(E,C)$, such that $\mon{A} \cong M$.
\end{corollary}

\begin{proof} Theorems \ref{CLhereditary} and \ref{computVMECS} and Proposition \ref{allconical}.
\end{proof}

Corollary \ref{cam=V}
incorporates a more explicit rendering of a result of Bergman
\cite[Theorem 6.4]{Berg2}, as corrected and extended by Bergman and
Dicks \cite[Remarks following Theorem 3.4]{BD}: Any conical abelian
monoid with an order-unit is isomorphic to $\mon{R}$ for some unital
hereditary $K$-algebra $R$. Our development does not replace the
mentioned result, since the theorems proved in \cite{Berg2} are
crucial to our proof of Theorem \ref{computVMECS}, and our algebras
are mostly non-unital.

%%%%%%%%%%%%%%%%%%%%%%%%%%%%%%%%%%%%%
\section{Refinement}
\label{sect:refinement}

In this section we initiate our systematic study of the properties
of the graph monoids $M(E,C,S)$ associated to separated
graphs $(E,C)$ with distinguished subsets $S\subseteq \Cfin$, focusing in particular on the Riesz refinement
property. For non-separated, row-finite graphs $E$, the corresponding monoid
$M(E)$ was proved to have
refinement in
 \cite[Proposition 4.4]{AMP}. This result can be extended
to general graphs $E$, once the definition of $M(E)$ is modified
along the lines of Definition \ref{defgraphmon} (use Theorem \ref{computVMECS} and \cite[Theorem 5.8]{Gdirlim}). It also follows from our results here (see Corollary \ref{M(E)refinement}). However, refinement does not always hold in $M(E,C,S)$ or even in $M(E,C)$. For example, let $E$ be the graph
$$\xymatrixrowsep{0.5pc}\xymatrixcolsep{6pc}\def\labelstyle{\displaystyle}
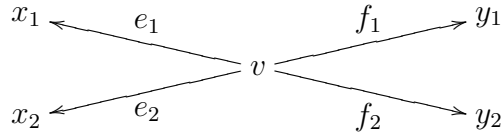
\xymatrix{
x_1 &&y_1 \\
&v \ar[ul]_{e_1} \ar[dl]^{e_2} \ar[ur]^{f_1} \ar[dr]_{f_2} \\
x_2 &&y_2 }$$ and take $C_v= \{X,Y\}$ where $X= \{e_1,e_2\}$ and $Y=
\{f_1,f_2\}$. Then $M(E,C)$ is the monoid presented by generators
$x_1$, $x_2$, $y_1$, $y_2$ and the relation $x_1+x_2= y_1+y_2$. This
is not a refinement monoid. Thus, extra conditions on $(E,C,S)$ are
needed to obtain refinement in $M(E,C,S)$.

We obtain refinement by imposing conditions of the following type. Suppose we have an object $(E,C,S) \in \SSGr$, a vertex $v\in E^0$, and distinct sets $X,Y\in S\cap C_v$. Then $v= \bfr(X)= \bfr(Y)$ in $M(E,C,S)$, and we require assumptions that allow a refinement of the equation $\bfr(X)= \bfr(Y)$. A sufficient condition is the existence of edges $g_{e,f}, h_{e,f} \in E^1$ for $e\in X$ and $f\in Y$, with $s(g_{e,f})= r(e)$, $s(h_{e,f})= r(f)$, and $r(g_{e,f})= r(h_{e,f})$ for all $e$, $f$, such that the sets $X_e := \{g_{e,f} \mid f\in Y\}$ and $Y_f := \{h_{e,f} \mid e\in X\}$ belong to $C_{r(e)} \cap S$ and $C_{r(f)} \cap S$ respectively. Then, in $M(E,C,S)$, we would have $r(e)= \sum_{f\in Y} r(g_{e,f})$ for $e\in X$ and $r(f)= \sum_{e\in X} r(h_{e,f})= \sum_{e\in X} r(g_{e,f})$ for $f\in Y$, providing the desired refinement of $\bfr(X)= \bfr(Y)$.

We actually impose a slightly weaker assumption, which we place on the free abelian monoid on a generating set for $M(E,C,S)$ (see Definition \ref{defstarstar}). This asumption has a smoother form when all the sets in $C$ are finite (Definition \ref{defstar}), and so we begin with that case. An easy reduction will allow us to assume, in addition, that $S=C$. Later, we reduce the general case to the one just mentioned.

\begin{assumption*} Throughout this section, $(E,C,S)$ will denote a fixed, arbitrary object of $\SSGr$. Until Proposition \ref{refinement}, we also assume that $(E,C)$ is a finitely separated graph (that is, $C=\Cfin$).
\end{assumption*}

Set $Q^0 := \{ q'_X \mid X\in C\setminus S\}$, and let $F$ be the free abelian monoid on $E^0 \sqcup Q^0$. (We use the same notation for elements of $E^0\sqcup Q^0$ in $F$ as for their images in $M(E,C,S)$.) As above, set $\bfr(X)= \sum_{e\in X} r(e)$ for $X\in C$ (in either $F$ or $M(E,C,S)$). Then set
$$\bfrho(X)= \bfrho_E(X) := \begin{cases} \bfr(X) &(X\in S)\\ \bfr(X)+q'_X &(X\in C\setminus S). \end{cases}$$
We identify $M(E,C,S)$ with the quotient monoid $F/{\sim}$, where $\sim$ is the congruence on $F$ generated by the relations
$$v\sim \bfrho(X)$$
for $v\in E^0$ and $X\in C_v$.

\begin{definition} \label{defsquigarrow}
[Assuming $C=\Cfin$] For $\alpha,\beta \in F$, write $\alpha \rightsquigarrow_1 \beta$
to denote the following situation:
\begin{itemize}
\item[] $\alpha= \sum_{i=1}^k u_i+\sum_{i=k+1}^m u_i$ and $\beta= \sum_{i=1}^k \bfrho(X_i)+\sum_{i=k+1}^m u_i$ for some $u_i\in E^0\sqcup Q^0$ and some $X_i\in C_{u_i}$, where $u_1,\dots ,u_k\in E^0\setminus \Si(E)$.
\end{itemize}
We view $\alpha=0$ as an empty sum of the above form (i.e.,
$k=m=0$), so that $0\rightsquigarrow_1 0$. Note also, taking $k=0$, that $\alpha \rightsquigarrow_1 \alpha$ for all $\alpha \in F$.

Now write $\alpha\rightsquigarrow_n \beta$, where $n\ge1$, in case there is a finite
string $\alpha=\alpha_0 \rightsquigarrow_1 \alpha_1
\rightsquigarrow_1 \cdots \rightsquigarrow_1 \alpha_n=\beta$. Finally, $\alpha\rightsquigarrow\beta$ just means
that $\alpha\rightsquigarrow_n\beta$ for some $n$. Observe that if
$\alpha \rightsquigarrow_m \beta$ and $\alpha' \rightsquigarrow_n
\beta'$, then $\alpha+\alpha' \rightsquigarrow_s
\beta+\beta'$ with $s= \max\{m,n\}$. (This follows from the fact that $\alpha \rightsquigarrow_1 \beta$ implies $\alpha+\gamma \rightsquigarrow_1 \beta+\gamma$ for all $\gamma\in F$.)
\end{definition}

\begin{definition}
\label{defstar} [Assuming $C=\Cfin$] {\bf Assumption $(*)$:} For all $v\in E^0$ and all
$X,Y\in C_v$, there exists $\gamma\in F$ such that $\bfrho(X)
\rightsquigarrow_1 \gamma$ and $\bfrho(Y) \rightsquigarrow_1 \gamma$.

This assumption only needs to be imposed when $X$ and $Y$ are disjoint, since otherwise $X=Y$ and the conclusion is trivially satisfied. Upcoming inductions (see the proof of Lemma \ref{confluence}) require the use of $\rightsquigarrow_1$ rather than $\rightsquigarrow$ in Assumption $(*)$.
\end{definition}

\begin{construction} \label{removeS}
Enlarge $E$ to a graph $\Etil$ by adjoining new vertices $w_X$ for $X\in C\setminus S$ and new edges $f_X: v\rightarrow w_X$ for $v\in E^0$ and $X\in C_v\setminus S$. For $v\in E^0$ and $X\in C_v$, set
$$\Xtil := \begin{cases} X &(X\in S)\\ X\sqcup \{f_X\} &(X\notin S), \end{cases}$$
and then set $\Ctil_v= \{ \Xtil \mid X\in C_v \}$, a partition of $s^{-1}_{\Etil}(v)$. Since the vertices $w_X\in \Etil^0$ are sinks, we just set $\Ctil_{w_X}= \emptyset$ for $X\in C\setminus S$. We have now built a separated graph $(\Etil,\Ctil)$, where $\Ctil= \bigsqcup_{w\in \Etil} \Ctil_w$. Since all the sets in $C$ are finite, so are those in $\Ctil$.

Take $\Stil=\Ctil$, and observe that $M(\Etil,\Ctil)= M(\Etil,\Ctil,\Stil)$ is the abelian monoid presented by the generating set $\Etil^0$ and the relations $w= \bfr_{\Etil}(\Xtil)$ for $w\in \Etil^0$ and $\Xtil\in \Ctil_w$. These relations are of two types:
\begin{align*}
v &= \bfr_E(X)  &&(v\in E^0,\ X\in C_v\cap S) \\
v &= \bfr_E(X)+w_X  &&(v\in E^0,\ X\in C_v\setminus S);
\end{align*}
there are no relations for $w\in \Etil^0\setminus E^0$. Comparing presentations, we see that there is a monoid isomorphism
$$M(E,C,S) \longrightarrow M(\Etil,\Ctil)$$
sending $v\mapsto v$ for $v\in E^0$ and $q'_X \mapsto w_X$ for $X\in C\setminus S$.

Now let $\Ftil$ be the free abelian monoid on $\Etil^0$, and let $\theta: F\rightarrow \Ftil$ be the isomorphism such that $\theta(v)=v$ for $v\in E^0$ and $\theta(q'_X)= w_X$ for $X\in C\setminus S$. Observe that
$$\theta(\bfrho_E(X))= \bfr_{\Etil}(\Xtil)$$
for all $X\in C$. It follows that if $\alpha,\beta \in F$ are any elements satisfying $\alpha \rightsquigarrow_1 \beta$, then $\theta(\alpha) \rightsquigarrow_1 \theta(\beta)$ in $\Ftil$. It is now clear that if Assumption $(*)$ holds for $(E,C,S)$, then it also holds for $(\Etil,\Ctil)$.

We aim to prove that when $(*)$ holds, $M(E,C,S)$ is a refinement monoid. By what we have just shown, it suffices to prove this result for $M(\Etil,\Ctil)$. Thus:
\end{construction}

\begin{assumption*} Until Proposition \ref{refinement}, we assume that $S=C$.
\end{assumption*}

With the above simplification, $F$ is the free abelian monoid on $E^0$, and $\bfrho(X)= \bfr(X)$ for $X\in C$.

We will need the following relation $\rightarrow_1$, a more elementary version of $\rightsquigarrow_1$, which provides an alternative way to build $\rightsquigarrow$.

\begin{definition}
\label{binary}
For nonzero $\alpha,\beta\in F$, write $\alpha \rightarrow_1 \beta$ to denote the following situation:
\begin{enumerate}
\item[] $\alpha= \sum_{i=1}^n v_i$ and $\beta= \bfr(X)+ \sum_{i=2}^n v_i$ for some $v_i\in E^0$ and some $X\in C_{v_1}$, where $v_1\notin \Si(E)$.
\end{enumerate}
Observe that $\sim$ is the congruence on $F$ generated by $\rightarrow_1$.

Let $\rightarrow $ be the transitive
and reflexive closure of $\rightarrow _1$ on $F$,
that is, $\alpha\rightarrow \beta$ if and only if there is a finite
string $\alpha =\alpha _0\rightarrow _1 \alpha _1\rightarrow _1
\cdots \rightarrow _1 \alpha _t=\beta$, in which case we write
$\alpha \rightarrow_t \beta$. In particular, $t=0$ is allowed, so
that $\alpha \rightarrow \alpha$ for all $\alpha\in F$ (including $\alpha=0$).

Observe that $\alpha\rightarrow_1 \beta$ implies $\alpha \rightsquigarrow_1 \beta$, so that
$\alpha\rightarrow_n \beta$ implies $\alpha \rightsquigarrow_n \beta$ for any $n>0$. On the other hand, the basic instance of $\alpha\rightsquigarrow_1 \beta$ given in Definition \ref{defsquigarrow} can clearly be achieved by $k$ iterations of $\rightarrow_1$. Consequently, $\alpha \rightsquigarrow_n \beta$ implies $\alpha \rightarrow_t \beta$ for some $t\ge n$. Therefore $\rightarrow$ coincides with $\rightsquigarrow$.

Note that $\rightarrow_1$ is partially compatible with sums: $\alpha \rightarrow_1 \beta$ implies $\alpha+\gamma \rightarrow_1 \beta+\gamma$ for any $\gamma\in F$. Consequently, if $\alpha \rightarrow_s \beta$ and $\alpha' \rightarrow_t \beta'$, then $\alpha+\alpha' \rightarrow_{s+t} \beta+\beta'$. This of course yields that $\rightarrow$ is compatible with sums: if $\alpha \rightarrow \beta$ and $\alpha' \rightarrow \beta'$, then $\alpha+\alpha' \rightarrow \beta+\beta'$.
\end{definition}

\begin{lemma} \label{simfromstrings}
Let $\alpha,\beta \in F$. Then $\alpha\sim\beta$ if and only if there is a finite
string $\alpha =\alpha _0,\alpha_1, \dots ,\alpha _n=\beta$, such
that, for each $i=0,\dots ,n-1$, either $\alpha _i\rightarrow _1
\alpha _{i+1}$ or $\alpha_{i+1}\rightarrow_1 \alpha _i$.
\end{lemma}

The number
$n$ above will be called the {\it length} of the string. Strings of length zero are allowed.

\begin{proof} Define a relation $\approx$ on $F$ by the given condition on strings. Namely, $\alpha\approx\beta$ if and only if there is a finite
string $\alpha =\alpha _0,\alpha_1, \dots ,\alpha _n=\beta$, such
that, for each $i=0,\dots ,n-1$, either $\alpha _i\rightarrow _1
\alpha _{i+1}$ or $\alpha_{i+1}\rightarrow_1 \alpha _i$. It is clear that $\approx$ is an equivalence relation, and it follows from the partial compatibility of $\rightarrow_1$ with sums that $\approx$ is a monoid congruence. As $\rightarrow_1$ implies $\approx$ and $\rightarrow_1$ generates $\sim$, it follows that $\sim\; \subseteq\; \approx$. The reverse inclusion holds by construction of $\approx$. Therefore $\approx$ coincides with $\sim$, proving the lemma.
\end{proof}

The {\em support} of an element $\gamma$ in $F$, denoted
$\mbox{supp}(\gamma)\subseteq E^0$, is the set of basis elements
appearing in the canonical expression of $\gamma$.

\begin{lemma}
\label{division} Assume that $\alpha,\alpha_1,\alpha_2,\beta\in F$
with $\alpha =\alpha _1+\alpha _2$ and $\alpha \rightarrow_n \beta
$ for some $n$. Then $\beta $ can be written as $\beta =\beta _1+\beta _2$ with
$\alpha _1\rightarrow_s \beta _1$ and $\alpha_2\rightarrow_t \beta
_2$, where $s,t\le n$.
\end{lemma}

\begin{proof}
By induction, it is enough to show the result in the case where
$\alpha \rightarrow _1\beta$. If $\alpha \rightarrow _1 \beta$, then
there is an element $v$ in the support of $\alpha$ such that $\beta
= (\alpha -v)+\bfr(X)$ for some $X\in C_v$. The element $v$
belongs either to the support of $\alpha _1$ or to the support of
$\alpha _2$. Assume, for instance, that $v$ belongs to
the support of $\alpha _1$. Then we set $\beta _1=(\alpha
_1-v)+\bfr(X)$ and $\beta _2=\alpha _2$.
\end{proof}

Note that the elements $ \beta _1$ and $\beta_2$ in Lemma
\ref{division} are not uniquely determined by $\alpha _1$ and
$\alpha _2$ in general, because the element $v\in E^0$ considered in
the proof could belong to {\it both} the support of $\alpha _1$ and
the support of $\alpha _2$.

\begin{lemma}\label{confluence}
Assume $(*)$, and let $\alpha,\beta,\gamma\in F$. If
$\alpha\rightarrow\beta$ and $\alpha\rightarrow\gamma$,
there exists $\delta\in F$ such that $\beta\rightarrow\delta$
and $\gamma\rightarrow\delta$.
\end{lemma}

\begin{proof} Since $\rightarrow$ coincides with $\rightsquigarrow$, it suffices to prove the corresponding confluence for $\rightsquigarrow$: If
$\alpha\rightsquigarrow\beta$ and $\alpha\rightsquigarrow\gamma$,
there exists $\delta\in F$ such that $\beta\rightsquigarrow\delta$
and $\gamma\rightsquigarrow\delta$.

If $\alpha=0$, then $\beta=\gamma=0$, and we take $\delta=0$. Hence, we may assume that $\alpha\ne0$,
in which case $\beta,\gamma\ne0$.

{\bf Claim 1}: If $\alpha\rightsquigarrow_1\beta$ and
$\alpha\rightsquigarrow_1\gamma$, there exists $\delta \in F$ such that
$\beta\rightsquigarrow_1\delta$ and
$\gamma\rightsquigarrow_1\delta$.

In this situation, we can write $\alpha = \sum_{i=1}^n u_i$ for some $u_i\in E^0$ and
\begin{align*}
 \beta &= \sum_{i=1}^k \bfr(X_i)+ \sum_{i=k+1}^l \bfr(X_i)+ \sum_{i=l+1}^m u_i+ \sum_{i=m+1}^n u_i \\
 \gamma &= \sum_{i=1}^k \bfr(Y_i)+ \sum_{i=k+1}^l u_i+ \sum_{i=l+1}^m \bfr(Y_i)+ \sum_{i=m+1}^n u_i
\end{align*}
with $0\le k\le l\le m \le n$ and all $X_i,Y_i \in C_{u_i}$. By $(*)$, there exist $\delta_1,\dots,
\delta_k\in F$ such that $\bfr(X_i) \rightsquigarrow_1 \delta_i$ and
$\bfr(Y_i) \rightsquigarrow_1 \delta_i$ for all $i=1,\dots,k$. Set
$$\delta=
\sum_{i=1}^k \delta_i+ \sum_{i=k+1}^l \bfr(X_i)+ \sum_{i=l+1}^m \bfr(Y_i)+ \sum_{i=m+1}^n u_i \,.$$
Then
$\beta\rightsquigarrow_1\delta$ and
$\gamma\rightsquigarrow_1\delta$. This verifies Claim 1.

{\bf Claim 2}: If $\alpha\rightsquigarrow_1\beta$ and
$\alpha\rightsquigarrow_n\gamma$, there exists $\delta \in F$ such that
$\beta\rightsquigarrow\delta$ and $\gamma\rightsquigarrow\delta$.

If $n=1$, the claim follows from Claim 1.
Now assume that $n>1$, and write $\alpha\rightsquigarrow_1 \alpha'
\rightsquigarrow_{n-1} \gamma$, for some $\alpha' \in F$. By Claim 1,
there is some $\delta' \in F$ such that $\beta\rightsquigarrow_1\delta'$
and $\alpha'\rightsquigarrow_1\delta'$. By induction on $n$, there
exists $\delta \in F$ such that $\delta'\rightsquigarrow\delta$ and
$\gamma\rightsquigarrow\delta$. Since then
$\beta\rightsquigarrow\delta$, Claim 2 is proved.

{\bf Claim 3}: If $\alpha\rightsquigarrow\beta$ and
$\alpha\rightsquigarrow_n\gamma$, there exists $\delta \in F$ such that
$\beta\rightsquigarrow\delta$ and $\gamma\rightsquigarrow\delta$.

The case $n=1$ holds by
Claim 2. Now assume that $n>1$, and write $\alpha\rightsquigarrow_1
\alpha' \rightsquigarrow_{n-1} \gamma$, for some $\alpha' \in F$. By Claim
2, there is some $\delta' \in F$ such that $\beta\rightsquigarrow\delta'$
and $\alpha'\rightsquigarrow\delta'$. By induction on $n$, there
exists $\delta \in F$ such that $\delta'\rightsquigarrow\delta$ and
$\gamma\rightsquigarrow\delta$. Since then
$\beta\rightsquigarrow\delta$, Claim 3 is proved.
\end{proof}

\begin{lemma}
\label{simfromconfluence} Assume $(*)$, and let $\alpha,\beta\in F$. Then
$\alpha \sim \beta$ if and only if there is some $\gamma\in F$ such
that $\alpha \rightarrow \gamma$ and $\beta \rightarrow \gamma$.
\end{lemma}

\begin{proof} This is clear if $\alpha$ or $\beta$ is zero, so assume $\alpha,\beta\ne 0$. If there exists $\gamma\in F$ such
that $\alpha \rightarrow \gamma$ and $\beta \rightarrow \gamma$, it is clear from Lemma \ref{simfromstrings} that $\alpha \sim \beta$.

Conversely, assume that $\alpha \sim \beta$. Then there exists a finite string
$\alpha =\alpha _0,\alpha_1, \dots ,\alpha _n=\beta$, such that, for
each $i=0,\dots ,n-1$, either $\alpha _i\rightarrow _1 \alpha
_{i+1}$ or $\alpha_{i+1}\rightarrow_1 \alpha _i$. We proceed by
induction on $n$. If $n=0$, then $\alpha=\beta$ and there is nothing
to prove. Assume that $n>0$ and the result is true for strings of length $n-1$, and
let $\alpha =\alpha_0,\alpha_1, \dots ,\alpha_n=\beta$ be a string
of length $n$. By the induction hypothesis, there is some
$\lambda\in F$ such that $\alpha\rightarrow \lambda$ and
$\alpha_{n-1}\rightarrow \lambda$. Now there are two cases to
consider. If $\beta\rightarrow_1 \alpha_{n-1}$, then $\beta
\rightarrow \lambda$ and we are done. Assume that
$\alpha_{n-1}\rightarrow _1 \beta$. Then by Lemma
\ref{confluence}, there is some $\gamma\in F $ such that
$\beta \rightarrow \gamma $ and $\lambda \rightarrow \gamma $.
Since $\alpha \rightarrow \lambda \rightarrow \gamma$, we get
$\alpha \rightarrow \gamma$, and so the result is proven.
\end{proof}

We are now ready to show the refinement property of $M(E,C,S)$ under
the current finiteness hypothesis.

\begin{proposition}
\label{refinement} Let $(E,C)$ be a separated graph such that all the sets in $C$ are finite, and let $S\subseteq C$. If Assumption $(*)$ holds, then the monoid $M(E,C,S)$ is a refinement monoid.
\end{proposition}

\begin{proof} Because of Construction \ref{removeS}, it suffices to prove that $M(E,C)$ has refinement, i.e., we may assume that $S=C$.

Let $\alpha=\alpha _1+\alpha_2\sim\beta =\beta _1+\beta _2$, with
$\alpha_1,\alpha_2,\beta _1,\beta _2\in F$. By Lemma
\ref{simfromconfluence}, there is some $\gamma\in F$ such that $\alpha
\rightarrow \gamma$ and $\beta \rightarrow \gamma$. By Lemma
\ref{division}, we can write $\gamma =\alpha_1 '+\alpha _2'=\beta
_1'+\beta _2'$, with $\alpha_i\rightarrow \alpha _i'$ and $\beta
_i\rightarrow \beta _i'$ for $i=1,2$. Since $F$ is a free abelian
monoid, $F$ has the refinement property and so there are
decompositions $\alpha _i'=\gamma _{i1}+\gamma _{i2}$ for $i=1,2$
such that $\beta _j'=\gamma _{1j}+\gamma _{2j}$ for $j=1,2$. Hence, $\alpha_i\sim \gamma_{i1}+\gamma_{i2}$ for $i=1,2$ and $\beta_j\sim \gamma_{1j}+\gamma_{2j}$ for $j=1,2$. The
result follows.
\end{proof}

We now return to the general case, where $(E,C,S)$ is an arbitrary object of $\SSGr$. Some of the above notation must then be modified.

\begin{definition} \label{defcalZ}
For $v\in E^0$, let $\Cvfin := C_v\cap \Cfin$ and $\Cvinf := C_v\setminus \Cfin$, and set
$$\calZ_v= \calZ_{E,C,v} := \Cvfin \sqcup \{ \text{nonempty finite subsets of members of\ } \Cvinf \}.$$
Then set $\calZ= \calZ_{E,C} := \bigsqcup_{v\in E^0} \calZ_v$.

Next, let $Q^0 := \{q'_Z \mid Z\in \calZ\setminus S\}$, and let $F$ be the free abelian monoid on $E^0\sqcup Q^0$. In $F$, set
$$\bfrho(Z) := \begin{cases} \bfr(Z) &(Z\in S)\\ \bfr(Z)+q'_Z &(Z\in \calZ\setminus S) \end{cases}$$
for $Z\in \calZ$, and set
$$\bfsig(Z_1,Z_2) := \bfr(Z_2\setminus Z_1)+q'_{Z_2} \qquad (\emptyset\ne Z_1 \subsetneq Z_2 \in \calZ\setminus \Cfin)$$
for nonempty sets $Z_1 \subsetneq Z_2$ from $\calZ$. (Note that if $Z_1,Z_2\in \calZ$ with $Z_1\subsetneq Z_2$, then necessarily $Z_1,Z_2\notin \Cfin$.) We identify $M(E,C,S)$ with $F/{\sim}$, where $\sim$ is the congruence on $F$ generated by the relations
$$v \sim \bfrho(Z) \qquad\qquad \text{and} \qquad\qquad q'_{Z_1} \sim \bfsig(Z_1,Z_2)$$
for $v\in E^0$, $Z\in \calZ_v$, and $\emptyset\ne Z_1 \subsetneq Z_2 \in \calZ\setminus \Cfin$.
\end{definition}

\begin{definition} \label{defRightarrow}
[General case] For $\alpha,\beta \in F$, write $\alpha\rightsquigarrow_1\beta$ to denote the following situation:
\begin{enumerate}
\item[] $\alpha= \sum_{i=1}^k u_i+ \sum_{i=k+1}^m u_i$ and $\beta= \sum_{i=1}^k r_i+ \sum_{i=k+1}^m u_i$ for some $u_i \in E^0\sqcup Q^0$ such that for each $i=1,\dots,k$, one of the following holds:
\begin{enumerate}
\item $u_i\in E^0\setminus \Si(E)$ and $r_i= \bfrho(Z_i)$ for some $Z_i\in \calZ_{u_i}$;
\item $u_i= q'_{Z_i} \in Q^0$ for some $Z_i\in \calZ\setminus \Cfin$, and $r_i= \bfsig(Z_i,Z'_i)$ for some $Z'_i\in \calZ$ with $Z_i \subsetneq Z'_i$.
\end{enumerate}
\end{enumerate}
\end{definition}

\begin{definition} \label{defstarstar}
[General case] {\bf Assumption $(*)$}: For all $v\in E^0$ and $X,Y\in \calZ_v$ such that $X\cup Y \notin \calZ_v$, there exists $\gamma\in F$ such that $\bfrho(X) \rightsquigarrow_1 \gamma$ and $\bfrho(Y) \rightsquigarrow_1 \gamma$.
\end{definition}

Note that there is no ambiguity in the two definitions of $\rightsquigarrow_1$ and $(*)$.

Two conditions similar to Assumption $(*)$ hold automatically, as follows.

\begin{lemma} \label{simstar}

{\rm (a)} If $v\in E^0$ and $X,Y\in \calZ_v$ with $X\cup Y \in \calZ_v$, there exists $\gamma\in F$ such that $\bfrho(X) \rightsquigarrow_1 \gamma$ and $\bfrho(Y) \rightsquigarrow_1 \gamma$.

{\rm (b)} If $Z_1,Z_2,Z_3\in \calZ$ with $Z_1\subsetneq Z_2$ and $Z_1\subsetneq Z_3$, there exists $\gamma\in F$ such that $\bfsig(Z_1,Z_2) \rightsquigarrow_1 \gamma$ and $\bfsig(Z_1,Z_3) \rightsquigarrow_1 \gamma$.
\end{lemma}

\begin{proof} (a) We may assume that $X\ne Y$. Then $Z:= X\cup Y$ lies in $\calZ_v \setminus \Cfin$, and we observe that $\bfrho(X) \rightsquigarrow_1 \bfr(X)+ \bfsig(X,Z) = \bfrho(Z)$ and similarly $\bfrho(Y) \rightsquigarrow_1 \bfrho(Z)$.

(b) Under the given hypotheses, $Z_2,Z_3\notin \Cfin$. Since they both contain $Z_1$, they are not disjoint, and so they must both be subsets of some infinite member of $C$. Hence, $Z:= Z_2\cup Z_3$ lies in $\calZ \setminus \Cfin$, and we conclude that $\bfsig(Z_1,Z_j) \rightsquigarrow_1 \bfsig(Z_1,Z)$ for $j=1,2$.
\end{proof}

\begin{construction} \label{makefinite}
As in Construction \ref{removeS}, we build a new graph $\Etil$ with appropriate finiteness conditions. First, let $\Etil^0$ consist of $E^0$ together with new vertices $w_Z$ for $Z\in \calZ\setminus S$. Then let $\Etil^1$ be the collection of the following four types of edges:
\begin{enumerate}
\item[] $\varepsilon_{Z,e}: v\rightarrow r(e)$, for $v\in E^0$ and $e\in Z\in \calZ_v$;
\item[] $f_Z: v\rightarrow w_Z$, for $v\in E^0$ and $Z\in \calZ_v\setminus S$;
\item[] $g_{Z,Z',e}: w_{Z}\rightarrow r(e)$, for $\emptyset\ne Z \subsetneq Z'\in \calZ\setminus\Cfin$ and $e\in Z'\setminus Z$;
\item[] $h_{Z,Z'}: w_{Z}\rightarrow w_{Z'}$, for $\emptyset\ne Z \subsetneq Z'\in \calZ\setminus\Cfin$.
\end{enumerate}
For each edge $e:v\rightarrow w$ in $E$, there is a unique smallest set $Z_e\in \calZ_v$ containing $e$: namely, the unique set in $C_v$ containing $e$, if this set is finite, or $\{e\}$ otherwise. If we identify $e$ with the edge $\varepsilon_{Z_e,e}$ in $\Etil$ for each $e\in E^1$, then $E$ becomes a subgraph of $\Etil$.

Define partitions $D_w$ of $s^{-1}_{\Etil}(w)$ for $w\in \Etil^0$ as follows. Set
$$X_Z := \begin{cases} \{\varepsilon_{Z,e} \mid e\in Z\} &(Z\in S)\\ \{\varepsilon_{Z,e}\mid e\in Z\} \sqcup \{f_Z\} &(Z\in \calZ\setminus S) \end{cases}$$
for $Z\in \calZ$, and define $D_v := \{ X_Z \mid Z\in \calZ_v\}$ for $v\in E^0$. Then set
$$Y_{Z,Z'} := \{ g_{Z,Z',e} \mid e\in Z'\setminus Z\}\sqcup \{h_{Z,Z'}\} \qquad (\emptyset\ne Z \subsetneq Z' \in \calZ\setminus\Cfin),$$
and define $D_{w_{Z}} := \{ Y_{Z,Z'} \mid Z \subsetneq Z'\in \calZ\}$ for $Z\in \calZ\setminus\Cfin$. For $Z\in \Cfin\setminus S$, the vertex $w_Z$ is a sink in $\Etil$, and we set $D_{w_Z}= \emptyset$. We have now built a separated graph $(\Etil,D)$, and all the sets in $D$ are finite. Hence, $(\Etil,D,D)$ is an object in $\SSGr$.

Comparing presentations, we see that there is a monoid isomorphism
$$M(E,C,S) \longrightarrow M(\Etil,D,D)$$
sending $v\mapsto v$ for $v\in E^0$ and $q'_Z\mapsto w_Z$ for $Z\in \calZ\setminus S$. There is an analogous isomorphism from $F$ onto the free abelian monoid $\Ftil$ on $\Etil^0$, and with its help, together with Lemma \ref{simstar}, we see that if Assumption $(*)$ holds for $(E,C,S)$, then it holds for $(\Etil,D,D)$. Therefore we obtain the following theorem from Proposition \ref{refinement}.
\end{construction}

\begin{theorem} \label{generalrefinement}
Let $(E,C)$ be a separated graph and $S\subseteq \Cfin$. If Assumption $(*)$ holds, then the monoid $M(E,C,S)$ is a refinement monoid.
\end{theorem}

Finally, let us apply Theorem \ref{generalrefinement} to the case of a non-separated graph $E$. The corresponding monoid $M(E)$ has only been defined in the literature in the row-finite case \cite{AMP}; for the general case, we follow the pattern of Definition \ref{defgraphmon}. Specifically, set $C= \bigsqcup_{v\in E^0} C_v$ and $S= \Cfin$ where $C_v= \{s^{-1}(v)\}$ for
$v\in E^0\setminus \Si(E)$ and $C_v= \emptyset$ for $v\in \Si(E)$, and define $M(E) := M(E,C,S)$. With these choices of $C$ and $S$, the collections $\calZ_v$ are closed under finite unions, and so Assumption $(*)$ holds vacuously. Therefore Theorem \ref{generalrefinement} has the following immediate consequence.

\begin{corollary} \label{M(E)refinement}
For any graph $E$, the monoid $M(E)$ is a refinement monoid.
\end{corollary}

%%%%%%%%%%%%%%%%%%%%%%%%%%%%%%%%%%%%%
\section{Ideal lattices}
\label{sect:ideallattice}

The ideal lattices of Cohn-Leavitt algebras $CL_K(E,C,S)$ are typically much more complicated than those of Leavitt path algebras (of non-separated graphs), as is already clear from Proposition \ref{1-vertex}. In particular, the lattice of graded ideals of $CL_K(E,C,S)$ may contain many ideals not generated by vertices of $E$. We can, however, identify a sublattice of ideals of that type, namely the lattice of those ideals generated by idempotents. We prove that, in analogy with \cite[Theorem 5.7]{Tomf07}, this lattice is isomorphic to a certain lattice $\ACS$ of \emph{admissible pairs} $(H,G)$ where $H\subseteq E^0$ and $G\subseteq C$ (see Definition \ref{defACS} for the precise definition). These lattices, in turn, are isomorphic to the lattice of all order-ideals in the monoid $M(E,C,S)$. We can thus say that this lattice consists of ``all the ideals of $CL_K(E,C,S)$ which can be detected by K-theory.'' Further, we can derive graph-theoretic criteria for $M(E,C,S)$ to be simple (see Section \ref{sec:simplicity}).

See Definition \ref{trace} for the lattice of trace ideals of a ring $A$, denoted $\Tr(A)$, and for the lattice of order-ideals of a monoid $M$, denoted  $\calL(M)$. We shall use the notation $\langle B\rangle$ to denote the order-ideal generated by a subset $B$ of $M$.

Throughout this section, let $(E,C,S)$ be a fixed object of $\SSGr$, and recall the class $\calZ= \calZ_{E,C}$ from Definition \ref{defcalZ}.
We will need the following idempotents $q_Z\in CL_K(E,C,S)$.

\begin{definition} \label{defqZ}
For any $v\in E^0$ and any nonempty finite subset $Z$ of any $X\in C_v$, set
$$q_Z := v- \sum_{e\in Z} ee^*.$$
Of course, $q_Z=0$ when $Z\in S$. The relations (SCK1) for $X$ imply that the elements $ee^*$ for $e\in X$ are pairwise orthogonal idempotents, and $ee^*\le v$ for such $e$. Hence, $q_Z$ is an idempotent, $q_Z\le v$, and $q_Z \perp ee^*$ for all $e\in Z$. In fact, $q_Ze= e^*q_Z =0$ for all $e\in Z$.

If $Z \notin S$, then we can ensure that $v$ and the paths $ee^*$ for $e\in Z$ are all reduced with respect to $S$ in the sense of Definition \ref{def:redform1}, by simply choosing $e_X\in X\setminus Z$ in case $X\in S$. By Theorem \ref{basisCSE}, $v$ and the $ee^*$ for $e\in Z$ are $K$-linearly independent. Thus, $q_Z\ne 0$ when $Z\notin S$.
\end{definition}

\begin{proposition}  \label{LVCLisoTrCL}
Let $A := CL_K(E,C,S)$. Then the trace ideals of $A$ are precisely the idem\-po\-tent-generated ideals, and the lattice isomorphism $\Phi: \calL(\mon{A}) \rightarrow \Tr(A)$ of Proposition {\rm\ref{LVAisoTrA}} can be expressed as
$$\Phi(I)= \langle\; \text{idempotents\ } e\in A \mid [e]\in I \;\rangle.$$
\end{proposition}

\begin{proof} Consider $J\in \Tr(A)$, and let $J'$ be the ideal of $A$ generated by the idempotents in $J$. In order to prove that $J=J'$, it suffices to show that $\Psi(J)= \Psi(J')$, by Proposition \ref{LVAisoTrA}.  By Theorem \ref{computVMECS}, there is a  monoid isomorphism $\Gamma: M(E,C,S) \rightarrow \calV(A)$ sending $v\mapsto [v]$ for $v\in E^0$ and $q'_Z \mapsto [q_Z]$ for $Z\in \calZ \setminus S$. Taking account of Definition \ref{defcalZ}, it follows that $\calV(A)$ is generated by
$$\{ [v]\mid v\in E^0\} \cup \{ [q_Z] \mid Z\in \calZ\setminus S\}.$$
Consequently, any $x\in \Psi(J)$ can be written $x= \sum_{i=1}^l [v_i] + \sum_{j=1}^m [q_{Z_j}]$ for some $v_i\in E^0$ and $Z_j\in \calZ\setminus S$. Since $\Psi(J)$ is an order-ideal of $\mon{A}$, we have $[v_i],[q_{Z_j}] \in \Psi(J)$ for all $i$, $j$, and hence $v_i,q_{Z_j} \in \Phi\Psi(J)=J$ for all $i$, $j$. Then $v_i,q_{Z_j} \in J'$ for all $i$, $j$, and consequently $x\in \Psi(J')$. Therefore $\Psi(J)= \Psi(J')$, as required.

Now let $I\in \calL(\mon{A})$. By what we have proved so far, the trace ideal $\Phi(I)$ is generated by the idempotents it contains. If $e\in A$ is idempotent, then $e\in \Phi(I)$ if and only if $[e]\in \Psi\Phi(I)= I$. Therefore $\Phi(I)$ is generated by the idempotents $e\in A$ such that $[e]\in I$. \end{proof}

Recall that for a non-separated row-finite graph $E$, the order-ideals of the monoid $M(E)$ correspond to hereditary saturated subsets of $E^0$ \cite[Theorem 5.3]{AMP}. In the case of $M(E,C,S)$, we require a version of saturation relative to the choices of $C$ and $S$, as defined below.

\begin{definition} \label{defCSsat}
Recall the relation $\ge$ defined on $E^0$ by setting $v\geq w$ if and only if there
is a path $\mu$ in $E$ with $s(\mu)=v$ and $r(\mu)=w$. A subset $H$
of $E^0$ is called \emph{hereditary} if $v \geq w$ and $v\in H$ always imply
$w\in H$. The set $H$ is called \emph{saturated} if $r(s^{-1}(v)) \subseteq H$ implies $v\in H$ for any $v\in E^0$ which is not a sink or an infinite emitter.

A subset $H\subseteq E^0$  is \emph{$(C,S)$-saturated} provided the following condition holds:
\begin{enumerate}
\item[] If $X\in S\cap C_v$ for some $v\in E^0$ and $r(X) \subseteq H$, then $v\in H$.
\end{enumerate}
In case $S=\Cfin$, we drop the reference to $S$ and use the terminology \emph{$C$-saturated} in place of $(C,S)$-saturated.
\end{definition}

In working with the monoid $M(E,C,S)$, we continue to use the presentation given in Definition \ref{defcalZ}. Thus, $M(E,C,S)$ is generated by $E^0\sqcup Q^0$ where $Q^0 = \{ q'_Z\mid Z\in \calZ\setminus S\}$ and $\calZ= \calZ_{E,C}$.

\begin{lemma} \label{HfromMideal}
If $I$ is an order-ideal of $M(E,C,S)$, then the set $H:= \{v\in E^0
\mid v\in I\} $ is hereditary and $(C,S)$-saturated.
\end{lemma}

\begin{proof} Set $M:= M(E,C,S)$.

Consider an edge $e:v\rightarrow w$ from $E^1$ such that $v\in H$.
Then $e\in Z$ for some $Z\in \calZ_v$, and $w\le \bfr(Z)\le
\bfrho(Z)= v$ in $M$. Since $v\in I$ and $I$ is hereditary, we
obtain $w\in I$ and so $w\in H$. This verifies that $H$ is
hereditary.

Next, consider $X\in S\cap C_v$ for some $v\in E^0$ such that
$r(X) \subseteq H$. Then $v= \bfrho(X)= \bfr(X) \in I$ and
$v\in H$, showing that $H$ is $(C,S)$-saturated.
\end{proof}

We need to pair hereditary, $(C,S)$-saturated subsets of $E^0$ with certain subsets of $C$. The resulting pairs of sets are analogous to the \emph{admissible pairs} with which Tomforde parametrized the graded ideals of $L_K(E)$ in \cite[Theorem 5.7]{Tomf07}.

\begin{definition}  \label{defACS}
Let $H$ be a hereditary, $(C,S)$-saturated subset of $E^0$. For any subset $X\subseteq E^1$, define
$$X/H := X\cap r^{-1}(E^0\setminus H).$$
(This notation follows that for the quotient graph $E/H$, which has vertex set $E^0\setminus H$ and edge set $r_E^{-1}(E^0\setminus H)$.) Set
$$\calG(H) := \{ X\in C\setminus S \mid X/H \text{\ is nonempty and finite}\, \}.$$
Any $X\in \calG(H)$ belongs to $C_v$ for some $v\in E^0$, and there is some $e\in X$ with $r(e) \notin H$, so $v=s(e) \notin H$ because $H$ is hereditary. Thus, $\calG(H) \subseteq \bigsqcup_{v\in E^0\setminus H} C_v$. We record this in the form $\calG(H)\cap C[H]= \emptyset$, where
$$C[H] := \bigsqcup_{v\in H} C_v\,.$$
Note that $X/H\in \calZ\setminus S$ for all $X\in \calG(H)\cap \Cinf$, and so $q'_{X/H} \in Q^0$ for all $X\in \calG(H)\cap \Cinf$.

Next, let $\ACS$ denote the set of all pairs $(H,G)$ where $H$ is a hereditary, $(C,S)$-saturated subset of $E^0$ and $G$ is a subset of $\calG(H)$. Define a relation $\le$ on $\ACS$ as follows:
$$(H_1,G_1) \le (H_2,G_2) \qquad\iff\qquad H_1\subseteq H_2 \quad\text{and}\quad G_1 \subseteq G_2\sqcup C[H_2].$$
Observe that $(\ACS,\le)$ is a partially ordered set, with minimum element $(\emptyset,\emptyset)$ and maximum element $(E^0,\emptyset)$.

We claim that any nonempty family $\bigl( (H_i,G_i) \bigr)_{i\in I}$ of elements of $\ACS$ has an infimum in $\ACS$, namely the pair
$$(H,G) := \biggl(\, \bigcap_{i\in I} H_i,\, \calG\bigl(\, \bigcap_{i\in I} H_i \bigr)\cap \bigcap_{i\in I} (G_i\sqcup C[H_i]) \biggr).$$
It is clear that $(H,G)\in \ACS$ and $(H,G)\le (H_i,G_i)$ for all $i$. If $(H',G')\in \ACS$ and $(H',G') \le (H_i,G_i)$ for all $i$, then clearly $H'\subseteq H$. Consider any $X\in G'\setminus C[H]$. Then $X\in C_v$ for some $v\in E^0\setminus H$, so $v\notin H_j$ for some $j\in I$, and $X\notin C[H_j]$. Since $(H',G')\le (H_j,G_j)$, it follows that $X\in G_j$. Hence, $X/H_j$ is nonempty, whence $X/H$ is nonempty. On the other hand, $X\in G'$ implies that $X/H'$ is finite, whence $X/H$ is finite. Thus, $X\in \calG(H)$. We also have $X\in G_i\sqcup C[H_i]$ for all $i$, because $(H',G') \le (H_i,G_i)$ for all $i$, and consequently $X\in G$. This shows that $G' \subseteq G\sqcup C[H]$, proving that $(H',G') \le (H,G)$, and verifying that $(H,G)$ is indeed the infimum of the family $\bigl( (H_i,G_i) \bigr)_{i\in I}$ in $\ACS$.

Therefore, $\ACS$ is a complete lattice.
\end{definition}

\begin{definition}  \label{defpsiphi}
If $I$ is any order-ideal of $M(E,C,S)$, set $\psi(I):= (H,G)$ where
\begin{align*}
H &:= \{ v\in E^0 \mid v\in I \} \\
G &:= \{X\in \calG (H)\cap \Cfin\mid q'_{X}\in I \}\sqcup \{X\in \calG (H)\cap \Cinf \mid q'_{X/H}\in I\}.
\end{align*}
As we have seen in Lemma \ref{HfromMideal}, $H$ is a hereditary, $(C,S)$-saturated subset of $E^0$, and so $(H,G) \in \ACS$.

Conversely, for any $(H,G) \in \ACS$, let $I(H,G)$ denote the submonoid of $M(E,C,S)$ generated by the set
$$H\sqcup
\{q'_{X}\mid X\in \Cfin \cap G\}\sqcup \{ q'_{X/H}\mid X\in
\Cinf\cap G\}.$$
\end{definition}

\begin{lemma}  \label{HGfromMideal}
If $I$ is any order-ideal of $M(E,C,S)$ and $\psi(I)= (H,G)$, then $I= \langle I(H,G) \rangle$.
\end{lemma}

\begin{proof} It is clear that $I(H,G) \subseteq I$, whence $\langle I(H,G)\rangle \subseteq I$. Now consider a nonzero element $x\in I$. Then  $x= \sum_i v_i+ \sum_j q'_{Z_j}$ for some $v_i\in E^0$ and some $Z_j\in \calZ\setminus S$. Each $v_i,q'_{Z_j} \in I$ because $I$ is an order-ideal, and so to prove that $x\in \langle I(H,G) \rangle$, it is enough to show that $v,q'_Z\in \langle I(H,G) \rangle$ for all $v\in E^0$ and $Z\in \calZ\setminus S$ such that $v,q'_Z\in I$.

If $v\in E^0\cap I$, then $v\in H$ by definition of $H$, whence $v\in I(H,G)$.

Now let $Z\in \calZ\setminus S$ such that $q'_Z\in I$. Then $Z\subseteq X$ for some $X\in C_v\setminus S$ and $v\in E^0$. If $r(Z)\subseteq H$, then $\bfr(Z)\in I$ and so
$$v= \bfrho(Z)= \bfr(Z)+q'_Z\in I,$$
whence $v\in H\subseteq I(H,G)$. Since $q'_Z\le v$, it follows that $q'_Z \in \langle I(H,G) \rangle$ in this case.

Assume now that $r(Z) \not\subseteq H$, so that $Z/H$ and $X/H$ are nonempty. If $X$ is finite, then $X=Z\in \calG(H)$ and $q'_X=q'_Z \in I$, so $X\in G$ by definition of $G$ and $q'_Z= q'_X \in I(H,G)$. If $X$ is infinite and $X/H$ is finite, then $X\in \calG(H)$ and
$$\bfr\bigl( (X/H)\setminus Z \bigr)= \bfr\bigl( [(X/H)\cup Z]\setminus Z \bigr) \le \bfsig\bigl( Z, (X/H)\cup Z \bigr)= q'_Z \in I.$$
In this case, it follows that $\bfr\bigl( (X/H)\setminus Z \bigr) \in I$ and so $r\bigl( (X/H)\setminus Z \bigr) \subseteq H$, whence $X/H \subseteq Z$. Since $r\bigl( Z\setminus (X/H) \bigr) \subseteq H$, we get
$$q'_{X/H}= \bfsig(X/H,Z)= \bfr\bigl( Z\setminus (X/H) \bigr)+ q'_Z \in I,$$
so that $X\in G$ by definition of $G$, and then $q'_{X/H} \in I(H,G)$. Since $q'_Z \le q'_{X/H}$, we get $q'_Z\in \langle I(H,G) \rangle$.

Finally, assume that $X$ and $X/H$ are both infinite. Then there exists $f\in (X/H)\setminus Z$, and we have
$$r(f)\le \bfsig(Z, Z\sqcup\{f\})= q'_Z \in I.$$
But this implies $r(f)\in I$ and so $r(f)\in H$, which is a contradiction to the assumption that $f\in X/H$. Thus, $q'_Z\in \langle I(H,G) \rangle$ in all cases, as required.
\end{proof}

We define the \emph{kernel} of a monoid homomorphism $\tau: M_1 \rightarrow M_2$ just as for abelian groups: $\ker\tau := \tau^{-1}(\{0\})$. This is always a submonoid of $M_1$, but if $M_2$ is conical, then $\ker\tau$ is an order-ideal of $M_1$.

\begin{construction}  \label{constrMtilpi}
Let $(H,G)\in \ACS$ and define the order-ideal $I:= \langle I(H,G) \rangle$ in $M:= M(E,C,S)$. We construct an object $(\Etil,\Ctil,\Stil)$ in $\SSGr$ and a homomorphism $\pi$ from $M$ to $\Mtil := M(\Etil,\Ctil,\Stil)$ such that $I= \ker\pi$.

Let $\Etil$ be the quotient graph $E/H$, that is, the subgraph of $E$ with
$$\Etil^0 = E^0\setminus H \qquad\quad \text{and} \qquad\quad \Etil^1= r_E^{-1}(E^0\setminus H) = E^1/H.$$
(Since $H$ is hereditary, $s(e)\notin H$ for all $e\in \Etil^1$.) For $v\in \Etil^0$, set
$$\Ctil_v := \{ X/H \mid X\in C_v \text{\ and\ } X/H \ne\emptyset \},$$
which is a partition of $s_{\Etil}^{-1}(v)$. Now set $\Ctil := \bigsqcup_{v\in \Etil^0} \Ctil_v$. This gives us a separated graph $(\Etil,\Ctil)$. Further, set
$$\Stil := \{ X/H \mid X\in S \text{\ and\ } X/H \ne\emptyset \} \sqcup \{ X/H \mid X\in G \}.$$
Clearly, $X/H \in \Ctilfin$ when $X\in S$ and $X/H \ne \emptyset$.
By definition of $\calG(H)$, the set $X/H$ is nonempty and finite
for any $X\in G$, whence $X/H \in \Ctilfin$. Thus, $\Stil \subseteq
\Ctilfin$, and therefore $(\Etil,\Ctil,\Stil)$ is an object in
$\SSGr$. Observe that, for $X\in C\setminus S$, we have
$$X\in \calG(H) \quad\iff\quad X/H \in \Ctilfin \,.$$

We next define elements $\vtil,\qtil_Z \in \Mtil$ for $v\in E^0$ and $Z\in \calZ\setminus S$. For $v\in E^0$, set
$$\vtil := \begin{cases} v &\text{if\ } v\notin H \\ 0 &\text{if\ } v\in H. \end{cases}$$
Next, assume that $Z\in \calZ\setminus S$; then $Z\subseteq X$ for some $X\in C_v\setminus S$ and $v\in E^0$. We distinguish several cases:
\begin{enumerate}
\item If $v\in H$, set $\qtil_Z :=0$.
\item If $v\notin H$ and $X\in G$, set
$\qtil_Z :=  \bfr\bigl( (X\setminus Z)/H \bigr)$.
(In case $X$ is finite, $X=Z$ and $\qtil_Z =0$.)
\item If $v\notin H$ and $X\notin G$, set
$\qtil_Z := \begin{cases} v &\text{if\ } r(Z) \subseteq H \\ q'_{Z/H} &\text{if\ } r(Z) \not\subseteq H \text{\ and\ } X\notin \calG(H) \\
q'_{X/H}+ \bfr\bigl( (X\setminus Z)/H \bigr)  &\text{if\ } r(Z) \not\subseteq H \text{\ and\ } X\in \calG(H). \end{cases}$
\end{enumerate}
Note that if $v\notin H$ and $X\in \calG(H)\setminus G$, then $\qtil_Z= q'_{X/H}+ \bfr\bigl( (X\setminus Z)/H \bigr)$ whether or not $r(Z) \subseteq H$.

{\bf Claim 1}: There is a homomorphism $\pi: M\rightarrow \Mtil$ sending $v\mapsto \vtil$ for all $v\in E^0$ and $q'_Z\mapsto \qtil_Z$ for all $Z\in \calZ\setminus S$. To see this, we need to verify that the elements $\vtil$ and $\qtil_Z$ satisfy the defining relations of the elements $v,q'_Z\in E^0\sqcup Q^0$. In order to write these relations compactly, it will be convenient to use the notation
$$\bfrtil(W) := \sum_{e\in W} \widetilde{r(e)} \in \Mtil$$
for subsets $W\subseteq E^1$. Note that if we view $W/H$ as a subset of $\Etil^1$, then $\bfrtil(W)= \bfr(W/H)$, since $\widetilde{r(e)} =0$ for $e\in W\setminus (W/H)$.

Suppose that $v\in E^0$ and $Z\in \calZ_v$. Then $Z\subseteq X$ for some $X\in C_v$. If $v\in H$, then $r(Z) \subseteq H$ because $H$ is hereditary, and we get
$$\vtil =0= \begin{cases} \bfrtil(Z) &\text{if\ } Z\in S \\ \bfrtil(Z)+ \qtil_Z &\text{if\ } Z\notin S. \end{cases}$$
If $v\notin H$ and $Z\in S$, we have $Z=X$ and $r(X) \not\subseteq H$ because $H$ is $(C,S)$-saturated, whence
$$\vtil= v= \bfrho(Z/H)= \bfr(Z/H)= \bfrtil(Z).$$
So, we may assume that $v\notin H$ and $Z\notin S$. Then also $X\notin S$.

In case (2), $X/H\in \Stil$ and
$$\vtil= v= \bfrho(X/H)= \bfr(X/H)= \qtil_Z+ \bfr(Z/H)= \bfrtil(Z)+ \qtil_Z \,.$$
In case (3) with $r(Z) \subseteq H$, we have
$$\vtil= v= \qtil_Z= \bfrtil(Z)+ \qtil_Z \,.$$
In case (3) with $r(Z)\not\subseteq H$ and $X\notin \calG(H)$, we have
$$\vtil= v= \bfrho(Z/H)= \bfr(Z/H)+ q'_{Z/H}= \bfrtil(Z)+ \qtil_Z \,.$$
Finally, in case (3) with $r(Z)\not\subseteq H$ and $X\in \calG(H)$, we have
$$\vtil=v= \bfrho(X/H)= \bfr(X/H)+ q'_{X/H}= \bfr(Z/H)+ \bfr \bigl( (X\setminus Z)/H \bigr) +q'_{X/H}= \bfrtil(Z)+ \qtil_Z \,.$$
Thus, the $\vtil$ and $\qtil_Z$ satisfy all the required relations of the type $v=\bfrho(Z)$.

Assume now that $\emptyset \ne Z_1 \subsetneq Z_2 \in \calZ_v\setminus \Cfin$ for some $v\in E^0$. Then $Z_2 \subseteq X$ for some $X\in \Cvinf$. If $v\in H$, then $r(Z_2) \subseteq H$ and
$$\qtil_{Z_1}= 0= \bfrtil(Z_2\setminus Z_1)+ \qtil_{Z_2} \,.$$
So, we may assume that $v\notin H$.

In case (2),
$$\qtil_{Z_1}= \bfr\bigl( (X\setminus Z_1)/H \bigr)= \bfr\bigl( (Z_2 \setminus Z_1)/H \bigr) + \bfr\bigl( (X\setminus Z_2)/H \bigr)= \bfrtil(Z_2\setminus Z_1) + \qtil_{Z_2} \,.$$
If $X\in \calG(H)\setminus G$, then
\begin{align*}
\qtil_{Z_1} &= q'_{X/H}+ \bfr \bigl( (X\setminus Z_1)/H \bigr)= q'_{X/H}+ \bfr \bigl( (X\setminus Z_2)/H \bigr)+ \bfr \bigl( (Z_2\setminus Z_1)/H \bigr) \\
 &= \qtil_{Z_2}+ \bfrtil(Z_2\setminus Z_1).
 \end{align*}
Hence, in dealing with case (3) we may assume that $X\notin \calG(H)$.
In case (3) with $r(Z_2) \subseteq H$, we have
$$\qtil_{Z_1}= v= \bfr\bigl( (Z_2 \setminus Z_1)/H \bigr) +v= \bfrtil(Z_2\setminus Z_1)+ \qtil_{Z_2} \,.$$
In case (3) with $r(Z_1)\subseteq H$ but $r(Z_2) \not\subseteq H$, we have
$$\qtil_{Z_1}= v= \bfrho(Z_2/H)= \bfr(Z_2/H)+ q'_{Z_2/H}= \bfrtil(Z_2\setminus Z_1) +\qtil_{Z_2} \,.$$
Finally, in case (3) with $r(Z_1) \not\subseteq H$, we have
$$\qtil_{Z_1}= q'_{Z_1/H}= \bfsig(Z_1/H, Z_2/H)= \bfr\bigl( (Z_2\setminus Z_1)/H \bigr) +q'_{Z_2/H}= \bfrtil(Z_2\setminus Z_1)+ \qtil_{Z_2} \,.$$
This verifies that the $\vtil$ and $\qtil_Z$ satisfy all the required relations of the type $q'_{Z_1}= \bfsig(Z_1,Z_2)$.

Therefore Claim 1 is established. Since $\Mtil$ is conical, $\ker\pi$ is an order-ideal of $M$.

{\bf Claim 2}: $I\subseteq \ker\pi$. Since $\ker\pi$ is an order-ideal, it suffices to show that $I(H,G) \subseteq \ker\pi$.

For $v\in H$, we have $\pi(v)= \vtil =0$. For $X\in G\cap \Cfin$, we have $X\in C_v$ for some $v\in E^0\setminus H$, and $\pi(q'_X)= \qtil_X =0$ by case (2) above. If $X\in G\cap \Cinf$, then $X/H\in \calZ\setminus S$ and we have $\pi(q'_{X/H})= \qtil_{X/H} =0$, using case (2) again. This verifies that all the generators of $I(H,G)$ lie in $\ker\pi$.

{\bf Claim 3}: $\psi(\ker\pi)= (H,G)$.

Let $(H',G')= \psi(\ker\pi)$. It is clear from the construction of $\pi$ that
$$H'= E^0\cap (\ker\pi)= H.$$
In view of Claim 2, it immediately follows that $G\subseteq G'$.

Consider $X\in \calG(H)$. Then $X\in C_v$ for some $v\in E^0\setminus H$. If $X$ is finite and $X\notin G$, then $\pi(q'_X)= \qtil_X \ne 0$ by (3). Hence, $q'_X\notin \ker\pi$ and so $X\notin G'$. If $X$ is infinite and $X\notin G$, then $\pi(q'_{X/H})=  \qtil_{X/H} \ne 0$ by (3). Thus $q'_{X/H} \notin \ker\pi$ in this case, and again $X\notin G'$. This verifies that $G'=G$, establishing Claim 3.

{\bf Claim 4}: $\ker\pi = I$, and therefore $\psi(I)= (H,G)$.

In view of Claim 3, Lemma \ref{HGfromMideal} implies that $\ker\pi= \langle I(H,G) \rangle= I$, which establishes Claim 4.
\end{construction}

\begin{theorem}
\label{lattmonoid} Let $(E,C)$ be a separated graph, and $S$ a subset of $\Cfin$. Then there are mutually inverse lattice
isomorphisms
$$ \varphi\colon \ACS \longrightarrow \calL (M(E,C,S)) \qquad\qquad \text{and} \qquad\qquad \psi\colon \calL(M(E,C,S)) \longrightarrow \ACS\,,$$
where $\varphi (H,G)= \langle I(H,G) \rangle$ for $(H,G)\in\ACS$ and $\psi$ is defined as in Definition {\rm\ref{defpsiphi}}.
\end{theorem}

\begin{proof}
The maps $\varphi$ and $\psi$ are well defined by construction, and Lemma \ref{HGfromMideal} shows that $\varphi\psi$ is the identity map on $\calL (M(E,C,S))$. That $\psi\varphi$ is the identity map on $\ACS$ follows from Claim 4 of Construction \ref{constrMtilpi}. It thus remains to show that $\varphi$ and $\psi$ are order-preserving, since then they are isomorphisms of posets, and therefore lattice isomorphisms.

Suppose $I_1\subseteq I_2$ are order-ideals of $M(E,C,S)$ and $(H_j,G_j)= \psi(I_j)$ for $j=1,2$. Obviously $H_1\subseteq H_2$. Let $X\in G_1$; then $X\in C_v$ for some $v\in E^0\setminus H_1$. First suppose that $X\in \calG(H_1)\cap \Cfin$ and $q'_X\in I_1$. If $X\in \calG(H_2)$, then $X\in G_2$. Otherwise, $r(X) \subseteq H_2$ and so $\bfr(X)\in I_2$, whence $v= \bfrho(X) \in I_2$, yielding $v\in H_2$ and $X\in C[H_2]$. Now suppose that $X\in \calG(H_1)\cap \Cinf$ and $q'_{X/H_1} \in I_1$. If $X\in \calG(H_2)$, then $q'_{X/H_2}$ is defined and
$$q'_{X/H_2}= \bfsig(X/H_2,X/H_1)= \bfr\bigl( X\cap r^{-1}(H_2\setminus H_1) \bigr) +q'_{X/H_1} \in I_2 \,,$$
so $X\in G_2$. Otherwise, $r(X)\subseteq H_2$ and so $\bfr(X/H_1) \in I_2$, whence $v= \bfrho(X/H_1) \in I_2$, again yielding $X\in C[H_2]$. We have now shown that $G_1\subseteq G_2\sqcup C[H_2]$, and so $(H_1,G_1) \le (H_2,G_2)$. Therefore $\psi$ is order-preserving.

Finally, let $(H_1,G_1)$ and $(H_2,G_2)$ be any elements of $\ACS$ such that $(H_1,G_1) \le (H_2,G_2)$. Obviously $H_1 \subseteq I(H_2,G_2)$. Consider $X\in G_1\cap \Cfin$. If $X\in G_2$, then $q'_X \in I(H_2,G_2)$ by definition of $I(H_2,G_2)$. If $X\in C_v$ for some $v\in H_2$, then
$$q'_X \le \bfrho(X)= v\in I(H_2,G_2)$$
and so $q'_X \in \langle I(H_2,G_2) \rangle$. Now consider $X\in G_1\cap \Cinf$. If $X\in G_2$, then $q'_{X/H_2} \in I(H_2,G_2)$. Since
$$q'_{X/H_1} \le \bfsig(X/H_2,X/H_1)= q'_{X/H_2} \,,$$ it follows that $q'_{X/H_1} \in \langle I(H_2,G_2) \rangle$. If $X\in C_v$ for some $v\in H_2$, then
$$q'_{X/H_1} \le \bfrho(X/H_1)= v\in I(H_2,G_2)$$
and again $q'_{X/H_1} \in \langle I(H_2,G_2) \rangle$. Thus, all the generators of $I(H_1,G_1)$ lie in $\varphi(H_2,G_2)$, and we conclude that $\varphi(H_1,G_1) \subseteq \varphi(H_2,G_2)$. Therefore $\varphi$ is order-preserving.
\end{proof}

In case $C=\Cfin=S$, the lattice $\ACS$ consists of pairs $(H,\emptyset)$, and so $\ACS$ is naturally isomorphic to the lattice of hereditary, $C$-saturated subsets of $E^0$. We thus obtain the following corollary of Theorem \ref{lattmonoid}.

\begin{corollary}  \label{MlattC=S}
Let $(E,C)$ be a separated graph such that all the sets in $C$ are finite, and let $\calH$ be the lattice of hereditary, $C$-saturated subsets of $E^0$. Then there are mutually inverse lattice isomorphisms
$$\varphi: \calH \longrightarrow \calL (M(E,C)) \qquad\qquad \text{and} \qquad\qquad \psi: \calL (M(E,C)) \longrightarrow \calH \,,$$
where $\varphi(H)=  \langle H\rangle$ for $H\in \calH$ and $\psi(I)= \{v\in E^0\mid v\in I\}$ for $I\in \calL (M(E,C))$.
\qed\end{corollary}

The combination of Theorem \ref{lattmonoid} with Propositions \ref{LVAisoTrA} and \ref{LVCLisoTrCL} yields the following description of the lattice of trace ideals in $CL_K(E,C,S)$.

\begin{theorem} \label{vqlatticeisom}
Let $(E,C)$ be a separated graph, $S$ a subset of $\Cfin$, and $A:= CL_K(E,C,S)$. Then there are mutually inverse lattice isomorphisms
$$\xi: \ACS \longrightarrow \Tr(A) \qquad\qquad \text{and} \qquad\qquad \theta: \Tr(A) \longrightarrow \ACS\,,$$
given by the rules
\begin{align*}
\xi(H,G) &:= \bigl\langle\; H\sqcup \{q_{X}\mid X\in G\cap \Cfin \}\sqcup \{ q_{X/H}\mid X\in G\cap \Cinf \} \; \bigr\rangle \\
\theta(J) &:= \bigl(\; E^0\cap J,\; \{X\in \calG (E^0\cap J)\cap \Cfin\mid q_{X}\in J \} \\
 &\hskip1.4truein
\sqcup \{X\in \calG (E^0\cap J)\cap \Cinf \mid q_{X/H}\in J\}
\;\bigr).
\end{align*}
\end{theorem}

\begin{proof} Set $M:= M(E,C,S)$, and let $\Gamma := \Gamma(E,C,S) : M\rightarrow \mon{A}$ be the monoid isomorphism given by Theorem \ref{computVMECS}. We shall also use $\Gamma$ to denote the induced lattice isomorphism $\calL(M) \rightarrow \calL(\mon{A})$. Due to Theorem \ref{lattmonoid} and Proposition \ref{LVAisoTrA}, we have mutually inverse lattice isomorphisms
$$\Phi\Gamma\varphi: \ACS \longrightarrow \Tr(A) \qquad\qquad \text{and} \qquad\qquad \psi\Gamma^{-1}\Psi: \Tr(A) \longrightarrow \ACS\,.$$
It is clear that $\psi\Gamma^{-1}\Psi(J)= \theta(J)$ for all $J\in \Tr(A)$. For $(H,G)\in \ACS$, we note, using Proposition \ref{LVCLisoTrCL}, that
$$\Phi\Gamma\varphi(H,G)= \langle\; \text{idempotents\ } e\in A \mid [e] \in J(H,G) \;\rangle,$$
where $J(H,G)$ is the order-ideal of $\mon{A}$ generated by the set
$$\{ [v] \mid v\in H\} \cup \{ [q_X] \mid X\in G\cap \Cfin \} \cup \{ [q_{X/H}] \mid X\in G\cap \Cinf \}.$$
Note that $\xi(H,G) \subseteq \Phi\Gamma\varphi(H,G)$.
If $e$ is an idempotent in $A$ such that $[e] \in J(H,G)$, then
$$[e] \le \sum_{i=1}^l\, [v_i] + \sum_{j=1}^m\, [q_{X_j}] + \sum_{k=1}^n\, [q_{Y_k/H}]$$
for some $v_i\in H$, $X_j\in G\cap\Cfin$, and $Y_k\in G\cap \Cinf$. Consequently, $e$ is equivalent to some idempotent
$$e' \le \text{diag}(v_1,\dots,v_l,q_{X_1},\dots,q_{X_m},q_{Y_1/H},\dots,q_{Y_n/H}),$$
from which it follows that $e$ lies in the ideal generated by these $v_i$, $q_{X_j}$, and $q_{Y_k/H}$. Therefore, $\Phi\Gamma\varphi(H)= \xi(H,G)$.
\end{proof}

The lattices $\Tr(CL_K(E,C,S))$ and $\calL(M(E,C,S))$ are also isomorphic to the lattice of hereditary, $(C,S)$-saturated subsets of a graph $XE\supseteq E$, which is a ``slimmer'' version of the graph in Construction \ref{makefinite}, as follows.

\begin{definition} \label{defXE}
The \emph{$(C,S)$-extension of $E$} is the graph $XE= X_{C,S}(E)$ with vertex set
$$XE^0 := E^0\sqcup (\calZ\setminus S)$$
and with edge set $XE^1$ consisting of the following four types of edges:
\begin{enumerate}
\item[] All edges in $E^1$;
\item[] One edge $v\rightarrow Z$ for each $v\in E^0$ and $Z\in \calZ_v\setminus S$;
\item[] One edge $Z\rightarrow r(f)$ for each $Z\in \calZ\setminus\Cfin$ and $f\in E^1\setminus Z$ such that $Z\sqcup \{f\} \in \calZ$;
\item[] One edge $Z\rightarrow Z'$ for each $Z,Z'\in \calZ\setminus \Cfin$ such that $Z\subsetneq Z'$.
\end{enumerate}
Of course, if all the sets in $C$ are finite, then $XE^1$ contains only edges of the first two types.
\end{definition}

\begin{definition} \label{defCSsatXE}
A subset $H\subseteq XE^0$ is \emph{$(C,S)$-saturated} provided the following three conditions hold:
\begin{enumerate}
\item If $X\in S\cap C_v$ for some $v\in E^0$ and $r(X) \subseteq H$, then $v\in H$.
\item If $Z\in \calZ_v\setminus S$ for some $v\in E^0$ with $Z\in H$ and $r(Z) \subseteq H$, then $v\in H$.
\item If $\emptyset\ne Z\subsetneq Z' \in \calZ\setminus\Cfin$ with $Z'\in H$ and $r(Z'\setminus Z) \subseteq H$, then $Z\in H$.
\end{enumerate}

Let $\HCS$ denote the family of all hereditary, $(C,S)$-saturated subsets of $XE^0$. This family is closed under arbitrary intersections, and hence it is a complete lattice with respect to inclusion.
\end{definition}

\begin{remark} \label{lattisoviaXE}
 Let $(E,C)$ be a separated graph, and $S$ a subset of $\Cfin$. There are lattice isomorphisms as indicated below. We omit the proofs, which are similar to those of Theorems \ref{lattmonoid} and \ref{vqlatticeisom}.

 First, there are mutually inverse lattice isomorphisms
$$ \varphi\colon \HCS \longrightarrow \calL (M(E,C,S)) \qquad\qquad \text{and} \qquad\qquad \psi\colon\calL(M(E,C,S)) \longrightarrow \HCS\,,$$
given by the rules
\begin{align*}
\varphi(H) &:= \bigl\langle\, (E^0\cap H) \cup \{q'_Z\mid Z\in (\calZ\setminus S)\cap H\} \,\bigr\rangle \\
\psi(I) &:= (E^0 \cap I) \sqcup \{Z\in \calZ\setminus S \mid q'_Z\in I\}.
\end{align*}
Second,  there are mutually inverse lattice isomorphisms
$$\xi: \HCS \longrightarrow \Tr(CL_K(E,C,S)) \qquad\qquad \text{and} \qquad\qquad \theta: \Tr(CL_K(E,C,S)) \longrightarrow \HCS\,,$$
given by the rules
\begin{align*}
\xi(H) &:= \langle\; (E^0\cap H)\cup \{q_Z \mid Z\in (\calZ\setminus S)\cap H\} \;\rangle \\
\theta(J) &:= \{v\in E^0 \mid v\in J\} \sqcup \{Z\in \calZ\setminus S \mid q_Z \in J\}.
\end{align*}
\end{remark}

%%%%%%%%%%%%%%%%%%%%%%%%%%%%%%%%%%%%%%%%%
\section{Simplicity}
\label{sec:simplicity}

The description of the lattices $\calL(M(E,C,S))$ in the previous section allows us to develop criteria for simplicity of the monoids $M(E,C,S)$. In view of Proposition \ref{LVAisoTrA}, we thus obtain criteria for $CL_K(E,C,S)$ to be ``trace-simple'' in the sense that the only trace ideals of this algebra are $0$ and $CL_K(E,C,S)$.

In general, a monoid $M$ is \emph{simple} provided $M$ has precisely two order-ideals, namely $M$ itself and the group of units of $M$. In the conical case (as for $M(E,C,S)$), simplicity means that $M$ is nonzero and its only order-ideals are $\{0\}$ and $M$.

\begin{theorem} \label{simplicity}
Let $(E,C)$ be a separated graph, and $S$ a subset of $\Cfin$. The following conditions are equivalent.
\begin{enumerate}
\item The only trace ideals of $CL_K(E,C,S)$ are $0$ and $CL_K(E,C,S)$.
\item $M(E,C,S)$ is a simple monoid.
\item \subitem {\rm (a)} $S= \Cfin$, and
\item[] \subitem {\rm (b)} The only hereditary, $(C,S)$-saturated subsets of $E^0$ are $\emptyset$ and $E^0$.
\end{enumerate}
\end{theorem}

\begin{proof} The equivalence of (1) and (2) is immediate from Proposition \ref{LVAisoTrA}, given that $M(E,C,S)$ is nonzero and conical (Lemma \ref{nonzeroconical}). Next, observe that for the hereditary, $(C,S)$-saturated subsets $\emptyset, E^0 \subseteq E^0$, we have $\calG(\emptyset)= \Cfin\setminus S$ and $\calG(E^0)= \emptyset$.

$(2)\Longrightarrow(3)$: Assume that $M(E,C,S)$ is a simple monoid. By Theorem  \ref{lattmonoid}, the only members of $\ACS$ are $(\emptyset,\emptyset)$ and $(E^0,\emptyset)$. Since $(\emptyset,\{X\}) \in \ACS$ for any $X\in \Cfin\setminus S$, condition (a) follows. Further, since $(H,\emptyset) \in \ACS$ for any hereditary, $(C,S)$-saturated subset $H$ of $E^0$, condition (b) follows as well.

$(3)\Longrightarrow(2)$: If (a) and (b) hold, the only members of $\ACS$ are $(\emptyset,\emptyset)$ and $(E^0,\emptyset)$. Theorem \ref{lattmonoid} then implies that $M(E,C,S)$ is simple.
\end{proof}

\begin{corollary}  \label{simpleME}
Let $E$ be any {\rm(}non-separated{\rm)} graph. Then $M(E)$ is a simple monoid if and only if the only hereditary, saturated subsets of $E^0$ are $\emptyset$ and $E^0$.
\end{corollary}

\begin{proof} We have $M(E)= M(E,C,S)$ where $S=\Cfin$ and $C$ is the union of the singleton collections $\{s^{-1}(v)\}$ for non-sinks $v\in E^0$. With these choices of $C$ and $S$, a subset of $E^0$ is $(C,S)$-saturated if and only if it is saturated. Therefore the corollary follows immediately from Theorem \ref{simplicity}.
\end{proof}

\begin{remark}
\label{rem:Csimple}
Note that if $C=\Cfin=S$, then every $(C,S)$-saturated subset of $E^0$ is saturated. Consequently, if $M(E)$ is simple, then so is $M(E,C,S)$. Namely, simplicity of $M(E)$ implies that $\emptyset$ and $E^0$ are the only hereditary, saturated subsets of $E^0$ (Corollary \ref{simpleME}), and so these are the only hereditary, $(C,S)$-saturated subsets of $E^0$.

Conversely, though, it is very easy to
produce examples $(E,C,S)$ with $C=\Cfin=S$ such that
$M(E,C,S)$ is simple but $M(E)$ is not simple. For example, consider
the following graph $E$:
$$\def\labelstyle{\displaystyle} \xymatrixcolsep{8ex}
\xymatrix{ v \dloopd{}_{e} \ar[r]^{f} &w
}$$
\smallskip

\noindent Let $C:= C_v :=\{\{e\}, \{f\}\}$, and take $S=C$.
Since $\{f\} \in S$, any $(C,S)$-saturated subset of $E^0$ which contains $w$ must also contain $v$. Thus, $\ACS= \{(\emptyset,\emptyset),(E^0,\emptyset)\}$, and so $M(E,C,S)$ is simple. This can be verified directly: $M(E,C,S)$ is generated by $v$ and $w$ with the relations $v=v=w$, and so $M(E,C,S) \cong \Z^+$. On the other hand, $\{w\}$ is a hereditary, saturated subset of $E^0$, whence $M(E)$ is not simple. In fact, $M(E)$ is generated by $v$ and $w$ with the sole relation $v=v+w$, and so $w$ generates a proper nonzero ideal of
$M(E)$.
\end{remark}

Corollary \ref{simpleME} is well known when $E$ is row-finite, in which case it is a consequence of \cite[Theorem 5.3]{AMP}.
It follows that $M(E)$ is simple if and only if $E$ is
cofinal, meaning that every vertex connects to any infinite path and
to any sink \cite[Lemma 2.8]{APS}. We develop here a suitable version of this criterion
for $M(E,C,S)$, provided Assumption $(*)$ of Definitions \ref{defstar}, \ref{defstarstar} holds.

We will only need to consider infinite \emph{forward} paths, that is, paths $(e_1,e_2,\dots)$
with $r(e_i)= s(e_{i+1})$ for all $i\in\N$. For any finite or infinite path $\gamma$ and any
integer $n\ge0$, let us write $\gamma[n]$ to denote the \emph{$n$-truncation} of $\gamma$,
that is, the initial subpath consisting of the first $n$ edges of $\gamma$.
(This is defined only if $\gamma$ has length at least $n$.)

It is convenient to introduce a further definition. A vertex $v$ in
$E$ is said to be a \emph{$\Cfin$-sink} in case $\Cvfin=\emptyset$.

\begin{definition}  \label{multipath} Let $(E,C)$ be a separated
graph and let $v$ be a vertex in $E$. An \emph{infinite
$C$-multipath in $E$ starting at $v$} is a nonempty collection
$\Gamma$ of paths in $E$ satisfying the following two conditions:
\begin{itemize}
\item[(a)] Every path in $\Gamma$ starts at $v$ and either is an infinite path or ends in a $\Cfin$-sink.
\item[(b)] For any integer $n\ge0$ and any $\gamma \in \Gamma $ of length at least $n$,
either $r(\gamma[n])$ is a $\Cfin$-sink or for every $X$ in
$C_{r(\gamma[n]),\text{fin}}$ there exist $f\in X$ and $\gamma'\in
\Gamma$ such that $\gamma'[n+1]=\gamma[n]f$.
\end{itemize}
Similarly, for $m\in \Z^+$, a \emph{$C$-multipath in $E$ of length $m$ starting at $v$} is a nonempty
collection $\Gamma$ of finite paths in $E$ satisfying
\begin{itemize}
\item[(a$'$)] Every path in $\Gamma$ starts at $v$ and either has length exactly $m$ or ends in a $\Cfin$-sink.
\item[(b$'$)] Condition (b), for all $n<m$.
\end{itemize}
In either case, we denote by $\Gamma^0$ (respectively, $\Gamma^1$) the set of all vertices
(respectively, edges) occurring in the paths in $\Gamma$.

Let us say that $E$ is \emph{$C$-cofinal} if, given any vertex $w\in
E^0$ and any infinite $C$-multipath $\Gamma$ in $E$, there is a path
from $w$ to some vertex in $\Gamma^0$. In particular, this condition
implies that there are paths in $E$ from any vertex to any $\Cfin$-sink.
\end{definition}

In the absence of condition $(*)$ (or some similar property), it is easily possible for
$M(E,C)$ to be simple without $E$ being $C$-cofinal. For example, let $E$ be the graph
$$\xymatrixrowsep{0.5pc}\xymatrixcolsep{6pc}\def\labelstyle{\displaystyle}
\xymatrix{
x &v \ar[l]_{e} \ar[r]^{f} &y
}$$
and take $C=S=C_v := \{\{e\},\{f\}\}$. Then $M(E,C)\cong \Z^+$,
a simple monoid. On the other hand, there is no path in $E$ from the vertex $x$ to the $\Cfin$-sink $y$,
so $E$ is not $C$-cofinal.

\begin{lemma}
\label{Csaturatiser} Assume condition $(*)$ for $(E,C)$ holds. Let
$H$ be a hereditary subset of $E^0$, and define $\Hbar=
\bigcup_{n=0}^\infty H_n$ where $H_0 := H$ and
$$H_n :=H_{n-1}\cup \{ v\in E^ 0\mid \text{there exists\ } X\in \Cvfin \text{\ such that\ } r(X)\subseteq H_{n-1} \}$$
for $n>0$. Then $\Hbar$
is hereditary and
$C$-saturated.
\end{lemma}

\begin{proof} It is clear from the construction of $\Hbar$ that this set is $C$-saturated.

In order to make use of condition $(*)$, we require the free abelian
monoid $F$ of Section \ref{sect:refinement}, where we take $S= \Cfin$. For $\alpha \in F$, we
shall denote by $\text{supp}_{E^0}(\alpha)$ the set $E^0\cap
\text{supp}(\alpha)$. The definition of the sets $H_n$ can be
rewritten in the form
$$H_n=H_{n-1}\cup \{ v\in E^ 0\mid \text{supp}_{E^0}(\bfr(X))\subseteq H_{n-1} \text{ for some } X\in \Cvfin \}.$$

We show by induction that each $H_n$ is hereditary. If $n=0$, this
is our hypothesis. Now assume that $H_{n-1}$ is hereditary, for some
$n\ge 1$. To see that $H_n$ is hereditary, it suffices to show that
$r(e)\in H_n$ for every $v\in H_n\setminus H_{n-1}$ and $e\in
s^{-1}(v)$. We choose $Z\in \calZ_v$ such that $e\in Z$. On the
other hand, by definition of $H_n$ there is some $X\in \Cvfin$ such
that $r(X)\subseteq H_{n-1}$.

By $(*)$, there is some $\alpha \in F$ such that $\bfr(X)
\rightsquigarrow_1 \alpha $ and $\bfrho (Z) \rightsquigarrow_1
\alpha$. Since $H_{n-1}$ is hereditary, we see that
$\text{supp}_{E^0} (\alpha)\subseteq H_{n-1}$. Now there are various
possibilities for $r(e)$. If $r(e)\in \text{supp}_{E^0}(\alpha)$ then
$r(e)\in H_{n-1}\subseteq H_n$. If $\bfr (Y')\le \alpha$ for some
$Y'\in C_{r(e),\text{fin}}$, then $r(e)\in H_n$ too. Finally, if
$\bfrho(Z') \le \alpha$ for some $Z'\in \calZ_{r(e)} \setminus \Cfin$, then $q'_{Z'}\le \alpha$. Since $\bfr(X)
\rightsquigarrow_1 \alpha $, it follows that $r(e)\le \bfr (X)$ and so
$r(e)\in H_{n-1}\subseteq H_n$. In any case we get $r(e)\in H_n$, as
desired.
\end{proof}

\begin{theorem}
\label{Ccofinal} Let $(E,C)$ be a separated graph and $S$ a subset of $\Cfin$.
Assume condition $(*)$ for $(E,C)$ holds. Then
$M(E,C,S)$ is simple if and only if the following conditions hold:
\begin{itemize}
\item[(a)] $S= \Cfin$.
\item[(b)] $E$ is $C$-cofinal.
\end{itemize}
\end{theorem}

\begin{proof}
In view of Theorem \ref{simplicity}, we may assume that $S=\Cfin$, and it suffices to show that $E$ is $C$-cofinal if and only if
the only hereditary $C$-saturated subsets of $E^0$ are $\emptyset $
and $E^0$.

Assume first that there is a proper nonempty hereditary
$C$-saturated subset $H$ in $E^0$. Let $v\in E^0\setminus H$, and
let $\Gamma_0$ be the set consisting of the path of length zero at
$v$. Thus, $\Gamma_0$ is a $C$-multipath of length $0$ starting at
$v$. Now suppose that, for some $m\in\Z^+$, we have constructed a
$C$-multipath $\Gamma _m$ of length $m$ starting at $v$, such that
$\Gamma _m^0$ is disjoint from $H$. For each $\gamma \in \Gamma _m$
such that $r(\gamma )$ is not a $\Cfin$-sink, and each $X\in C_{r(\gamma
),\text{fin}}$, it follows from the $C$-saturation of $H$ that there
is some $f_{m+1}\in X$ such that $r(f_{m+1})\notin H$. So, we
enlarge $\gamma $ to a path $\gamma'= \gamma f_{m+1}$ of length
$m+1$. The set of paths obtained in this way, for all $\gamma \in
\Gamma _m$ and all $X\in C_{r(\gamma ),\text{fin}}$, together with
the set of paths in $\Gamma_m$ that end in $\Cfin$-sinks, forms a
$C$-multipath $\Gamma_{m+1}$ of length $m+1$ starting at $v$, such
that $\Gamma_{m+1}^0$ is disjoint from $H$. Thus, we can build
$C$-multipaths $\Gamma_m$ of length $m$ for every $m\in\Z^+$ in a
compatible way. We now define $\Gamma $ as the set of paths from
$\bigcup_{m\ge0} \Gamma_m$ ending in $\Cfin$-sinks, together with those
infinite paths $\gamma$ such that $\gamma[m]\in \Gamma _m$ for all
$m\in\Z^+$. Clearly, $\Gamma $ is an infinite $C$-multipath, and
$\Gamma^0$ is disjoint from $H$. Since $H$ is hereditary, this means
that there is no path from any vertex of $H$ to any vertex in
$\Gamma^0$. Since $H$ is nonempty, we conclude that $E$ is not
$C$-cofinal.

Conversely, assume that $\Gamma $ is an infinite $C$-multipath in
$E$ and that $w$ is a vertex of $E$ not connecting to $\Gamma^0$.
Let $H$ be the hereditary subset of $E^0$ generated by $w$, and form
$\Hbar$ as in Lemma \ref{Csaturatiser}. Clearly, $H_0=H$ is disjoint
from $\Gamma^0$. Now assume, for some $n\ge0$, that $H_n$ is
disjoint from $\Gamma^0$, and suppose that $v\in H_{n+1}\cap
\Gamma^0$. Then by definition, there exists $X\in \Cvfin$ such that
$r(X) \subseteq H_n$. On the other hand, there is a path
$\gamma \in \Gamma $ such that $v=r(\gamma[m])$ for some $m$. By the
definition of a $C$-multipath, it follows that there are a path
$\gamma '$ in $\Gamma $ and an edge $f\in X$ such that $\gamma
'[m+1] =\gamma[m] f$. In particular, $r(f)\in H_n\cap \Gamma^0$, a
contradiction. Thus, $\Hbar\cap \Gamma^0=\emptyset$. Since $\Hbar$
is hereditary and $C$-saturated, we conclude that there is a
nonempty, proper hereditary $C$-saturated subset in $E^0$.
\end{proof}

It is interesting to compare our situation with the one in \cite[Lemma 2.7]{APS}: If $E$ is row-finite, $M(E)$
is simple, and $E$ contains a sink, then $E^0$ contains a unique sink,
and there are no infinite paths in $E$. In our case, if $M(E,C)$ is simple, then there is at most one sink in $E$ if Assumption $(*)$ holds, but not otherwise
(see the example following Definition \ref{multipath}). Further, even if $(*)$ holds,
infinite paths may occur, as the example in Remark \ref{rem:Csimple} shows. Moreover, $E$ may contain arbitrarily many $\Cfin$-sinks. For example, let $E^0$ be an arbitrary nonempty set, choose $E^1$ to contain infinitely many edges from any vertex in $E^0$ to any other, and set $C_v= \{ s^{-1}(v) \}$ for $v\in E^0$. Then all vertices of $E$ are $\Cfin$-sinks, and $(*)$ holds vacuously. Since $E$ is clearly cofinal, $M(E,C)= M(E)$ is simple.

%%%%%%%%%%%%%%%%%%%%%%%%%%%%%%%%%%%%%
\section{Resolutions}
\label{sect:unitembeddings}

Our next aim is to develop a construction that allows us to embed graph monoids without refinement into ones with refinement. As Wehrung has proved in \cite[Proposition 1.5 and Theorem 1.8]{W}, every conical abelian monoid can be embedded in a conical refinement monoid. We obtain a more ``visual'' version of this result, in that all the monoids that appear are graph monoids. We restrict attention to finitely separated graphs since arbitrary conical abelian monoids can be obtained as graph monoids of finitely separated graphs (Proposition \ref{allconical}). Our construction process involves adjoining vertices and edges designed to satisfy Assumption $(*)$. We also want the resulting monoid embeddings to preserve properties such as failure of cancellation or separativity. To obtain this, we arrange for embeddings of the following type.

\begin{definition}{\rm Following \cite{W}, a monoid homomorphism $\psi:M\rightarrow F$ is \emph{unitary} provided
\begin{enumerate}
\item $\psi$ is injective;
\item $\psi(M)$ is \emph{cofinal} in $F$, that is, for each $u\in F$ there is some $v\in M$ with $u\le \psi(v)$;
\item whenever $u,u'\in M$ and $v\in F$ with $\psi(u)+v= \psi(u')$, we have $v\in \psi(M)$.
\end{enumerate}
}\end{definition}

\begin{lemma}
\label{embedding1} Let $F$ be the free abelian monoid generated by elements
$a_{ij}$ for $1\le i\le n$, $ 1\le j\le m$. Let $(\delta _{ij})$ be
an $n\times m$ matrix of positive integers. Consider the elements
$c_i=\sum _{j=1}^m \delta _{ij}a_{ij}$ for $i=1,\dots ,n$, and
$d_j=\sum _{i=1}^n \delta_{ij} a_{ij}$ for $j=1,\dots ,m$. Let $M$
be the monoid with generators $x_1,\dots ,x_n,y_1,\dots ,y_m$, subject to the
single relation
$$\sum _{i=1}^n x_i= \sum _{j=1}^m y_j.$$
Then the natural monoid homomorphism $\psi \colon M\to F$ sending
$x_i$ to $c_i$ and $y_j$ to $d_j$ is unitary.
\end{lemma}

\begin{proof} Since all $\delta_{ij}>0$, we have $a_{ij}\le c_i= \psi(x_i)$ for all $i$, $j$. Therefore $\psi(M)$ is cofinal in $F$.

Suppose that $u,u'\in M$ and $v\in F$ with $\psi(u)+v= \psi(u')$.
Write $u= \sum _{i=1}^n \lambda _i x_i + \sum_{j=1}^m \mu _j y_j$
and $u'= \sum _{i=1}^n \lambda' _i x_i + \sum_{j=1}^m \mu' _j y_j$
for some $\lambda_i,\mu_j,\lambda'_i,\mu'_j \in \Zplus$. Then
$\psi(u)= \sum_{i,j} (\lambda_i+\mu_j)\delta_{ij}a_{ij}$ and
$\psi(u')= \sum_{i,j} (\lambda'_i+\mu'_j)\delta_{ij}a_{ij}$. Hence,
$\lambda_i+\mu_j\le \lambda'_i+\mu'_j$ for all $i$, $j$, and
$$v= \sum_{i,j} (\lambda'_i+\mu'_j-\lambda_i-\mu_j)\delta_{ij}a_{ij}.$$
Since $\lambda_i-\lambda'_i\le \mu'_j-\mu_j$ for all $i$, $j$, there
exists $\lambda\in\Z$ such that $\lambda_i-\lambda'_i\le \lambda\le
\mu'_j-\mu_j$ for all $i$, $j$. Then we can define
$$w= \sum_{i=1}^n (\lambda'_i+\lambda-\lambda_i)x_i+ \sum_{j=1}^m (\mu'_j-\mu_j-\lambda)y_j \in M,$$
and $\psi(w)=v$. This verifies the third unitarity condition, and it
only remains to show that $\psi$ is injective.

It is clear that each element $x$ in $M$ can be written in the form
$$x=\sum _{i=1}^n \lambda _i x_i + \sum_{j=1}^m \mu _j y_j$$
where at least one of the $\mu _j$'s is $0$. So it suffices to show
that the elements $\sum _{i=1}^n \lambda _i c_i + \sum_{j=1}^m \mu
_j d_j$ with at least one $\mu _j=0$ are all distinct in $F$.

Assume that
$$\sum _{i=1}^n \lambda _i c_i + \sum_{j=1}^m \mu _j d_j=\sum _{i=1}^n \lambda _i' c_i + \sum_{j=1}^m \mu _j'
d_j $$ in $F$, where there are $j_0$ and $j_1$ such that $\mu
_{j_0}=0=\mu_{j_1}'$. We have
$$\sum _{i=1}^n\sum _{j=1}^m (\lambda _i+\mu _j)\delta_{ij}a_{ij}=\sum _{i=1}^n\sum _{j=1}^m (\lambda _i'+\mu_j')\delta _{ij}a_{ij} \,,$$
so that $\lambda _i+\mu _j=\lambda _i'+\mu
_j'$ for all $i$, $j$. Assume first that $j_0=j_1$. Then we get
$\lambda _i=\lambda _i '$ for all $i$, and then $\mu _j=\mu_j'$ for
all $j$, as desired. If $j_0\ne j_1$, then we get
$$\lambda _i= \lambda'_i+\mu'_{j_0}= \lambda _i+\mu _{j_0}'+\mu_{j_1},$$
which implies that $\mu_{j_0}'=\mu_{j_1}=0$, and we reduce to the
above case.
\end{proof}

\begin{corollary}
\label{embeddingcor} Let $F$ be the free abelian monoid generated by
elements $a_{ij}^\gamma$ for $1\le i\le n_{\gamma}$, $1\le j\le
m_{\gamma }$, $\gamma \in \Gamma$, and let $(\delta _{ij}^\gamma)$,
for $\gamma\in\Gamma$, be corresponding matrices of positive
integers. Consider the elements $c_i^\gamma=\sum _{j=1}^{m_{\gamma}}
\delta _{ij}^{\gamma}a_{ij}^{\gamma}$ for $i=1,\dots ,n_{\gamma}$
and $d_j^{\gamma }=\sum _{i=1}^{n_{\gamma}}
\delta_{ij}^{\gamma}a_{ij}^\gamma$ for $j=1,\dots ,m_{\gamma}$. Let
$M$ be the monoid given by generators $x_i^\gamma$ and $y_j^\gamma
$ for all $i$, $j$, $\gamma$, subject to the relations
$$\sum _{i=1}^{n_{\gamma}} x_i^\gamma = \sum _{j=1}^{m_{\gamma}} y_j^\gamma, \qquad (\gamma\in\Gamma),$$
and let $\psi:M \rightarrow F$ be the natural homomorphism sending
$x_i^\gamma \mapsto c_i^\gamma$ and $y_j^\gamma \mapsto d_j^\gamma$
for all $i$, $j$, $\gamma$. Then $\psi$ is unitary.
\end{corollary}

\begin{proof}
Observe that $F= \bigoplus_{\gamma\in\Gamma} F_\gamma$ where
$F_\gamma$ is the free abelian monoid generated by the $a_{ij}^\gamma$ for
$1\le i\le n_{\gamma}$ and $1\le j\le m_{\gamma }$, and $M=
\bigoplus_{\gamma\in\Gamma} M_\gamma$ where $M_\gamma$ is the monoid
given by generators
$x^\gamma_1,\dots,x^\gamma_{n_\gamma},y^\gamma_1,\dots,y^\gamma_{m_\gamma}$
and the relation $\sum _{i=1}^{n_{\gamma}} x_i^\gamma = \sum
_{j=1}^{m_{\gamma}} y_j^\gamma$. Further, $\psi=
\bigoplus_{\gamma\in\Gamma} \psi_\gamma$ where $\psi_\gamma:
M_\gamma \rightarrow F_\gamma$ is the natural homomorphism sending
$x_i^\gamma \mapsto c_i^\gamma$ and $y_j^\gamma \mapsto d_j^\gamma$.
By Lemma \ref{embedding1}, each $\psi_\gamma$ is unitary, and
therefore $\psi$ is unitary.
\end{proof}

Recall from Definitions \ref{defSSGr} and \ref{defSGr} the categories $\SSGr$ and $\SGr$, as well as the
definition of an $SG$-subgraph of a separated graph $(E,C)$. We shall need a
category of finitely separated graphs, defined as follows:

\begin{definition} \label{defFSGr}{\rm
Define a category $\FSGr$ whose objects are all finitely separated graphs $(E,C)$. A morphism
$\phi: (E,C) \rightarrow (F,D)$ in $\FSGr$ is any graph morphism
$\phi: E\rightarrow F$ such that
\begin{enumerate}
\item $\phi^0$ is injective.
\item For each $v\in E^0$ and each $X\in C_v$, there is some $Y\in D_{\phi^0(v)}$
such that $\phi^1$ induces a bijection $X\rightarrow Y$.
\end{enumerate}}
\end{definition}

Observe that $\FSGr$ is a full subcategory of $\SGr$. Recall that $\SGr$ admits arbitrary direct limits (Proposition \ref{colim}), and observe that $\FSGr$ is closed (in $\SGr$) under direct limits. Moreover, every object $(E,C)$ in $\FSGr$ is the
direct limit of the directed system of its finite complete
subobjects, as one can see from the proof of Proposition
\ref{dirlimcomplete}. It is worth to mention that this latter
property does {\it not} hold in the category $\SGr$.

\begin{construction} \label{resolution2}{\rm
Let $(E,C)$ be an object in $\FSGr$, and let $T= \{(w_k,X_k,Y_k)\mid
k\in I\}$ be a collection of distinct ordered triples such that for each
$k\in I$, $w_k$ is a vertex of $E$ and $X_k$, $Y_k$ are distinct
members of $C_{w_k}$. Let $\delta =\{\delta ^k\mid k\in I\}$, where
$\delta ^k $ is an $X_k\times Y_k$ matrix of positive integers for
all $k$. We construct a \emph{$\delta$-$T$-resolution for $(E,C)$} as
follows. It is a  finitely separated graph $(E_T,C^T)$, containing $(E,C)$ as
an SG-subgraph, with the following data:
\begin{enumerate}
\item $E_T^0 := E^0\sqcup \{v^k_{e,f} \mid k\in I,\; e\in X_k,\; f\in Y_k\}$.
\item $E_T^1 := E^1\sqcup \{ g^k_{e,f,j},\, h^k_{e,f,j} \mid k\in I,\; e\in X_k,\; f\in Y_k,\; 1\le j\le \delta^k_{e,f} \}$.
\item $s(g^k_{e,f,j})= r(e)$, $s(h^k_{e,f,j})= r(f)$ and $r(g^k_{e,f,j})= r(h^k_{e,f,j})= v^k_{e,f}$
for all $k\in I$, $e\in X_k$, $f\in Y_k$, $1\le j\le \delta^k_{e,f}$.
\item $X^k_e := \{g^k_{e,f,j} \mid f\in Y_k, \; 1\le j\le \delta^k_{e,f}\}$ for $k\in I$, $e\in X_k$, and\\
$Y^k_f := \{h^k_{e,f,j} \mid e\in X_k, \; 1\le j\le \delta^k_{e,f}\}$ for $k\in I$, $f\in Y_k$.
\item $C^T_u := C_u \sqcup \{X^k_e\mid k\in I,\; e\in X_k,\; r(e)=u\}
\sqcup \{Y^k_f \mid k\in I,\; f\in Y_k,\; r(f)=u \}$ for all $u\in
E^0$.
\item $C^T_u := \emptyset$ for all $u\in E_T^0\setminus E^0$.
\end{enumerate}

Observe that in the free abelian monoid on $E_T^0$, we have $r(e)
\rightarrow_1 \bfr(X^k_e)= \sum_{f\in Y_k} \delta^k_{e,f}v^k_{e,f}$
for all $k\in I$, $e\in X_k$, and so
$$\bfr(X_k)
\rightsquigarrow_1 \sum_{e\in X_k,\,f\in Y_k} \delta^k_{e,f} v^k_{e,f}$$
for all $k$. Similarly, $\bfr(Y_k) \rightsquigarrow_1
\sum_{e\in X_k,\,f\in Y_k} \delta^k_{e,f}v^k_{e,f}$ for all $k$, and thus $(*)$
holds in $(E_T,C^T)$ for $w_k$, $X_k$, $Y_k$. }\end{construction}

\begin{lemma} \label{unitary2}
Let $(E,C)$ be a separated graph, and $T$, $\delta$ as in Construction {\rm\ref{resolution2}}. Let
$(E_T,C^T)$ be a $\delta$-$T$-resolution for $(E,C)$, and $\iota:
(E,C) \rightarrow (E_T,C^T)$ the inclusion morphism. Then $M(\iota):
M(E,C)\rightarrow M(E_T,C^T)$ is unitary.
\end{lemma}

\begin{proof} Keep the notation of Construction \ref{resolution2}, and set $\mu := M(\iota)$.

Let $F$ be the free abelian monoid with generators $a^k_{e,f}$ for $k\in I$,
$e\in X_k$, $f\in Y_k$, and let $M$ be the monoid given by
generators $x^k_e$ for $k\in I$, $e\in X_k$ and $y^k_f$ for $k\in
I$, $f\in Y_k$ subject to the relations $\sum_{e\in X_k} x^k_e= \sum_{f\in
Y_k} y^k_f$ for $k\in I$. There is a natural homomorphism $\psi:
M\rightarrow F$ sending $x^k_e \mapsto \sum_{f\in Y_k} \delta^k_{e,f} a^k_{e,f}$ for $k\in I$, $e\in X_k$ and $y^k_f\mapsto
\sum_{e\in X_k} \delta^k_{e,f}a^k_{e,f}$ for $k\in I$, $f\in Y_k$,
and Corollary \ref{embeddingcor} shows that $\psi$ is unitary.

There is a unique homomorphism $\eta: M\rightarrow M(E,C)$ sending
$x^k_e\mapsto r(e)$ for $k\in I$, $e\in X_k$ and $y^k_f\mapsto
r(f)$ for $k\in I$, $f\in Y_k$ (because $\bfr(X_k)= w_k=
\bfr(Y_k)$ in $M(E,C)$), and there is a unique homomorphism
$\eta': F\rightarrow M(E_T,C^T)$ sending $a^k_{e,f} \mapsto
v^k_{e,f}$ for $k\in I$, $e\in X_k$, $f\in Y_k$. We observe that the following diagram commutes:
\begin{equation} \label{monres}
\xymatrixrowsep{3pc}\xymatrixcolsep{5pc} \xymatrix{
M \ar[r]^{\psi} \ar[d]_{\eta} &F \ar[d]^{\eta'} \\
M(E,C) \ar[r]^{\mu} &M(E_T,C^T) }\end{equation}
This holds because
$$\mu\eta(x^k_e)= r(e)= \bfr(X^k_e)= \sum_{f\in Y_k} \delta^k_{e,f} v^k_{e,f}= \sum_{f\in Y_k} \delta^k_{e,f} \eta'(a^k_{e,f})= \eta'\psi(x^k_e)$$
for all $k\in I$, $e\in X_k$, and similarly $\mu\eta(y^k_f)= \eta'\psi(y^k_f)$ for all $k\in I$, $f\in Y_k$.
We claim that
\eqref{monres} is a pushout in the category of abelian monoids. It
will then follow from \cite[Lemma 1.6]{W} that $\mu$ is unitary,
completing the proof.

Suppose we are given a monoid $N$ and homomorphisms $\sigma:
M(E,C)\rightarrow N$ and $\sigma': F\rightarrow N$ such that
$\sigma\eta= \sigma'\psi$. Set $n_u := \sigma(u)$ for $u\in E^0$
and $b^k_{e,f} := \sigma'(a^k_{e,f})$ for $k\in I$, $e\in X_k$,
$f\in Y_k$. Let $F^T$ denote the free abelian monoid on $E_T^0$, and
identify $M(E_T,C^T)$ with $F^T/{\sim}$ as in Section
\ref{sect:refinement}. Then let $\tau: F^T \rightarrow N$ be the
unique homomorphism such that $\tau(u)=n_u$ for $u\in E^0$ and
$\tau(v^k_{e,f})= b^k_{e,f}$ for $k\in I$, $e\in X_k$, $f\in Y_k$.

In order to see that $\tau$ factors through $M(E_T,C^T)$, we need to
show that
\begin{equation} \label{factor}
\tau(u)= \tau(\bfr(Z))
\end{equation}
for any $u\in E_T^0$ and any $Z\in C^T_u$. Since $C^T_u=
\emptyset$ for all $u\in E^0_T \setminus E^0$, we only need
to consider $u\in E^0$. There are two cases, depending on whether or
not $Z\in C_u$.

Assume first that $Z\in C_u$. Then $u= \bfr(Z)$ in $M(E,C)$, and
since the ranges of the edges in $Z$ all lie in $E^0$, we have
$$\tau(u)= n_u= \sigma(u)= \sigma(\bfr(Z))= \sum_{d\in Z}
\sigma(r(d)) = \sum_{d\in Z} n_{r(d)}= \sum_{d\in Z} \tau(r(d))= \tau(\bfr(Z)).$$
This verifies \eqref{factor} when $Z\in C_u$. If $Z\notin C_u$, then
either $Z=X^k_e$ for some $k\in I$, $e\in X_k$ with $r(e)=u$, or
$Z=Y^k_f$ for some $k\in I$, $f\in Y_k$ with $r(f)=u$. These two
cases are symmetric; we treat the former. Then,
\begin{align*}
\tau(u) &= n_u= \sigma(r(e))= \sigma\eta(x^k_e)=
\sigma'\psi(x^k_e)\\
 &= \sum_{f\in Y_k} \delta^k_{e,f}\sigma'(a^k_{e,f})= \sum_{f\in Y_k}
\delta^k_{e,f}b^k_{e,f}
 =\sum_{f\in Y_k} \delta^k_{e,f}\tau(v^k_{e,f})= \tau({\mathbf
r}(X^k_e)).
\end{align*}
This verifies \eqref{factor} in the case $Z= X^k_e$.

Now $\tau$ induces a homomorphism $\taubar: M(E_T,C^T) \rightarrow
N$. In particular, $\taubar(u)= n_u= \sigma(u)$ for all $u\in
E^0$, which implies that $\taubar\mu= \sigma$. Also,
$\taubar\eta'(a^k_{e,f})= \taubar(v^k_{e,f})= b^k_{e,f}=
\sigma'(a^k_{e,f})$ for $k\in I$, $e\in X_k$ and $f\in Y_k$, whence
$\taubar\eta'=\sigma'$. Uniqueness of $\taubar$ with respect to
these equations is clear, and therefore \eqref{monres} is a pushout,
as desired.
\end{proof}

\begin{remark}
\label{remark:univpro} (Universal property) Given a set $I$ and
an (abelian) monoid $N$, consider a set $\calE$ of equations in $N$:
$$\sum _{n=1}^{N_k} a_n^k=\sum _{m=1}^{M_k} b_m^k \qquad (k\in I).$$
Let $\delta =\{ \delta^k\mid k\in I\} $ be a set of integer
matrices with strictly positive entries, where each $\delta^k$ is  $N_k\times M_k$. A
{\it $\delta $-refinement} of $\calE$ in $N$ is a set of elements
$$\{c^k_{n,m}\mid k\in I, \; 1\le n\le N_k, \; 1\le m\le M_k \} \subseteq N$$
such that $a_n^k=\sum _{m=1}^{M_k} \delta ^k_{n,m} c^k_{n,m}$ and
$b_m^k=\sum _{n=1}^{N_k} \delta ^k_{n,m}c^k_{n,m}$ for all $k$, $n$, $m$.

The pushout property appearing in the proof of the above lemma is
equivalent to the following universal property of
$M(E_T,C^T)$: Given a monoid homomorphism $\Phi\colon M(E,C)\to N$,
given a family $T=\{(w_k,X_k,Y_k) \mid k\in I \}$ as in Construction \ref{resolution2} with $N_k= |X_k|$ and $M_k= |Y_k|$ for all $k\in I$, and given a $\delta
$-refinement
$$\{c^k_{e,f}\mid k\in I,\ e\in X_k,\ f\in Y_k\}$$
of the
set of equations $\Phi ({\bf r}(X_k))=\Phi ({\bf r} (Y_k))$, $k\in
I$, in $N$, there exists a unique monoid homomorphism
$\tilde{\Phi}\colon M(E_T,C^T)\to N$ such that $\tilde{\Phi}(v)= \Phi(v)$ for all $v\in E^0$ and $\tilde{\Phi} (v^k_{e,f})=
c^k_{e,f}$ for all $k\in I$, $e\in X_k$, $f\in Y_k$.
\end{remark}

\begin{construction} \label{resolution3}{\rm
Let $(E,C)$ be an object in $\FSGr$. Recursively, construct a $\delta
_n$-$T_n$-res\-o\-lu\-tion $(E_{n+1},C^{n+1})$ for $(E_n,C^n)$, for
$n=0,1,2,\dots$, where
\begin{enumerate}
\item $(E_0,C^0) := (E,C)$.
\item For each $v\in E_n^0\setminus E^0_{n-1}$ (where $E^0_{-1} := \emptyset$) and distinct $X,Y\in C^n_v$, at least one of the triples
$(v,X,Y)$ or $(v,Y,X)$ appears in $T_n$. (We explicitly allow the possibility that both of these triples appear. It could be useful to exclude triples for which
$(*)$ already holds in $(E_n,C^n)$.)
\item $\delta_n$ is a set of matrices of positive integers corresponding to the triples in $T_n$ as in Construction \ref{resolution2}.
\end{enumerate}
The direct limit of the sequence $(E_0,C^0) \rightarrow (E_1,C^1)
\rightarrow \cdots$ will be called a \emph{complete resolution of
$(E,C)$}. }\end{construction}

\begin{theorem} \label{unitary3}
Let $(E,C)$ be a finitely separated graph, $(E_+,C^+)$ a complete resolution
for $(E,C)$, and $\iota: (E,C) \rightarrow (E_+,C^+)$ the inclusion
morphism. Then $M(E_+,C^+)$ is a refinement monoid and $M(\iota):
M(E,C)\rightarrow M(E_+,C^+)$ is unitary.
\end{theorem}

\begin{proof} It follows immediately from Lemma \ref{unitary2} and Construction \ref{resolution3} that $M(\iota)$ is unitary. If $v\in E_+^0$ and $X$, $Y$ are distinct members of $C^+_v$, then for some $n$, either $(v,X,Y)$ or $(v,Y,X)$ appears in $T_n$. By constuction, axiom $(*)$ holds for $v$, $X$, $Y$ in $(E_{n+1},C^{n+1})$, and hence also in $(E_+,C^+)$.
Therefore Proposition \ref{refinement} shows that $M(E,C)$ is a
refinement monoid.
\end{proof}

Theorem \ref{unitary3} provides a way to construct explicit examples of refinement monoids with particular properties. We illustrate this by constructing a conical refinement monoid which fails to have \emph{separative cancellation}, that is, it contains elements $x$ and $y$ satisfying $2x=x+y=2y$ but $x\ne y$. The existence of such monoids was first obtained by Wehrung, as a consequence of \cite[Corollary 2.7]{W}.

\begin{example} \label{nonseparative}
Let $E$ be the graph below
$$\xymatrixrowsep{2pc}\xymatrixcolsep{4pc}\def\labelstyle{\displaystyle}
\xymatrix{
 &v \ar@/_2ex/[dl]_{e_1} \ar[dl]^{e_2} \ar@/^5ex/[dl]^{e_3} \ar@/^2ex/[dr]^{f_1} \ar[dr]_{f_2} \ar@/_5ex/[dr]_{f_3} \\
x &&y
}$$
and define $C := \bigl\{ \{e_1,e_2\}, \{f_1,f_2\}, \{e_3,f_3\} \bigr\}$. Then $M(E,C)$ is presented by the generators $v$, $x$, $y$ and the relations
$$v=x+x \qquad\qquad\qquad v=y+y \qquad\qquad\qquad v=x+y.$$
Observe that $x\ne y$, so that $M(E,C)$ is not separative.

Now choose a complete resolution $(E_+,C^+)$ for $(E,C)$. No special properties are required of the matrices in the sets $\delta_n$, so they may be chosen with all entries equal to $1$. (We leave to the reader the task of choosing suitable notation for the vertices and edges of $E_+$.) By Theorem \ref{unitary3}, $M(E_+,C^+)$ is a conical refinement monoid, and the inclusion morphism $(E,C) \rightarrow (E_+,C^+)$ induces a unitary embedding $M(E,C)\rightarrow M(E_+,C^+)$, which we denote $a\mapsto a'$. Hence, $2x'=x'+y'=2y'$ in $M(E_+,C^+)$, but $x'\ne y'$. Therefore $M(E_+,C^+)$ is not separative.
\end{example}

%%%%%%%%%%%%%%%%%%%%%%%%%%%%%%%%%%%%%
\section{Refining Leavitt path algebras of separated graphs}

We may apply the resolution process of the previous section to \emph{embed}
Leavitt path algebras of separated graphs into others with
refinement.

\begin{theorem}
\label{embedref} Let $(E_+,C^+)$ be a complete resolution of
an object $(E,C)$ in $\FSGr$, and let $\iota: (E,C) \rightarrow (E_+,C^+)$ be the inclusion
morphism. Then
$\mon{L_K(E_+,C^+)}\cong M(E_+,C^+)$ is a refinement monoid, the $K$-algebra homomorphism
$L_K(\iota): L_K(E,C) \rightarrow L_K(E_+,C^+)$ is an embedding, and the monoid homomorphism
$$\mon{L_K(\iota)}: \mon{L_K(E,C)}\longrightarrow \mon{L_K(E_+,C^+)}$$
is a unitary embedding.
\end{theorem}

\begin{proof}
The first and third conclusions follow from Theorems \ref{computVMECS} and \ref{unitary3}. Since $(E,C)$ is a complete $SG$-subgraph of
$(E_+,C^+)$, it follows from (the proof of) Proposition
\ref{CLalg-dirlim} that the homomorphism $L_K(\iota)$
is injective.
\end{proof}

As a specific application of our construction, we can unitarily embed
any simple conical monoid $M$ into a simple refinement monoid $M^+$ in such
a way that $M^+$ corresponds to the set of order-units in a
monoid of the form $\mon{L_K(E_+,C^+)}$ together with $\{0\}$, and such that the latter is a simple, divisible, refinement monoid. The existence of such monoid embeddings has been proved by Wehrung \cite[Corollary 2.7]{W}; our construction provides a more ``visual'' version.

\begin{definition} \label{defdivmon}
A monoid $M$ is \emph{divisible} provided every element $x\in M$ is divisible by every positive integer $n$, that is, there exist elements $y_n\in M$ such that
$ny_n=x$ for all $n\in\N$.
\end{definition}

\begin{theorem}
\label{simple-emb} Let $M$ be a simple, conical, abelian
monoid with a given presentation by generators $x_j$ $(j\in J)$ and relations $r_i$ $(i\in I)$ as in \eqref{Mpres}. Assume that for each $j\in J$, there is some $i\in I$ such that $a_{ij}+b_{ij}>0$. Let $(E,C)$ be the
finitely separated graph associated to the given presentation of
$M$ as in the proof of Proposition {\rm\ref{allconical}}.

Then there is a suitable complete resolution $(E_+,C^+)$ of $(E,C)$
such that
$\mon{L_K(E_+,C^+)}$ is a divisible refinement monoid and $M\cong \mon{L_K(E,C)}$ unitarily embeds in
the simple refinement monoid consisting of the union of $\{ 0\}$ and
the semigroup of order-units of $\mon{L_K(E_+,C^+)}$. If the given presentation of $M$ is countable {\rm(}i.e., $I$ and $J$ are countable{\rm)}, then $L_K(E_+,C^+)$ is a countably generated $K$-algebra.
\end{theorem}

\begin{remark*} The hypothesis concerning the $a_{ij}+b_{ij}$ is harmless, since we can always impose additional relations of the form $x_j=x_j$.
\end{remark*}

\begin{proof}
The assumptions of Proposition \ref{allconical} require that for each $i\in I$, at least
one $a_{ij}>0$ and at least one $b_{ij}>0$. Hence, each vertex $u_i$ emits at least two edges. Our
present hypotheses require that each $v_j$ receives at least one edge. Thus, none of the $u_i$ is a sink in $E$, and none of the $v_j$ is a source in $E$.

Build a complete resolution $(E_+,C^+)$ for $(E,C)$ as in
Construction \ref{resolution3}, where each matrix in each $\delta_n$
is chosen with all entries equal to $n!$. Further, choose the
collections $T_n$ to be \emph{symmetric}, meaning that for all $v\in
E^0_n\setminus E^0_{n-1}$ and distinct $X,Y\in C^n_v$, both $(v,X,Y)$ and $(v,Y,X)$
appear in $T_n$. Set $A:= L_K(E,C)$ and $A^+:= L_K(E_+,C^+)$. By
Theorem \ref{embedref}, $\mon{A^+}$ is a refinement monoid and the
monoid homomorphism $\mon{L_K(\iota)}: \mon{A} \rightarrow
\mon{A^+}$ is a unitary embedding, where $\iota: (E,C) \rightarrow
(E_+,C^+)$ is the inclusion morphism. If $I$ and $J$ are countable,
then $E_0=E$ is a countable graph, all the collections $T_n$ are
countable, and all the graphs $E_n$ are countable. Consequently,
$A^+$ is a countably generated $K$-algebra. Now return to the
general case.

Since $\mon{A}\cong M$ is conical and simple, each of its nonzero elements is an order-unit. By definition of a unitary embedding, the image of $\mon{L_K(\iota)}$ is cofinal in $\mon{A^+}$, from which it follows that $\mon{L_K(\iota)}$ maps all nonzero elements of $\mon{A}$ to order-units of $\mon{A^+}$. Consequently, $\mon{A}$ embeds unitarily in the submonoid $\{0\}\sqcup S$ of $\mon{A^+}$ by way of $\mon{L_K(\iota)}$, where $S$ is the semigroup of order-units in $\mon{A^+}$. Note that $\{0\}\sqcup S$ is a simple monoid.

We now show that $\mon{A^+}$ is divisible. In particular, it is then weakly divisible in the sense of \cite[Definition 2.2]{OPR}. Once this is established, a recent result of Ortega, Perera, and R\o rdam \cite[Theorem 3.4]{OPR} will show that $\{0\}\sqcup S$ is a refinement monoid, completing the proof of the theorem.

{\bf Claim 1}: $|C^{n+1}_v| \ge2$ for all $v\in E^0_n$. In particular, this will show that $E_+$ has no sinks.

For any $i\in I$, we have $C^0_{u_i}= \{ X_{i1}, X_{i2} \}$ by construction, and $C^0_{u_i} \subseteq C^1_{u_i}$,
so $|C^1_{u_i}| \ge 2$. For any $j\in J$, there is some $i\in I$ such that
$a_{ij}+b_{ij}>0$ (by hypothesis), so there exists $e\in E^1_0$ with $s(e)=u_i$ and $r(e)=v_j$.
Since $T_0$ is symmetric, it contains both $(u_i,X_{i1},X_{i2})$ and $(u_i,X_{i2},X_{i1})$, say labelled $(w_k,X_k,Y_k)$ and $(w_l,X_l,Y_l)$. After possibly interchanging $k$ and $l$, we may assume that $e\in X_k=Y_l$. Then we have $X^k_e,Y^l_e\in C^1_{v_j}$, and so $|C^1_{v_j}| \ge 2$. This establishes the claim when $n=0$.

Now let $n>0$, and assume the claim holds for $n-1$, that is,
$|C_v^n|\ge 2$ for all $v\in E^0_{n-1}$. For $v\in E^0_{n-1}$, we
have $C^n_v\subseteq C^{n+1}_v$, and so $|C^{n+1}_v| \ge |C^n_v| \ge
2$ by our induction hypothesis. Now consider $v\in E^0_n\setminus
E^0_{n-1}$. Then there exist  $(w_k,X_k,Y_k) \in T_{n-1}$ and $e\in
X_k$, $f\in Y_k$ such that $v= v^k_{e,f}$. There is an edge $g :=
g^k_{e,f,1} \in E^1_n$ with $s(g)=r(e)$ and $r(g)=v$, and $g\in
X^k_e \in C^n_{r(e)}$. Since $r(e)\in E^0_{n-1}$, the induction
hypothesis implies that $C^n_{r(e)}$ contains a set $Z\ne X^k_e$. As
in the case of $v_j$ above, it follows that there are distinct sets
$X^k_g$ and $Y^l_g$ in $C^{n+1}_v$. This establishes Claim 1.

Set $M^+ := M(E_+,C^+) \cong \mon{A^+}$, and let $m\in\N$. To see that $\mon{A^+}$ is $m$-divisible, it suffices to show that each $v\in E^0_+$ is divisible by $m$ in $M^+$. Assume first that $v\in E^0_n\setminus E^0_{n-1}$ for some $n\ge m$. By Claim 1, there exists a set $X\in C^{n+1}_v$. Note that $v$ must be a sink in $E_n$, so $X\notin C^n_v$. Consequently, $X$ must have the form $X^k_e$ or $Y^k_f$ for appropriate $k$, $e$, $f$. Since we have chosen $\delta^{n,k}_{e,f}=n!$ for all $k$, $e$, $f$, the number $n_w$ of edges in $X$ from $v$ to any vertex $w\in E^0_{n+1}$ is divisible by $m$. In $M^+$, we have $v= \bfr(X)= \sum_{w\in E^0_{n+1}} n_ww$, and thus $v$ is divisible by $m$.

Next, assume that $v\in E^0_n\setminus E^0_{n-1}$ where $0<n<m$, and that all vertices from $E^0_{n+1}\setminus E^0_n$ are divisible by $m$ in $M^+$. By Claim 1, there exists a set $X\in C^{n+1}_v$; as above, $X\notin C^n_v$. Then $v= \sum_{e\in X} r(e)$ in $M^+$, and $r(e)$ is divisible by $m$ for each $e\in X$ (because $r(e)\in E^0_{n+1}\setminus E^0_n$), whence $v$ is divisible by $m$. The same argument applies to the vertices $v$ in $\{v_j \mid j\in J\}$. Finally, for $i\in I$ we have $u_i= \bfr(X_{i1})= \sum_{j\in J} a_{ij}v_j$, and thus $u_i$ is divisible by $m$ in $M^+$. Therefore all the generators of $M^+$ are divisible by $m$, and the proof is complete.
\end{proof}

\begin{example}
\label{mainrevisited} Let $M$ be the monoid $\langle x\mid mx=nx\rangle$ presented by one generator $x$ and one relation
$mx=nx$, where $1<m< n$. The separated graph $(E,C)$ associated to
this presentation is just the separated graph $(E(m,n), C(m,n))$ of
Example \ref{Lmn}, so that $L_K(E,C)= A_{m,n}$ in the notation of the example. By Proposition \ref{keyexample}, we can identify $\mon{A_{m,n}}= \langle x\mid mx=nx\rangle$ so that $[w]=x$. It is immediate from the properties of this monoid that the idempotent matrices $w,2{\cdot} w,\dots, (m-1){\cdot} w$ are finite (i.e., not equivalent to proper subidempotents of themselves), whereas $m{\cdot} w\sim n{\cdot} w\sim v$ is properly infinite (i.e., $m{\cdot} w\oplus m{\cdot} w \lesssim m{\cdot} w$).

The hypotheses of Theorem \ref{simple-emb} are clearly satisfied, so
that $A_{m,n}$ is embedded in the count\-ably generated algebra
$L_K(E_+,C^+)$ given by the theorem, and $\mon{A_{m,n}}$ is
unitarily embedded in $\mon{L_K(E_+,C^+)}$. Due to the unitarity of
this embedding, the idempotent matrices $w,2{\cdot} w,\dots, (m-1){\cdot} w$
remain finite (and full) over $L_K(E_+,C^+)$, while obviously $m{\cdot} w$
remains properly infinite. This information is faithfully recorded
in $\mon{L_K(E_+,C^+)}$, which by the theorem is a divisible
refinement monoid.
\end{example}

%%%%%%%%%%%%%%%%%%%%%%%%%%%%%%%%%%%%%
\section{Appendix. Projective modules and trace ideals for nonunital rings}
\label{appendix}

In this appendix, we gather some information and results concerning projective modules, trace ideals, and $\calV$-monoids for nonunital rings. Some of this material is well known, but some is not readily accessible, and some has not been developed in the literature to our knowledge.

\begin{definition} \label{idemprings}
A ring $R$ is \emph{idempotent} provided $R^2=R$, and it is \emph{s-unital} if for each $x\in R$, there exist $u,v\in R$ such that $ux=x=xv$. The latter property carries over to finite sets by \cite[Lemma 2.2]{arasunital}: If $R$ is s-unital and $x_1,\dots,x_n\in R$, there exists $u\in R$ such that $ux_i=x_i=x_iu$ for all $i$.

As is common (see \cite{GS}, for instance), we define $R\Mod$ to be the category of those left $R$-modules which are \emph{full} (meaning that $RM=M$) and \emph{nondegenerate} (meaning that $Rx=0$ implies $x=0$, for any $x\in M$). The morphisms in $R\Mod$ are arbitrary module homomorphisms between the above modules. We refer to $R$ itself as \emph{nondegenerate} if it is nondegenerate as both a left and a right $R$-module.

Note that if $R$ is s-unital, then so is any full left $R$-module $M$: for any $y_1,\dots,y_m\in M$, there is some $u\in R$ such that $uy_j=y_j$ for all $j$. In particular, it follows that $M$ is nondegenerate. In fact, all $R$-submodules of $M$ are full and nondegenerate.

Let $R^+$ denote the canonical unitization of $R$ (irrespective of whether $R$ may have a unit), namely the unital ring containing $R$ as a two-sided ideal such that $R^+= \Z\oplus R$. The forgetful functor provides a category isomorphism from $R^+\Mod$ to the category of arbitrary left $R$-modules \cite[Proposition 8.29B]{Fai}. We identify these two categories, and then $R\Mod$ becomes identified with the full subcategory of $R^+\Mod$ whose objects are those $R^+$-modules which are full and nondegenerate as $R$-modules.
\end{definition}

\begin{lemma} \label{projobj}
Let $R$ be a nondegenerate idempotent ring. The projective objects in $R\Mod$ are precisely those which are projective as $R^+$-modules, that is, the projective $R^+$-modules $P$ such that $RP=P$.
\end{lemma}

\begin{proof}
As noted in \cite[\S5.3]{Gdirlim}, all epimorphisms in $R\Mod$ are surjective. Consequently, any object of $R\Mod$ which is projective in $R^+\Mod$ must also be projective in $R\Mod$. Conversely, suppose $P$ is a projective object of $R\Mod$. Choose a free $R^+$-module $F$ and an $R^+$-module epimorphism $f:F\rightarrow P$. Then $RF$ is a full nondegenerate $R$-module, and $f$ maps $RF$ onto $P$, so there exists an $R$-module homomorphism $g:P\rightarrow RF$ such that $fg=\id_P$. Since $g$ is also an $R^+$-module homomorphism, $P$ is a projective $R^+$-module.

Since $R$ is nondegenerate, all projective $R^+$-modules are nondegenerate as $R$-modules. Hence, any projective $R^+$-module $P$ satisfying $RP=P$ is an object of $R\Mod$, projective by the previous paragraph. The previous discussion also shows that all projective objects of $R\Mod$ have this form.
\end{proof}

\begin{definition} \label{hereditary}
Let us define a ring $R$ to be \emph{left hereditary} if every subobject of a projective object in $R\Mod$ is projective. We define ``right hereditary'' symmetrically, and say $R$ is \emph{hereditary} provided it is both left and right hereditary.
\end{definition}

\begin{corollary} \label{heredimplic}
Let $R$ be a nondegenerate idempotent ring.
\begin{enumerate}
\item If $R^+$ is {\rm(}left{\rm)} hereditary, then $R$ is {\rm(}left{\rm)} hereditary.
\item If $R$ is s-unital and {\rm(}left{\rm)} hereditary, then $eRe$ is {\rm(}left{\rm)} hereditary for all idempotents $e\in R$.
\end{enumerate}
\end{corollary}

\begin{proof} (1) This follows from Lemma \ref{projobj}.

(2) Assume that $R$ is s-unital and left hereditary, and let $e\in R$ be an idempotent. We note that $Re=R^+e$ and $eRe= eR^+e$. By Lemma \ref{projobj}, $R^+e$ is projective in $R\Mod$. Since $R$ is s-unital, any $R^+$-submodule $N$ of $R^+e$ is an object of $R\Mod$. Then $N$ is projective in $R\Mod$ by our hypothesis on $R$, and hence projective as an $R^+$-module by Lemma \ref{projobj}. Thus, all $R^+$-submodules of $R^+e$ are projective, which implies that $eR^+e$ is left hereditary \cite[Theorem 2.5]{Hill} (or see \cite[\S39.16]{Wis}).
\end{proof}

\begin{definition} \label{defAtilde}
If $A$ is a $K$-algebra, we shall write $A^\sim$ for the canonical $K$-algebra unitization of $A$, that is, the unital $K$-algebra containing $A$ as a two-sided ideal such that $A^\sim= K\oplus A$.

In order to apply the previous results to this setting, we need to be able to identify $A\Mod$ with a full subcategory of $A^\sim\Mod$, which requires that all full nondegenerate $A$-modules are vector spaces over $K$ in a canonical fashion. It is not clear whether this holds in general, but it does when $A$ is s-unital, as follows.
\end{definition}

\begin{lemma} \label{Kstructure}
Let $A$ be an s-unital $K$-algebra and $M$ a full left $A$-module. Then the multiplication map $\mu: A\otimes_{A^+} M\rightarrow M$ is an isomorphism of $A$-modules. Consequently, there is a unique $K$-vector space structure on $M$ such that $\alpha(ax)= (\alpha a)x$ for all $\alpha\in K$, $a\in A$, $x\in M$, and thus $M$ has a canonical left $A^\sim$-module structure.
\end{lemma}

\begin{proof} Note: We must use $A^+$ in the definition of $\mu$, since we do not yet know that $M$ is an $A^\sim$-module.

We claim there is a map $\lambda: M\rightarrow A\otimes_{A^+} M$ defined as follows: Given $x\in M$, choose $u\in A$ such that $ux=x$, and set $\lambda(x)= u\otimes x$. To see that $\lambda$ is well defined, suppose that also $u'x=x$ for some $u'\in A$. There exists $v\in A$ such that $vu=u$ and $vu'=u'$, whence
$$u\otimes x= vu\otimes x= v\otimes x= vu'\otimes x= u'\otimes x.$$
A similar computation shows that $\lambda$ is surjective, and it is clear that $\mu\lambda= \id_M$. Hence, $\mu$ is a bijection, and thus an $A$-module isomorphism.

Since $A$ is a $(K,A^+)$-bimodule, $A\otimes_{A^+} M$ has a natural $K$-vector space structure, and this transfers to $M$ via $\mu$. Uniqueness is clear, and the $A^\sim$-module structure follows.
\end{proof}

\begin{proposition} \label{alghered}
Let $A$ be an s-unital $K$-algebra. If $A^\sim$ is {\rm(}left{\rm)} hereditary, then $A$ is {\rm(}left{\rm)} hereditary.
\end{proposition}

\begin{proof} There is a natural ring homomorphism $A^+\rightarrow A^\sim$, with respect to which the pullback functor embeds $A^\sim\Mod$ in $A^+\Mod$. Under this embedding, the $A^\sim$-modules which are full $A$-modules correspond precisely to the objects of $A\Mod \subseteq A^+\Mod$, in view of Lemma \ref{Kstructure}.

Just as in Lemma \ref{projobj}, the projective objects of $A\Mod$ are precisely those which are projective as $A^\sim$-modules. The conclusion of the proposition follows.
\end{proof}

\begin{definition} \label{FP}
Consider again a ring $R$ with unitization $R^+$. (Here $R^+$ could be replaced by any unital ring containing $R$ as a two-sided ideal.) Let $\FP(R,R^+)$ denote the full subcategory of $R^+\Mod$ whose objects are those finitely generated projective left $R^+$-modules $P$ such that $RP=P$. There is a monoid isomorphism from $\mon{R}$ onto the Grothendieck monoid of $\FP(R,R^+)$, that is, the monoid of isomorphism classes of objects from $\FP(R,R^+)$, with addition induced from direct sum. The isomorphism sends the class $[e]$ of an idempotent matrix $e\in M_n(R)$ to the isomorphism class of the module $(R^+)^ne$. (For details, see e.g. \cite[\S5.1]{Gdirlim}.)

If the ring $R$ is idempotent, the objects of $\FP(R,R^+)$ are exactly the compact projective objects of $R\Mod$ \cite[Lemma 5.5]{Gdirlim}.
\end{definition}

In general, the \emph{trace} of a left module $M$ over a ring $R$ is the sum of the images of all homomorphisms from $M$ to $R$; it is a two-sided ideal of $R$. If $R$ is unital and $M$ is finitely generated projective, say $M\cong R^ne$ for some idempotent matrix $e\in M_n(R)$, then the trace of $M$ is generated by the entries of $e$. This leads one to define the \emph{trace ideals} of $R$ as those ideals which can be generated by the entries of idempotent matrices. In the unital case, there is an isomorphism between the lattice of trace ideals of $R$ and the lattice of order-ideals in $\mon{R}$, as follows from \cite[Theorem 2.1(c)]{FHK}. Since the proof is indirect, and does not extend immediately to the non-unital case, we develop the general result here.

\begin{definition} \label{trace}
Let $R$ be an arbitrary ring. Recall the ring $M_\infty(R)$ from \S\ref{introV(R)}, and let $\Idem(M_\infty(R))$ denote the set of idempotents in $M_\infty(R)$. An ideal $I$ of $R$ is called a \emph{trace ideal} provided $I$ can be generated by the entries of the matrices in some subset of $\Idem(M_\infty(R))$. We denote by $\Tr(R)$ the set of all trace ideals of $R$. Since $\Tr(R)$ is closed under arbitrary sums, it forms a complete lattice with respect to inclusion.

Recall that an \emph{order-ideal} of a monoid $M$ is a submonoid $I$ of $M$ such
that $x+y\in I$ (for some $x,y\in M$) implies that both $x$ and $y$ belong to
$I$. An order-ideal can also be described as a submonoid $I$ of $M$
which is \emph{hereditary} with respect to the canonical pre-order $\le $
on $M$, meaning that $x\le y$ and $y\in I$ imply $x\in I$. Recall that the
pre-order $\le $ on $M$ is defined by setting $x\le y$ if and only
if there exists $z\in M$ such that $y=x+z$.

The family $\calL(M)$ of all order-ideals of $M$ is closed under arbitrary intersections, and hence it forms a complete lattice with respect to inclusion. The supremum of a family $\{I_i\}$ of order-ideals of $M$ is the set $\overline{\sum
}\, I_i$ consisting of those elements $x\in M$ such that $x\le y$ for some $y$
belonging to the algebraic sum $\sum I_i$. Note that $\overline{\sum}\, I_i=\sum I_i$ whenever $M$ is a refinement
monoid.
\end{definition}

\begin{proposition}  \label{LVAisoTrA}
For any ring $R$ there are mutually inverse lattice isomorphisms
$$\Phi: \calL(\mon{R}) \longrightarrow \Tr(R) \qquad\qquad\text{and} \qquad\qquad \Psi: \Tr(R)\longrightarrow \calL(\mon{R})$$
given by the rules
\begin{align*}
\Phi(I) &= \langle\; \text{entries of\ } e \mid e\in \Idem(M_\infty(R)) \text{\ and\ } [e]\in I \;\rangle \\
\Psi(J) &= \{ [e]\in \mon{R} \mid e\in \Idem(M_\infty(J)) \}.
\end{align*}
\end{proposition}

\begin{proof} Clearly, $\Phi$ is a well-defined, inclusion-preserving map from $\calL(\mon{R})$ to $\Tr(R)$. Now let $J\in \Tr(R)$; we need to show that $\Psi(J)$ is an order-ideal in $\mon{R}$. First note that if $e,e'\in \Idem(M_\infty(R))$ and $e\sim e'$, then $e\in M_\infty(J)$ if and only if $e'\in M_\infty(J)$. This holds simply because $M_\infty(J)$ is an ideal of $M_\infty(R)$. Hence, we can check whether an element $x$ of $\mon{R}$ lies in $\Psi(J)$ by considering any representative of $x$ in $M_\infty(R)$. Now if $[e],[f]\in \Psi(J)$, then $e,f\in M_\infty(J)$ and so $e\oplus f\in M_\infty(J)$, whence $[e]+[f] \in \Psi(J)$. Similarly, if $e\oplus f\in M_\infty(J)$, then both $e$ and $f$ belong to $M_\infty(J)$, so that $[e]+[f] \in \Psi(J)$ implies $[e],[f]\in \Psi(J)$. This shows that $\Psi(J)\in \calL(\mon{R})$. Therefore $\Psi$ is a well-defined map from $\Tr(R)$ to $\calL(\mon{R})$. It clearly preserves inclusions.

It remains to show that $\Phi$ and $\Psi$ are inverses of each other, for then they will be isomorphisms of posets, and therefore lattice isomorphisms.

First, we observe that $\Phi\circ \Psi= \id_{\Tr(R)}$: If $J$ is a trace ideal of $R$, then $J$ is generated by the entries of the matrices in $\Idem(M_\infty(J))$, and so $\Phi\Psi(J)= J$.

Finally, we show that $\Psi\circ \Phi= \id_{\calL(\mon{R})}$. Let $I\in \calL(\mon{R})$, and set $J := \Phi(I)$. Clearly, $I\subseteq \Psi(J)$. If $[f]\in \Psi(J)$, there are idempotents $e_1,\dots,e_m$ in $M_\infty(R)$, with $[e_l]\in I$ for each $l$, such that every entry of $f$ has the form $f_{ij}= \sum_{l=1}^m a^l_{ij} u^l_{ij} b^l_{ij}$ where the $a^l_{ij},b^l_{ij} \in R$ and the $u^l_{ij}$ are entries of $e_l$. Now $e:= e_1\oplus \cdots\oplus e_m$ is an idempotent in $M_\infty(R)$ with $[e]\in I$, and each $u^l_{ij}$ is an entry of $e$. We may treat $e$ as a $t\times t$ matrix, for some $t$. Each term $a^l_{ij} u^l_{ij} b^l_{ij}$ can be expressed in the form $x^l_{ij} e y^l_{ij}$ where $x^l_{ij}$ is a $1\times t$ matrix over $R$ and $y^l_{ij}$ is a $t\times1$ matrix over $R$. Let $e'$ denote the block diagonal $mt\times mt$ matrix with $m$ copies of $e$ on the diagonal. Then each
$$f_{ij}= \sum_{l=1}^m x^l_{ij} e y^l_{ij}= x'_{ij}e'y'_{ij}$$
for some $1\times mt$ matrix $x'_{ij}$ and some $mt\times1$ matrix $y'_{ij}$. It follows that we can express $f$ as a product of block matrices of the form
$$f=x''e''y''= \begin{bmatrix} x'_{11} &x'_{12} &\cdots \\ &&&x'_{21} &x'_{22} &\cdots \\ &&&&&&\ddots \\ &&&&&&&\ddots \end{bmatrix}  \begin{bmatrix} e' \\ &e' \\ &&\ddots \\ &&&e' \end{bmatrix}  \begin{bmatrix} y'_{11} \\ &y'_{12} \\ &&\ddots \\ y'_{21} \\ &y'_{22} \\ &&\ddots \\ &\vdots \end{bmatrix}  ,$$
where $e''$ is idempotent and $[e'']\in I$. Now $g := e''y''fx''e''$ is an idempotent in $M_\infty(R)$ such that $f\sim g$ and $g\le e''$. Consequently, $[f]\le [e'']$, and so $[f]\in I$ because $I$ is an order-ideal of $\mon{R}$. Therefore $I= \Psi(J)= \Psi\Phi(I)$, as required. \end{proof}

%%%%%%%%%%%%%%%%%%%%%%%%%%%%%%%%%%%%%
\section*{Acknowledgements} We thank G. Bergman, T. Katsura, E. Pardo, and the referee for useful comments and correspondence.

\end{document}